\documentclass[12pt]{amsart}
\usepackage{amssymb}
\usepackage{amscd}
\usepackage{amsfonts}
\usepackage{latexsym}
\usepackage[all]{xy}
\usepackage{amsmath}
\usepackage{verbatim}
\usepackage{longtable}
\title{Dimensions of irreducible  modules over W-algebras and Goldie ranks}
\usepackage{amssymb}
\author{Ivan Losev}
\newcommand\g{{\mathfrak g}}
\newcommand\bfr{\mathfrak{b}}
\newcommand\h{{\mathfrak h}}
\newcommand{\q}{\mathfrak{q}}

\newcommand{\Sk}{\mathcal{S}}

\newcommand\m{\mathfrak m}
\newcommand\Coh{\operatorname{Coh}}
\newcommand\lf{\mathfrak l}
\newcommand\n{\mathfrak n}
\newcommand\z{\mathfrak z}
\newcommand\vf{\mathfrak v}

\renewcommand\t{\mathfrak t}
\newcommand\p{\mathfrak p}

\newcommand\cell{c}

\newcommand\Spec{\operatorname{Spec}}

\newcommand\F{\operatorname{F}}
\newcommand\W{{\mathbb{A}}}
\newcommand\K{\mathbb K}
\newcommand\U{\mathcal U}

\newcommand\Ann{\operatorname{Ann}}

\newcommand\Walg{\mathcal W}
\newcommand\ZZ{\mathbb Z}
\newcommand\A{\mathcal A}

\newcommand\M{\mathcal M}

\newcommand\gr{\operatorname{gr}}
\newcommand\OCat{\mathcal O}
\newcommand\I{\mathcal I}
\newcommand\J{\mathcal J}
\renewcommand\sl{\mathfrak{sl}}

\newcommand\Hom{\operatorname{Hom}}
\newcommand{\ad}{\mathop{\rm ad}\nolimits}
\newcommand{\Ad}{\mathop{\rm Ad}\nolimits}
\newcommand{\HC}{\operatorname{HC}}
\newcommand\Centr{\mathcal Z}
\newcommand\Goldie{\operatorname{Grk}}
\newcommand\mult{\operatorname{mult}}
\newcommand{\VA}{\operatorname{V}}

\newcommand\dcell{\mathbf{c}}

\newcommand\Nil{\mathcal{N}}

\newcommand\Orb{\mathbb{O}}

\newcommand\Irr{\operatorname{Irr}}
\newcommand\Wh{\operatorname{Wh}}
\newcommand\HL{H}
\newcommand\BQ{\mathbb{Q}}
\newcommand\Fi{\operatorname{F}}
\newcommand\Spr{\operatorname{Spr}}
\newcommand\bA{\bar{A}}

\newcommand{\Fun}{\operatorname{Fun}}

\newcommand{\Prim}{\operatorname{Pr}}

\newcommand{\GL}{\operatorname{GL}}

\newcommand{\pr}{\operatorname{pr}}
\newcommand{\so}{\mathfrak{so}}
\renewcommand{\sp}{\mathfrak{sp}}
\newcommand{\gl}{\mathfrak{gl}}
\newcommand{\Ort}{\operatorname{O}}

\newcommand{\gol}{\operatorname{g}}

\newcommand{\Dix}{\mathcal{A}}
\newcommand{\Kfun}{\mathcal{K}}
\renewcommand{\b}{\mathfrak{b}}
\newcommand{\wU}{\,^\wedge\U}
\newcommand{\wW}{\,^\wedge\W}
\newcommand{\wWalg}{\,^\wedge\Walg}

\newcommand{\Vfun}{\mathbb{V}}
\newcommand{\dual}{\mathcal{D}}
\newcommand{\ch}{\operatorname{ch}}
\renewcommand{\v}{\mathfrak{v}}
\renewcommand{\k}{\mathfrak{k}}
\newtheorem{Thm}{Theorem}[section]
\newtheorem{Prop}[Thm]{Proposition}
\newtheorem{Cor}[Thm]{Corollary}
\newtheorem{Lem}[Thm]{Lemma}
\theoremstyle{definition}

\newtheorem{defi}[Thm]{Definition}
\newtheorem{Rem}[Thm]{Remark}
\newtheorem{Conj}[Thm]{Conjecture}
\unitlength=1mm   \numberwithin{equation}{section}
\numberwithin{table}{section} \oddsidemargin=0cm
\thanks{Supported by the NSF grants DMS-0900907, DMS-1161584}
\address{Department
of Mathematics, Northeastern University, Boston MA 02115 USA}
\email{i.loseu@neu.edu}
\begin{document}
\begin{abstract}
The main goal of this paper is to compute two related numerical invariants of a primitive ideal
in the universal enveloping algebra of a semisimple Lie algebra. The first one, very classical,
is the Goldie rank of an ideal. The second one is the dimension of an irreducible module corresponding to
this ideal over an appropriate finite W-algebra. We concentrate on the integral central character
case. We prove, modulo a conjecture, that in this case the two are equal. Our conjecture asserts that there is a one-dimensional module over the W-algebra with certain additional properties. The conjecture is proved for the classical types. Also, modulo the same
conjecture, we compute certain scale factors introduced by Joseph, this allows to compute the
Goldie ranks of the algebras of locally finite endomorphisms of  simples in the BGG category $\mathcal{O}$.
This completes a program of computing Goldie ranks proposed by Joseph in the 80's (for integral central characters and modulo our conjecture).

We also provide an essentially Kazhdan-Lusztig type formula for computing the characters of the irreducibles
in the Brundan-Goodwin-Kleshchev category $\mathcal{O}$ for a W-algebra again under the assumption that the central
character is integral. In particular, this allows to compute the dimensions of the
finite dimensional irreducible modules.
The formula is based on a certain functor from an appropriate parabolic
category $\mathcal{O}$ to the W-algebra category $\mathcal{O}$. This functor can be regarded as a generalization of functors
previously constructed by Soergel and by Brundan-Kleshchev. We prove a number of properties of this
functor including the quotient property and the double centralizer property.

We develop several side topics related to our generalized Soergel functor. For example, we discuss its analog
for the category of Harish-Chandra modules. We also discuss generalizations to the case of categories
$\mathcal{O}$ over Dixmier algebras. The most interesting example of this situation comes from the theory of quantum
groups: we prove that an algebra that is a mild quotient of Luszitg's form of a quantum group at a root of unity
is a Dixmier algebra. For this we check that the quantum Frobenius epimorphism splits.
\end{abstract}
\maketitle

\begin{center}{\it Dedicated to Tony Joseph, on his 70th birthday.}\end{center}

\tableofcontents
\section{Introduction}
\subsection{Primitive ideals and Goldie ranks}
The study of primitive ideals in universal enveloping algebras
is a classical topic in Lie representation theory (recall that by a primitive ideal
in an associative algebra one means the annihilator of a simple module; our base field is always an algebraically closed
field $\K$ of characteristic 0). The case when the Lie algebra is
solvable is understood well: the classification of primitive ideals is known and relatively easy, see, for example,
\cite[Section 6]{Dixmier}. A nice feature of that case is that all primitive ideals are completely
prime, i.e., the quotients have no zero divisors, see \cite[Section 5]{BGR}.

It seems that the most interesting  case in the study of primitive ideals  is when the Lie algebra
$\g$ under consideration is semisimple (partial results on a reduction of the general case to the semisimple one
can be found in \cite[Section 4]{Joseph_ICM}). There was a lot of work on primitive ideals in the
semisimple case and related topics in 70's and  in 80's
(by Barbasch, Duflo, Joseph, Lusztig, Vogan, to mention a few authors) that resulted,
in particular, in a classification of the primitive ideals, for a survey, see \cite{Joseph_ICM}.
A basic result is a theorem of Duflo
that says that any primitive ideal in the universal enveloping algebra $U(\g)$
has the form $J(\lambda):=\Ann_{U(\g)}L(\lambda)$, where $L(\lambda)$ stands for
the irreducible highest weight module with highest weight $\lambda-\rho$, as usual,
$\rho$ denotes half the sum of all positive roots.

One of the classical problems about primitive ideals is to compute their Goldie
ranks. Let us recall the definition of the Goldie rank. Take a prime noetherian
algebra $\A$. According to the Goldie theorem, there is a full fraction
algebra $\operatorname{Frac}(\A)$ of $\A$. The algebra $\operatorname{Frac}(\A)$
is simple and so is the matrix algebra of rank $r$ over some skew-field. The number
$r$ is called the Goldie rank of $\A$ and is denoted by $\Goldie(\A)$ or $\Goldie \A$. Abusing the notation, by the Goldie rank
of a primitive ideal $\J$ of  the universal enveloping algebra $\U:=U(\g)$ one
means the Goldie rank of the quotient $\U/\J$.

There are many results and constructions  related to the computation of Goldie ranks due to Joseph,
see, in particular, \cite{Joseph_I}-\cite{Joseph_scale}. Some of them will be recalled in
Subsection  \ref{SUBSECTION_Goldie_rem}. Here we only going to recall one important
construction -- Joseph's scale factors. Throughout the paper we basically only consider the
case of integral central characters, i.e., we restrict ourselves to ideals $J(\lambda)$,
where $\lambda$ is in the weight lattice $\Lambda$ of $\g$.

Any integral weight $\lambda$ can be represented in the form $w\varrho$, where $w$ is an element
in the Weyl group, and $\varrho\in \Lambda$ is  dominant. This representation is unique
if we require that $w\alpha<0$ for all roots $\alpha$ with $\langle\varrho,\alpha\rangle=0$
(in this case we say  that $w$ and $\varrho$ are compatible).

Consider the algebra $L(L(w\varrho),L(w\varrho))$ of all $\g$-finite linear endomorphisms of $L(w\varrho)$.
According to Joseph, \cite[2.5]{Joseph_I},  this algebra is prime and noetherian. So one can define its Goldie rank.
It turns out that the ratio $$\frac{\Goldie L(L(w\varrho),L(w\varrho))}{\Goldie(U(\g)/J(w\varrho))}$$
does not depend on $\varrho$ as long as $\varrho$ is a dominant integral weight compatible with $w$, \cite[5.12]{Joseph_I}.
This ratio, Joseph's {\it scale factor}, is denoted by $z_w$.
The number $z_w$ is an integer, \cite[5.12]{Joseph_II}. It has many remarkable properties studied, for example,
in \cite{Joseph_Duflo,Joseph_scale}. One of these properties -- a connection to Lusztig's asymptotic Hecke algebra, \cite[5.8]{Joseph_Duflo},
was an important motivation for the present project.

In \cite[5.5]{Joseph_II} Joseph proposed a program of computing Goldie ranks.
Knowing the scale factors is the first crucial step in this program.
The second one is to prove that there are ``sufficiently many'' completely prime (=of Goldie
rank 1) primitive ideals. For our purposes,   this means that,
for each special nilpotent orbit $\Orb$, there is a completely prime primitive ideal $\J$ with integral
central character and $\VA(\U/\J)=\overline{\Orb}$ (the terminology and notation will be explained in
Subsections \ref{SS_W_irreps},\ref{SS_not}).

There is one relatively easy case in the problem  of computing Goldie ranks: when $\g$ is of type $A$, see \cite{Joseph_Kostant}. Namely, all scale factors
equal $1$, and there are sufficiently many completely prime ideals, so the problem of
computing the Goldie ranks is settled. Also in \cite[8.1,10]{Joseph_Kostant}, Joseph gets
formulas for Goldie ranks. For recent developments  here
see \cite{Brundan} (Theorem 1.6 in {\it loc.cit.} presents a combinatorial formula for the Goldie rank).
In all other types the problem of computing Goldie ranks has been open.

In this paper we are going to state a conjecture, the affirmative answer to which
will complete Joseph's program, i.e., will yield both a formula for scale factors and the existence
of sufficiently many completely prime primitive ideals. Then we will prove the conjecture for the B,C,D types.

Also the affirmative answer will provide a formula for the Goldie ranks.
The formula, see Subsection \ref{SS_W_results}, will compute the dimensions of finite dimensional irreducible modules
over certain associative algebras known as W-algebras. To each primitive ideal, one can assign a collection of irreducible
finite dimensional modules over the corresponding W-algebra, all of the modules have the same dimension.
More details on W-algebras will be given in the next subsection.
In the integral central character case, we will prove that the dimension coincides
with the corresponding Goldie rank.

Let us summarize. We obtain the following results (in the integral character case):
\begin{enumerate}
\item A formula for the scale factors $z_w$ (for classical types).
\item The existence of sufficiently many completely prime primitive ideals (for classical types).
\item A formula for the dimension of an irreducible W-algebra module corresponding to a primitive ideal (for all types).
\item The coincidence of the dimension and the Goldie rank (for classical types).
\end{enumerate}
Thanks to (4), (3) gives a formula for the Goldie rank. (1) and (2) are irrelevant for  (3), but (2) is used to prove (4)
and (1) is proved simultaneously with (4).

\subsection{W-algebras and their finite dimensional irreducible representations}\label{SS_W_irreps}
Now fix a nilpotent orbit $\Orb$. From the data of $\g$ and $\Orb$ one can construct
an associative algebra $\Walg$ called a {\it finite W-algebra} (below we omit the adjective
``finite''). In the full generality this was first done in \cite{Premet}.
The properties of $\Walg$ we need will be recalled in Sections \ref{S_W_alg},\ref{S_Cats}.
For now, we will only need to know two things regarding W-algebras.

One is a result obtained in \cite{HC} and describing a relationship between the sets $\Prim_{\Orb}(\U)$
of all primitive ideals $\J\subset\U$ with associated variety $\VA(\U/\J)$ equal to $\overline{\Orb}$
and the set $\Irr_{fin}(\Walg)$ of all finite dimensional irreducible $\Walg$-modules. The component group $A(=A(e))$
of the centralizer of $e\in \Orb$ in the adjoint group $\operatorname{Ad}(\g)$
acts on $\Irr_{fin}(\Walg)$. This action is induced from a certain reductive group action
on $\Walg$ by algebra automorphisms. It turns out that there is a natural identification
of the orbit space $\Irr_{fin}(\Walg)/A$ with $\Prim_{\Orb}(\U)$, see \cite[1.1]{HC}. Moreover,
in \cite[Theorem 1.1]{W_classif},  the author and V. Ostrik have computed the stabilizers in $A$ of
irreducible $\Walg$-modules with integral central character. We will need some details regarding this computation
to state the main results of the present paper, so let us recall these details now.

Recall, see, for example, \cite[Chapter 5]{orange}, that the Weyl group $W$ of $\g$ splits into the union of subsets called two-sided
cells. There is a bijection between the two-sided cells in $W$ and certain nilpotent orbits in
$\g$ called {\it special}. This bijection sends $\dcell$ to the dense orbit in the associated
variety $\VA(\U/J(w\varrho))$, where $w\in \dcell$ and $\varrho$ is regular dominant, see \cite[14.15]{Jantzen} or \cite[Theorem 10.3.3]{CM}.
Until a further notice we will assume that $\Orb$ is special. To each two-sided cell $\dcell$ Lusztig in \cite[Chapter 13]{orange}
assigned a finite group $\bA$ that is naturally represented as a quotient of the component group $A$ of $\Orb$.

Further, each two-sided cell splits into the union of subsets called left cells
(and also into the union of subsets called right cells, those are obtained
from the left ones by inversion, see \cite[Chapter 5]{orange}). For dominant regular
$\varrho\in \Lambda$, the map $$w\mapsto  J(w\varrho):\dcell\rightarrow \{\text{primitive } \J\text{ with central character }\varrho\text{ and }\VA(\U/\J)=\overline{\Orb}\}$$
is surjective, its fibers are precisely the left cells, see e.g. \cite[14.15]{Jantzen}.
If $\varrho$ is dominant but not regular,  one needs to restrict to the left cells that are
compatible with $\varrho$ in the sense that any -- or equivalently some (see, for example, \cite[6.2]{W_classif}) --
element is compatible with
$\varrho$. In \cite{Lusztig_subgroups}, to each left cell $c\subset \dcell$ Lusztig
assigned a subgroup $H_c\subset \bA$ defined up to conjugacy. The main result
of \cite{W_classif}, Theorem 1.1,  is that the $A$-orbit in  $\Irr_{fin}(\Walg)$ lying over the primitive ideal
$J(w\varrho)$, where $w\in c$ is compatible with dominant $\varrho$, is $\bA/H_c$.

We will need one more construction related to that result. This construction is
a parametrization of elements in $\dcell$ conjectured by Lusztig and established
in \cite{BFO1}. In our language it can be stated as follows. Fix a regular dominant
weight $\varrho$. Let $Y$ be the set of all finite dimensional irreducible
$\Walg$-modules with central character $\varrho$. By the results of
\cite{W_classif}, recalled above, the group $\bA$ acts on this set.
In particular, we can consider the $\bA$-equivariant sheaves of finite dimensional
vector spaces on the square $Y\times Y$. Irreducible sheaves are parameterized by triples
$(x,y,\mathcal{V})$, where $x,y\in Y$ and $\mathcal{V}$ is an irreducible module over the
stabilizer $\bA_{(x,y)}$ of the pair $(x,y)$. These triples are defined up to $\bA$-conjugacy.
Here $x$ is lies in the $\bA$-orbit corresponding to $J(w\varrho)$, while $y$ is in  the orbit
corresponding to $J(w^{-1}\varrho)$.

The identification of $\dcell$ with the set of triples $(x,y,\mathcal{V})$ (considered up to $\bA$-conjugacy) in the previous paragraph
comes from a certain  functor established in \cite{HC} that maps the category
of Harish-Chandra $\U$-bimodules to a suitable category of equivariant bimodules over $\Walg$ (see Subsection
\ref{SS_dagger_functor} for a construction).
Namely, we have a subquotient in the former category corresponding to $\Orb$, and $\dcell$ parameterizes the simples
in the subquotient. The functor is well-defined on the subquotient and defines its embedding into the category
of finite dimensional equivariant $\Walg$-bimodules. The triples naturally parameterize the simples in
the image. This gives a required bijection, as was established in \cite[Section 7]{W_classif}, see, in particular, Remark 7.7 there. We will elaborate on the bijection in
Subsection \ref{SS_HC_semis}.
In type $A$, this bijection reduces to the RSK correspondence. Indeed, in type $A$ (for a regular integral character)
the primitive ideals are parameterized by standard Young tableax, and simple HC bimodules are parameterized by elements
of $S_n$. In this language, the RSK correspondence sends a simple HC bimodule to its left and right  annihilator,
see, for example, \cite[5.22-25,14.15]{Jantzen}.

Another fact about W-algebras we need is that they have categories $\mathcal{O}$ introduced by Brundan, Goodwin and Kleshchev
in \cite{BGK} that are analogous to the BGG categories $\mathcal{O}$ for $\U$. To define such a category
one needs to fix an element  $e\in\Orb$ and, most importantly, an integral
semisimple element $\theta$ centralizing $e$. A resulting category will be denoted
by $\OCat^\theta(\g,e)$. The category $\OCat^\theta(\g,e)$
contains all finite dimensional simple modules with integral central characters
(we restrict  our attention to the integral blocks) and also has
analogs of Verma modules labeled by irreducible modules over a smaller W-algebra (one for $\z_\g(\theta)$
and $e$). One can define the notion of a character
(a graded dimension) for a module $\Nil$ in that category $\mathcal{O}$, to be denoted
by $\operatorname{ch}\Nil$. As usual,
the characters of Verma modules are computable, i.e.,  one can write the character
of a Verma module starting from  its label, see formula (\ref{eq:Verma_W_char}) below.
So to compute the characters of simple modules (and hence dimensions of the finite dimensional ones)
it is enough to determine the multiplicities of simples in Vermas.
We would like to point out that in a relatively easy special case when $e$
is a principal element in some Levi subalgebra, the multiplicities are already
known, see \cite[Section 4]{LOCat}. Our approach in the present paper builds on a construction
from there.

To determine the multiplicities we will define an exact functor
from the integral block of an appropriate parabolic category
$\mathcal{O}$ for $\g$ to the integral block of an ``equivariant
version'' of the W-algebra category $\mathcal{O}$. This functor
may be thought as a generalization of Soergel's functor $\Vfun$,
\cite{Soergel}, and so will be called a generalized Soergel functor.
An analog of Soergel's functor for parabolic categories $\mathcal{O}$
was studied by Stroppel in \cite{Stroppel_ger},\cite{Stroppel_eng}.
In the special case when $\g$ is of type A,  Brundan and Kleshchev, \cite{BK2},
identified the target category for this functor with a subcategory in
the category of modules over a W-algebra. We remark that the functor
we consider, in general, is not isomorphic to Stroppel's.
Our functor will map Verma modules to some variant of Verma modules
and will be a quotient functor onto its image. Using this functor
we will relate the multiplicities in the W-algebra category $\mathcal{O}$
to those in a parabolic category $\mathcal{O}$. We will also study some
other properties of $\Vfun$ such as a relation to dualities and
the double centralizer property.

\subsection{Results on Goldie ranks}\label{SS_Goldie_results}
Now we are going to state the main results and conjectures of the present paper
related to Goldie ranks. We start with a conjecture.

\begin{Conj}\label{Conj:main}
Let $\Orb$ be a special nilpotent orbit and suppose that $\Orb$ is not one of
the following three {\it exceptional} orbits: $A_4+A_1$ in $E_7$, $A_4+A_1, E_6(a_1)+A_1$ in $E_8$ (we use the Bala-Carter
notation, see \cite[8.4]{CM}). Then there exists an $A$-stable 1-dimensional $\Walg$-module with integral
central character. For the three exceptional orbits there is a 1-dimensional $\Walg$-module with integral central character.\footnote{In a recent preprint of Premet, \cite{Premet_new},
he proved that any W-algebra admits an $A$-stable 1-dimensional representation,
but the central character is generally not integral. For example, \cite{W_classif} implies
that, in the case of the three exceptional orbits, there are no $A$-stable finite dimensional
representations with integral central character, so Premet's representations automatically have
non-integral central characters in these cases.}
\end{Conj}

We remark that the  primitive ideal
corresponding to a 1-dimensional $\Walg$-module is automatically completely
prime, see \cite[Proposition 3.4.6]{wquant} (in fact, this also easily follows from the previous work of Moeglin, \cite{Moeglin2}).

\begin{Thm}\label{Thm:classical}
Conjecture \ref{Conj:main} holds for all special orbits provided $\g$ is classical.
\end{Thm}

Let us state an important corollary of Conjecture \ref{Conj:main}.
Let $\dcell$ be the two-sided cell corresponding to $\Orb$.
Let $Y,\varrho, A$ have the same meaning as in the Subsection \ref{SS_W_irreps}.
Pick $w\in \dcell$ compatible with $\varrho$. To $w$ assign a triple $(x,y,\mathcal{V})$ as explained in
Subsections \ref{SS_W_irreps}, \ref{SS_HC_semis} (for any regular dominant weight).

\begin{Thm}\label{Thm:main_appendix}
Assume that Conjecture \ref{Conj:main} holds for the pair $(\g,\Orb)$.
Then
\begin{enumerate}
\item  The Goldie rank of $\J(w\varrho)$ coincides with the dimension of an irreducible
$\Walg$-module in $Y$ whose $A$-orbit corresponds to $\J(w\varrho)$ (all such modules
differ by outer automorphisms, so  their dimensions are the same).
\item $\displaystyle z_w=\dim \mathcal{V}\frac{|A_y|}{|A_{(x,y)}|}$.
\end{enumerate}
\end{Thm}

Thanks to assertion (2) of Theorem \ref{Thm:main_appendix}, Conjecture \ref{Conj:main}
completes one of Joseph's programs of computation of the Goldie ranks (for integral central
character), see Subsection \ref{SUBSECTION_Goldie_rem} for details.

\begin{Rem}\label{Rem:data_comput} For classical Lie algebras, the group $A$ is commutative, moreover, it is the sum of several copies of $\ZZ/2\ZZ$ (see, for example,
\cite[Theorem 6.1.3]{CM})
and so  $A_{(x,y)}=A_x\cap A_y$, and $\dim \mathcal{V}=1$.
In this case, one can recover all numbers appearing in the right hand side in (2) combinatorially starting from
$w$, see \cite[6.9]{W_classif}. We note however that we do not know how to  recover $\mathcal{V}$ itself
(and $\dim \mathcal{V}$) in the case when $A=S_3,S_4,S_5$ combinatorially starting from $w$, we refer
the reader to \cite[6.8]{W_classif} for the list of orbits with these component groups.
\end{Rem}
We also would like to remark that the formula in (2) is compatible with results on $z_w$ from \cite[2.4,2.13]{Joseph_scale}.

Let us complete this subsection by describing previous results relating dimensions of irreducible
$\Walg$-modules to Goldie ranks. The existence of such a relationship was conjectured by Premet
in \cite[Question 5.1]{Premet2}. In \cite[Proposition 3.4.6]{wquant}, the author proved that the Goldie rank
of an arbitrary primitive ideal $\J$ does not exceed the dimension of the corresponding irreducible module.
In \cite{Premet_Goldie} Premet significantly improved that result by showing that the Goldie
rank always divides the dimension. Moreover, he proved that the equality always holds in type $A$.
However, in other classical types the equality does not need to hold (for ideals with non-integral
central character). A counter-example is provided by the ideal $\J(\rho/2)$ in type $C_n$
for $n$ large enough, see \cite[1.3]{Premet_Goldie}, for details. In \cite{Brundan}
Brundan reproved the equality of the Goldie rank and the dimension for type $A$.

\subsection{Results on characters of simple $\Walg$-modules}\label{SS_W_results}
We will need to fix some notation to state the main result that is the existence of a functor
with certain properties.

Fix a nilpotent element $e\in \g$. We include $e$ into a minimal Levi subalgebra $\g_0$
so that $e$ is a distinguished nilpotent element in $\g_0$.
Choose Cartan and Borel subalgebras $\h\subset \b_0\subset \g_0$. Let $\g=\bigoplus_{i\in \ZZ}\g(i)$ stand for the eigendecomposition for $h$. We recall that a distinguished element is always even (see, for example, \cite[Theorem 8.2.3]{CM}) so $\g_0(i):=\g_0\cap \g(i)$ is zero when $i$ is odd.
Form the W-algebra $\Walg$ from $(\g,e)$.

Pick an integral element $\theta\in \z(\g_0)$ such that $\z_\g(\theta)=\g_0$. Consider the eigen-decomposition
$\g=\bigoplus_{i\in \ZZ}\g_i$ for $\theta$. Let $\b:=\b_0\oplus \g_{>0}$, where we set $\g_{>0}:=\bigoplus_{i>0}\g_i$,
this is a Borel subalgebra in $\g$. Further, set $\p:=\g_0(\geqslant 0)\oplus \g_{>0}$. This is a parabolic
subalgebra in $\g$. Let $P$ denote the corresponding parabolic subgroup.

Let $\Lambda$ denote the weight lattice for $\g$. For $\lambda\in \Lambda$ let $L_{00}(\lambda), L_0(\lambda), L(\lambda)$ denote the irreducible modules with highest weight $\lambda-\rho$ for $\g_0(0),\g_0$ and $\g$, respectively.

We consider two categories: the sum of integral blocks of the parabolic category $\mathcal{O}$ for $(\g,\p)$,
to be denoted by $\mathcal{O}^P$, and  the category $\mathcal{O}$ for $\Walg$, denoted by
$\OCat^\theta(\g,e)$. The definition of the latter will be recalled in Subsection \ref{SS_Cats_noncompl}.
Let $\Lambda_{\p}$ denote the subset of all weights $\lambda$ such that $\langle\lambda,\alpha^\vee\rangle>0$
for all simple roots  $\alpha$ of $\g_0(0)$. The simples in $\mathcal{O}^P$ are precisely $L(\lambda),\lambda\in \Lambda_{\p}$.
Let $\Walg^0$ denote the W-algebra for the pair $(\g_0,e)$. The simples in $\OCat^\theta(\g,e)$
are parameterized by the finite dimensional irreducible $\Walg^0$-modules,
the simple corresponding to $\Nil^0$ will be denoted by $L^\theta_{\Walg}(\Nil^0)$.
Further to every finite dimensional $\Walg^0$-module $\Nil^0$  we can assign
a ``Verma module'' $\Delta^\theta_\Walg(\Nil^0)$.

One can compute the character of $\Delta^\theta_{\Walg}(\Nil^0)$ as follows. Consider the action of $\z(\g_0)$
on the centralizer $\z_\g(e)$ of $e$ in $\g$. Let $\mu_1,\ldots,\mu_k$ be all weights (counted with multiplicities)
of this action that are negative on $\theta$. Suppose that $\z(\g_0)$ acts on $\Nil^0$ with a single weight $\mu_0$.
Then the character of $\Delta^\theta_{\Walg}(\Nil^0)$ equals
\begin{equation}\label{eq:Verma_W_char} e^{\mu_0}\dim\Nil^0\prod_{i=1}^k (1-e^{\mu_i})^{-1}.\end{equation}
We remark that we will be dealing with modules $\Nil^0$ for which the dimension is easy to compute, see 1) below,
generally, these modules will be reducible.

We will produce an exact functor $\Vfun: \OCat^P\rightarrow \OCat^\theta(\g,e)$. The character computation
for the simples with integral central character in $\OCat^\theta(\g,e)$ will be based on the following two properties of $\Vfun$.

1) $\Vfun$ maps the parabolic Verma $\Delta_P(\lambda)$ to $\Delta^\theta_\Walg(\Nil^0)$, where
$\Nil^0$ is a $\Walg^0$-module of dimension equal to $\dim L_{00}(\lambda)$. In particular,
we have
\begin{equation}\label{eq:char_Verma} \operatorname{ch}\Vfun(\Delta_P(\lambda))=e^{\lambda-\rho} \dim L_{00}(\lambda)\prod_{i=1}^k (1-e^{\mu_i})^{-1}.\end{equation}

2) One can also describe the image of $L(\lambda)$ under $\Vfun$.   By $\Lambda^{max,0}_{\p}$ we denote the subset of $\Lambda_{\p}$ consisting of all $\lambda$
such that the Gelfand-Kirillov (GK) dimension of  $L_0(\lambda)$  equals to $\dim \g_0(<0)$, this is the maximal
possible GK dimension for an object in the parabolic category $\mathcal{O}$ for $(\g_0, \p\cap \g_0)$. If $\lambda\not\in \Lambda^{max,0}_{\p}$,
then $\Vfun(L(\lambda))=0$.
%Otherwise, represent $\lambda$ in the form $w\varrho_0$, where
%$w$ is an element of the Weyl group $W_0$ of $G_0$, and $\varrho_0$ is a weight dominant
%for $\g_0$ compatible with $w$.
If $\lambda\in \Lambda^{max,0}_{\p}$, then the associated variety of the ideal $\J_0(\lambda)=\Ann_{U(\g_0)}L_0(\lambda)$ is the closure of $G_0e$. We have $$\Vfun(L(\lambda))=\mathcal{V}\otimes
\bigoplus_{\Nil^0}L_{\Walg}^\theta(\Nil^0).$$ Here the sum is taken over all irreducible $\Walg^0$-modules $\Nil^0$
lying over $\J_0(\lambda)$, the description of the set of such $\Nil^0$ was recalled in
Subsection \ref{SS_W_irreps}, let $n_\lambda$ be the number. When $\g$ is classical, then $\mathcal{V}$ is just $\K$. In general,
$\mathcal{V}$ is described as follows. We represent $\lambda$ in the form $w\varrho_0$, where $\varrho_0\in \Lambda$
is dominant for $\g_0$ and $w\in W_0$ is compatible with $\varrho_0$. Then $\mathcal{V}$ is the vector space from
the triple $(x,y,\mathcal{V})$ assigned to $w$ (viewed as an element in $W_0$) in Subsection \ref{SS_W_irreps},
see also Subsection \ref{SS_HC_semis}.
%So we can only have $\dim \mathcal{V}>1$ if the stabilizer of $(x,y)$ in $\bA_0$ (the Lusztig quotient
%assigned to the two-sided cell of $w$ in $W_0$) is non-abelian, this happens for some orbits
%in the exceptional Lie algebras, in all cases but two the group $\bA_0$ is $S_3$, see \cite[6.8]{W_classif}
%for the list of orbits.

%Modulo $w\varrho_0\in \Lambda_{\p}$,
%the inclusion $w\varrho_0\in \Lambda^{max,0}_{\p}$ is equivalent to  $w$
%lying in the two-sided cell $\dcell_0$ corresponding to the special orbit $G_0e$
%(while without the assumption $w\varrho_0\in \Lambda$ the inclusion $w_0\varrho_0\in \Lambda_{\p}^{max,0}$ is equivalent %to $w_0$ lying in the fixed right
%cell corresponding to $\g_0(\geqslant 0)$, we will elaborate on that in \ref{SSS_BG_equiv}).
%So we can consider the Lusztig
%quotient $\bA_0$ corresponding to $\dcell_0$, the subgroup $H_0$ corresponding to the left cell $\cell_0\subset W_0$
%containing $w$ and the $H_0$-module $\mathcal{V}$ corresponding to $w_0$ (the $\bA_0$-orbit corresponding to
%the fixed right cell is a single point, see \ref{Rem:right_cell}).
%Then $\Vfun$ maps $L(\lambda)$ to  $\mathcal{V}\otimes \bigoplus L^\theta_\Walg(\Nil^0)$,
%where the sum is taken over all (total of $|\bA_0/H_0|$) irreducible $\Walg^0$-modules
%$\Nil^0$ in the $\bA_0$-orbit corresponding to $w_0\varrho_0$.

It is not difficult to show (see Subsection \ref{SS_summary})
that  $\operatorname{ch}L^\theta_{\Walg}(\Nil^0)$ does not depend
on the choice of $\Nil^0$ in the $A_0$-orbit (where $A_0$ denotes the component group for $e$ viewed as an element
in $\g_0$). Let us explain how to compute this character.

Define the integers  $c_{wu}$ by the equality $L(w\varrho_0)=\sum_{u}c_{wu}\Delta_P(u\varrho_0)$
in the Grothendieck group $K_0(\OCat^P_{\varrho_0})$ in the infinitesimal block of $\OCat^P$ with generalized central character corresponding to $\varrho_0$. Recall that $K_0(\OCat^P_{\varrho_0})$ is a free abelian group with basis $\Delta_P(u\varrho_0)$,
where $u$ runs over all elements in $W_0$ compatible with $\varrho_0$ and such that $u\varrho_0\in \Lambda_\p$.
Let us also point out that the numbers $c_{wu}$ are known, they equal $c_{wu}(1)$, where
$c_{wu}(q)$ is a parabolic Kazhdan-Lusztig polynomial.

Then we have
\begin{equation}\label{eq:mult}n_{w\varrho_0}\dim \mathcal{V}\cdot \operatorname{ch}L^\theta_{\Walg}(\Nil^0)=\sum_{u}c_{wu}\operatorname{ch}\Vfun(\Delta_P(u\varrho_0)).\end{equation}
where $\operatorname{ch}\Vfun(\Delta_P(u\varrho_0))$ can be computed by (\ref{eq:char_Verma}) and $n_{w\varrho_0}$
is the number of the irreducible finite dimensional $\Walg^0$-modules lying over $\J_0(w\varrho_0)$.
As in Remark \ref{Rem:data_comput},  when all simple summands of $\g_0$ are classical, we can recover $n_{w\varrho_0}\dim \mathcal{V}=n_{w\varrho_0}$ combinatorially starting from $w$.
%We do not know
%how to determine $\mathcal{V}$ as a $\bA_{0,x,y}$-module combinatorially, in particular, we do not
%know how to determine $\dim \mathcal{V}$ in the cases when $\bA_0$ is not abelian.
%We remark that we can have $\dim \mathcal{V}>1$ only if $\bA_0=S_3,S_4,S_5$  that can only happen in
%exceptional types (with $S_4$ occurring for one orbit
%in type $F_4$ and $S_5$ for one orbit in type $E_8$), see, for example, \cite[6.8]{W_classif}, for the list of orbits. In %general, it is unclear how to determine $\mathcal{V}$ from $w$.
%And  even if $A_0$ is not abelian  we could compute the right hand side of (\ref{eq:mult})
%for all $w$ in the intersection of a suitable left cell and a suitable right cell in $W_0$.
%Since we have $\dim\mathcal{V}=1$ for at least
%one such $w$, this can be used to determine $\operatorname{ch}L^\theta_{\Walg}(\Nil^0)$.

To finish this discussion we would like to point out that the functor $\Vfun$ has other nice properties
to be investigated in the present paper. We summarize them in Theorem \ref{Thm:Vfun}.

Let us mention some special cases where the character formulas were known before.
In type A, they follow from the work of Brundan and Kleschev, \cite{BK2}. Somewhat more generally,
if $e$ is principal in a Levi subalgebra, then the character formulas follow from the results of
\cite[Section 4]{LOCat}. We remark that both computations are  based on using essentially opposite
special cases of the functor $\Vfun$.

Another result related to ours is the main result of \cite{BM}. There Bezrukavnikov and Mirkovic deal
with simple non-restricted representations over semisimple Lie algebras
in characteristic $p$. More precisely, let $\mathbb{F}$ be an algebraically
closed field of characteristic $p\gg 0$. Since $p\gg 0$, there is a natural bijection
between the nilpotent orbits in $\g^*$ and in $\g^*_{\mathbb{F}}$, let $\Orb_{\mathbb{F}}$
denote the orbit in $\g_{\mathbb{F}}$ corresponding to $\Orb$.
Bezrukavnikov and Mirkovic consider the category of all
$\g_{\mathbb{F}}$-modules with trivial Harish-Chandra character and $p$-character  $e_{\mathbb{F}}\in\Orb_{\mathbb{F}}$
(we can view $\Orb_{\mathbb{F}}$ as an orbit in the Frobenius twist $\g^{*(1)}$ because $\g^*$ and $\g^{*(1)}$
are naturally identified).
They identify the complexified $K_0$ of this category
with $H^*(\mathcal{B}_e)$, where $\mathcal{B}_e$ stands for the Springer fiber (in characteristic $0$).
Then they consider a deformation  $H^*(\mathcal{B}_e)$ over $\K[q^{\pm 1}]$ given by $H^*_{\K^\times}(\mathcal{B}_e)$
for a suitable $\K^\times$-action on $\mathcal{B}_e$ (so that $q$ becomes the equivariant parameter for the
$\K^\times$-action). Then $H^*_{\K^\times}(\mathcal{B}_e)$ becomes an explicit module over the affine Hecke algebra
$\mathcal{H}$ for $W$. The $\mathcal{H}$-module $H^*_{\K^\times}(\mathcal{B}_e)$ comes equipped with a
$\K[q^{\pm 1}]$-bilinear form and a semilinear bar-involution
that can be written in terms of the $\mathcal{H}$-module structure.
From these data one can define a canonical basis of $H^*_{\K^\times}(\mathcal{B}_e)$
in the sense of Kashiwara (the elements have to be fixed by the bar-involution and orthonormal modulo $q$).
The classes of simple $\g_{\mathbb{F}}$-modules in $H^*(\mathcal{B}_e)$ are then the specializations of the
canonical basis elements at $q=1$. If one knows the canonical basis elements, then one can write the
dimension formulas for the simples. However it is unclear how to write those elements explicitly
(with a possible exception of the case when $e$ is principal in a Levi subalgebra but that case was
not worked out either).

Now let us explain a connection of \cite{BM} to our work.
It follows from \cite[Theorem 1.1]{Premet0} that the central reduction of $U(\g_{\mathbb{F}})$ at the $p$-character $e_{\mathbb{F}}$
is isomorphic to the matrix algebra of rank $p^{\dim \Orb/2}$ over the analogous reduction of the W-algebra.
So \cite{BM} basically describes the classes of the simples over the latter algebra.
For $p$ large enough, one can take an irreducible finite dimensional
representation of $\Walg$, reduce this representation (or rather its integral form; see, for example,
\cite[Section 2]{Premet_Goldie} for technical details on this procedure) modulo $p$ and get an irreducible finite
dimensional representation for the W-algebra in characteristic $p$ with $p$-character $e_{\mathbb{F}}$.
However, a relatively small portion of simples in characteristic $p$ is obtained in this way. So, comparing to
\cite{BM}, we get explicit formulas but for a smaller number of modules.

Finally,  let us point out that we get two ways to compute the Goldie ranks of primitive ideals: one via a more explicit formula
coming from the character formulas in the W-algebra category $\mathcal{O}$ and one implementing
a program of Joseph. It is completely unclear to the author why these formulas give the same result!

\subsection{Organization and content of the paper}
Let us describe the content of this paper, section by section.

Section \ref{S_W_alg} is preliminary and does not contain new results. There we explain (a slight ramification)
of the definition of a W-algebra via a quantum slice construction that first appeared in \cite{wquant} somewhat implicitly
and in \cite{W-prim} in a more refined form. Next, in Subsection \ref{SS_dagger_functor}, we recall the main construction
of \cite{HC}: a functor $\bullet_{\dagger}$ between the categories of Harish-Chandra bimodules for
$\U$ and for $\Walg$. There are several reasons why we need this functor, for example, the functor
$\Vfun$ mentioned above can be defined analogously to $\bullet_{\dagger}$. Also $\bullet_{\dagger}$
plays an important role both in the study of $\Vfun$ and in the computation of Goldie ranks.
In Subsection \ref{SS_W_classif} we recall the classification of finite dimensional irreducible $\Walg$-modules
obtained in \cite{W_classif}. Subsection \ref{SS_HC_semis} describes the images of certain simple
Harish-Chandra $\U$-bimodules under $\bullet_\dagger$. This gives rise to the generalized RSK correspondence mentioned in Subsection \ref{SS_W_irreps} and again plays an important role in the computation of characters and a crucial role in the computation
of Goldie ranks. Subsection \ref{SS_parab_ind} briefly recalls some results on a parabolic
induction for W-algebras obtained in \cite{Miura}. We will need these results to prove Conjecture \ref{Conj:main}
for classical groups in Subsection \ref{SECTION_main_conj}. Finally, large Subsection \ref{SS_iso_comp} deals with various isomorphisms,
mostly of certain completions, that can be deduced from the quantum slice construction. These isomorphisms play an important role
in the construction of $\Vfun$.

In Section \ref{S_Cats} we introduce categories of modules that are involved in our construction and study
of $\Vfun$. In the first subsection we recall some known facts about parabolic categories $\mathcal{O}$
including the Bernstein-Gelfand equivalence. The latter is a crucial tool to transfer the properties of
$\bullet_\dagger$ to those of $\Vfun$. In Subsection  \ref{SS_Cats_noncompl} we recall the definition
of a category $\mathcal{O}$ for $\Walg$ basically following \cite{BGK}. We also recall the construction
of Verma modules and explain how to compute their characters again following \cite{BGK}.  Next, we introduce
the category of Whittaker $\U$-modules basically as it appeared
in \cite{LOCat} and recall an equivalence between W-algebra categories $\OCat$ and the Whittaker
categories.  In the final part of Subsection \ref{SS_Cats_noncompl} we introduce a duality  functor for a W-algebra category $\mathcal{O}$,
the construction is pretty standard. Subsection \ref{SS_cats_compl} is new but is not very original.
There we study the completed versions of $W$-algebra categories $\mathcal{O}$ and of  Whittaker categories.
A completed version of a W-algebra category $\mathcal{O}$ is obtained from the usual one via completion.
This is no longer so for Whittaker categories, a completed version is different from the usual one.
Yet, we show that it still has analogs of Verma modules and is equivalent to a completed version of a category $\mathcal{O}$
for the W-algebra.

Section \ref{S_Vfun} is one of the two central sections of this paper. There we introduce a functor $\Vfun$
and study its properties.   In Subsection \ref{SS_dagger_e} we define a functor denoted by $\bullet_{\dagger,e}$ that generalizes
$\bullet_{\dagger}$. Its source category is an appropriate category of $\U$-modules and the target
category is a  category of $\Walg$-modules. After the functor is constructed we study
some its general properties: we show that it is exact, intertwines tensor products with Harish-Chandra
bimodules and study the question of existence of a right adjoint functor. In Subsection \ref{SS_Vfun_constr}
we construct a functor $\Vfun:\OCat_\nu^{P}\rightarrow \OCat^\theta(\g,e)_\nu^R$ (the notation will
be explained below). We provide three constructions and show that they are equivalent. The first construction
is as $\bullet_{\dagger,e}$, while the other two constructions use the equivalences between the two
versions (usual  and completed) of the Whittaker category and the two versions of the category $\OCat$
for $\Walg$.  In process of proving that the three constructions are equivalent we establish
some important properties of $\Vfun$, for example, its behavior on parabolic Verma modules
and a relation with the duality functors.
The last Subsection \ref{SS_Vfun_prop} establishes some further properties of $\Vfun$. Most importantly,
there we show that $\Vfun$ is a quotient functor onto its image and has the double centralizer property.
The properties of $\Vfun$ including those needed for character formulas are summarized in
Theorem \ref{Thm:Vfun}. Further, under some restrictions on $P$ we show that
$\Vfun$ is 0-faithful, i.e. fully faithful on modules admitting a parabolic Verma filtration.

In Section \ref{S_Goldie} we deal with Goldie ranks. First, we recall some known results on them,
almost entirely due to Joseph. In Subsection \ref{SS_scale}  we use
the results quoted in Subsections \ref{SS_dagger_functor},\ref{SS_HC_semis} to prove some formula for the scale
factors $z_w$. In Subsection \ref{SS_Thm_Goldie_proof} we complete the proof of Theorem \ref{Thm:main_appendix}
modulo Conjecture \ref{Conj:main}. Finally, in Subsection \ref{SECTION_main_conj} we prove Conjecture
\ref{Conj:main} for the classical types. For this we first use a reduction procedure based on the parabolic
induction for W-algebras. Then we deal with the three classical types (we do not consider type A)
one by one, explaining type B in detail and then describing modifications to be made for types C and D.

Section \ref{S_suppl} deals with three topics that are related to the functor $\Vfun$. In Subsection
\ref{SS_HC_mod} we discuss the functor $\bullet_{\dagger,e}$ for Harish-Chandra $(\g,K)$-modules.
In Subsection \ref{SS_parab_top_quot} we study a different version
of the functor $\bullet_{\dagger,e}$ for a parabolic category $\mathcal{O}$, one associated with the Richardson
orbit. This functor was studied by Stroppel in \cite{Stroppel_ger},\cite{Stroppel_eng} but she did not
relate the target category to $W$-algebras.
We show that many properties of the functor $\Vfun$ carry over to this situation. We also analyze
conditions that guarantee the 0-faithfulness of our functor in type A. In the last subsection of Section \ref{S_suppl}
we extend (in a straightforward way) the definition of $\bullet_{\dagger,e}$ to the categories of modules
over Dixmier algebras (i.e., algebras
equipped with a homomorphism from $\U$ turning them into Harish Chandra bimodules).
We are basically interested in two classes of Dixmier algebras: one coming from
Lie superalgebras and the other from quantum groups at roots of unity: we show that the quantum Frobenius
epimorphism splits turning (in fact, some mild quotient of) the Lusztig form of a quantum group into
a Dixmier algebra.

\subsection{Notation and conventions}\label{SS_not}
In this subsection we will list some notation and conventions used in this paper. They will be
duplicated (and explained in more detail)  below.

Our base field is an algebraically closed field $\K$ of characteristic $0$.

\subsubsection{Algebras and groups} Let $G$ be a reductive algebraic group over $\K$ and $\g$ be its Lie algebra.
Fix a nilpotent element $e$ and include it into an $\sl_2$-triple $(e,h,f)$.
Let $\Orb$ denote the $G$-orbit of $e$ (under the adjoint action). Fix a non-degenerate symmetric $G$-invariant
form on $\g$, say $(\cdot,\cdot)$, whose restriction to the rational form of a Cartan
subalgebra is positive definite. Using this form we can identify $\g$ with $\g^*$.
We write $\chi$ for the image of $e$ under this identification. By $Q$ we denote the centralizer of $(e,h,f)$ in
$G$. This is a reductive subgroup of $G$. We write $\g(i)$ for the eigenspace of $\operatorname{ad}h:=[h,\cdot]$
with eigenvalue $i$ and use the notation like $\g(>0)$ for $\bigoplus_{i>0}\g(i)$.

By $\U$ or $U(\g)$ we denote the universal enveloping algebra of $\g$. We consider the Slodowy slice
$S=e+\z_\g(f)\subset \g\cong \g^*$, where  $\z_\g(\bullet)$ stands for the centralizer in $\g$. Inside $\g^*$, the slice
$S$ is realized as $\chi+[\g,f]^\perp$, where the superscript $\perp$ indicates the annihilator. Let $\Walg$ denote the W-algebra
of the pair $(\g,e)$. Sometimes when we want to explicitly indicate the Lie algebra and the nilpotent
element used to produce a W-algebra, we use the notation like $U(\g,e)$ for the W-algebra.

Set $V:=[\g,f]$. This is a symplectic vector space with form $\omega(x,y)=(e,[x,y])$. By $\W$ we denote the Weyl algebra of $V$,
i.e., $\W=T(V)/(u\otimes v-v\otimes u-\omega(u,v))$.

We also will consider the ``homogenizations'' of the algebras above, they will be decorated
with the subscript ``$\hbar$''.

In Subsection \ref{SS_iso_comp}, Section \ref{S_Cats} and Subsections \ref{SS_Vfun_constr},\ref{SS_Vfun_prop}
we will use the following notation. We will fix an integral element $\theta\in \g$ centralizing
$e,h,f$ such that  the element $e$ is even in $\g_0:=\z_\g(\theta)$. By $G_0$ we denote the connected subgroup
of $G$ corresponding to $\g_0$. We will consider the gradings
by eigenspaces of $\operatorname{ad}\theta$: $\g=\bigoplus_{i\in \ZZ}\g_i, V=\bigoplus_i V_i, \U=\bigoplus_{i}\U_i,
\Walg=\bigoplus_i \Walg_i$. We use the notation like $\U_{\geqslant 0}$ similarly to the above.
Further, we consider the algebras $\U^0:=\U_{\geqslant 0}/(\U_{\geqslant 0}\cap \U\U_{>0}),\Walg^0, \W^0$
(to see that these spaces are actually algebras we notice that, for example, $\U_{\geqslant 0}$
is a subalgebra in $\U$, and $\U_{\geqslant 0}\cap \U\U_{>0}$ is a two-sided ideal in $\U_{\geqslant 0}$).

We write $\m$ for $\g_0(<0)\oplus \g_{>0}$, $\bar{\m}$  for $\g_0(<0)\oplus \g_{<0}$, $\p$ for
$\g_{0}(\geqslant 0)\oplus \g_{>0}$, and $\t$ for $\z(\g_0)$. Let $M,\bar{M},P,T$
be the corresponding connected subgroups of $G$ so  that $M,\bar{M}$ are unipotent,
$P$ is parabolic, $T$ is a torus. Further we choose a Cartan subalgebra $\h\subset\g_0$
and a Borel subalgebra $\b\subset\g$ such that $\b_0:=\b\cap\g_0$ is a Borel subalgebra in  $\g_0$.
Let $\Lambda\subset \h^*$ denote the weight lattice of $G$ and $\Lambda^+$ be the subset of dominant weights.
As usual, we write $\rho$ for half the sum of positive roots.

We write $\m^0:=\m\cap \g_0, \bar{\m}^0:=\bar{\m}\cap \g_0$.
We consider the shift $ \m_\chi:=\{x-\langle \chi,x\rangle| x\in \m\}\subset \g\oplus\K$ and
the similar shifts $\bar{\m}_\chi, \m^0_\chi, \bar{\m}^0_\chi$.
Further, we set $\v:=\m\cap V, \bar{\v}:=\bar{\m}\cap V$.

We consider the completions
$\U^{\wedge}:=\varprojlim_{n\rightarrow +\infty} \U/\U\m_\chi^n, \wU:=\varprojlim_{n\rightarrow +\infty} \U/\bar{\m}_\chi^n \U, \W^\wedge:=\varprojlim_{n\rightarrow +\infty} \W/\W\v^n,
\wW:=\varprojlim_{n\rightarrow +\infty} \W/\bar{\v}^n\W, \Walg^\wedge:=\varprojlim_{n\rightarrow +\infty} \Walg/\Walg\Walg_{>0}^n,
\wWalg:=\varprojlim_{n\rightarrow +\infty} \Walg/\Walg_{<0}^n\Walg$ and the similar completions of $\U^0,\Walg^0,\W^0$.

Finally, in this context, for $R$ we take the centralizer of $T\subset Q$ in $Q$. This is a reductive subgroup.

\subsubsection{Categories and functors}
Here we are going to explain some notation used mostly in Sections \ref{S_Cats},\ref{S_Vfun}.

We usually abbreviate ``Harish-Chandra'' as ``HC''. We use the notation $\HC(\U)$ for the category
of HC $\U$-bimodules, and $\HC^Q(\Walg)$ for the category of $Q$-equivariant HC $\Walg$-bimodules.

For an algebraic subgroup $K\subset G$ and a character $\nu$ of its Lie algebra $\k$ we write $\tilde{\OCat}^K_\nu$ for the category
of $(K,\nu)$-equivariant modules. We say that a $(\U,K)$-module $\M$ is $(K,\nu)$-equivariant if the structure map $\U\otimes \M\rightarrow \M$ is $K$-equivariant and the differential of the $K$-action coincides with the $\k$-action given by $(x,m)\mapsto xm-\nu(x)m$.
We write $\OCat^K_\nu$ for the full subcategory in $\tilde{\OCat}^K_\nu$ consisting of finitely generated modules. When $\nu=0$,
we suppress the subscript.

In particular, we are going to consider the parabolic categories $\OCat^P_\nu$ for $(\g,\p)$ and $\OCat^{P_0}_\nu$
for $(\g_0,\p_0)$. The parabolic Verma modules with highest weights $\lambda-\rho$ are denoted by
$\Delta_P(\lambda), \Delta_{P_0}(\lambda)$.  Their simple quotients are denoted by $L(\lambda),L_0(\lambda)$.
We write $L_{00}(\lambda)$ for the irreducible $\g_0(0)$-module with highest weight $\lambda-\rho$.
The Verma modules in the whole BGG category $\mathcal{O}$ are denoted by $\Delta(\lambda)$.

The categories $\OCat$ for $\Walg$ constructed from the element $\theta$ are denoted by $\OCat^\theta(\g,e)_\nu$
(the usual version), $\OCat^\theta(\g,e)_\nu^R$ (the $R$-equivariant version), $\hat{\OCat}^\theta(\g,e)_\nu$
(the completed version). For the Whittaker categories we use the notation like $\Wh^\theta(\g,e)_\nu$, etc.
Again, when $\nu=0$, we suppress it from the notation.

There are various functors between the categories under consideration. We write $\widehat{\bullet}$
for the weight completion functor, it maps a module of the form $\bigoplus_\mu \M_\mu$, where $\M_\mu$
is the $\t$-weight space with weight $\mu$, to $\prod_\mu \M_\mu$. For example, we will have the completion
functor $\widehat{\bullet}: \OCat^\theta(\g,e)_\nu\rightarrow \hat{\OCat}^\theta(\g,e)_\nu$.

We also consider a different type of completions to be denoted by $\bullet^\wedge$. These are $\bar{\m}_\chi$-adic completions:
$\M^\wedge:=\varprojlim_{n\rightarrow +\infty}\M/\bar{\m}_\chi^n \M$, this gives a functor, for instance, from $\OCat^P_\nu$
to $\hat{\Wh}^\theta(\g,e)_\nu^R$. Somewhat dually, we have a functor $\Wh_\nu$ that takes a $\U$-module and
maps it to the sum of all submodules that are objects in $\Wh^\theta(\g,e)_\nu$.

On some categories of interest, we have a ``naive'' duality functor to be denoted by $\bullet^\vee$.
This functor is the composition of taking the restricted dual, to be denoted by $\M\mapsto \M^{(*)}$,
and the twist by an anti-automorphism, say $\tau$, $\M^{(*)}\mapsto \,^\tau\!\M^{(*)}$.
We will also consider the homological duality functor that comes from $\operatorname{RHom}$.
This functor is denoted by $\dual$.

Next, we have ``Verma module'' functors: $\Delta^\theta_{\Walg}: \OCat^\theta(\g_0,e)_\nu\rightarrow
\OCat^\theta(\g,e)_\nu, \OCat^\theta(\g_0,e)^R_\nu\rightarrow \OCat^\theta(\g,e)^R_\nu$, $\Delta^0: \U^0$-$\operatorname{mod}\rightarrow
\U$-$\operatorname{mod}, \Delta_\U^\theta: \OCat^\theta(\g_0,e)_\nu\rightarrow \Wh^\theta(\g,e)_\nu,\hat{\Delta}^\theta_\Walg:
\OCat^\theta(\g_0,e)\rightarrow \hat{\OCat}^\theta(\g,e)$, etc. We will consider the  quotients $L^?_?(\bullet)$
by the maximal submodule that does not intersect the highest weight subspace.

\subsubsection{Cells, primitive ideals, etc.}
This notation is mostly used in Sections \ref{S_Vfun},\ref{S_Goldie}. We write $W$ (resp., $W_0$)
for the Weyl group of $\g$ (resp., $\g_0$) acting on the Cartan  subalgebra $\h$.

We usually denote a two-sided cell by $\dcell$ and left cells by $\cell,\cell_1$, etc.. We write
$\Orb_\dcell$ for the  special orbit corresponding to $\dcell$. Let
$\bA$ denote the Lusztig quotient of the component group $A(\Orb_\dcell)$ (i.e., the component
group of the centralizer of some element of $\Orb_\dcell$ in the adjoint group).
By $H_\cell$ we denote the Lusztig subgroup of $\bA$ corresponding to $\cell$. The corresponding
objects for $\g_0$ are decorated with the subscript ``$0$''.

To an element $w$ we can assign a simple HC bimodule $\M_w$ of $\g$-finite maps $\Delta(\rho)\rightarrow L(w\rho)$.
We write $J(\lambda)$ for the annihilator of $L(\lambda)$ in $\U$. For a HC bimodule $\M$ we write
$\VA(\M)$ for its associated variety that is a subvariety in $\g\cong \g^*$.

By $\J_{P,\nu}$ we denote the annihilator of $\Delta_P(\nu+\rho)$ in $\U$. Here $\nu$ is a character of $\p$.

To a left cell $c$ in $W$ one assigns its {\it cell module}, we denote it by $\mathsf{M}(c)$.

There will be some other notation related to Goldie ranks introduced in Section \ref{S_Goldie}.

\subsubsection{Miscellaneous notation}
Below we present some other notation to be used in the paper.

\setlongtables

\begin{longtable}{p{2.5cm} p{12.5cm}}
$\A^{opp}$& the opposite algebra of $\A$.\\
$\widehat{\otimes}$&the completed tensor product of complete topological vector spaces/ modules.\\
$(a_1,\ldots,a_k)$& the two-sided ideal in an associative algebra generated by  elements $a_1,\ldots,a_k$.\\
 $A^{\wedge_\chi}$&
the completion of a commutative (or ``almost commutative'') algebra $A$ with respect to the maximal ideal
of a point $\chi\in \Spec(A)$.\\
$\Ann_\A(\M)$& the annihilator of an $\A$-module $\M$ in an algebra
$\A$.\\
$D(X)$& the algebra of differential operators on a smooth variety $X$.\\
$G^\circ$& the connected component of unit in an algebraic group $G$.\\
$(G,G)$& the derived subgroup of a group $G$.\\
$G_x$& the stabilizer of $x$ in $G$.\\
$\Goldie(\A)$& the Goldie rank of a prime Noetherian algebra $\A$.\\
$\gr \A$& the associated graded vector space of a filtered
vector space $\A$.\\
$R_\hbar(\A)$&$:=\bigoplus_{i\in
\mathbb{Z}}\hbar^i \F_i\A$ :the Rees $\K[\hbar]$-module of a filtered
vector space $\A$.\\
$S(V)$& the symmetric algebra of a vector space $V$.\\
$\mathfrak{X}(H)$& the group of characters of an algebraic group $H$.
\end{longtable}

\subsection{Acknowledgements}
I would like to thank J. Adams, R. Bezrukavnikov, J. Brundan, I. Gordon, A. Joseph, G. Lusztig, V. Ostrik, A. Premet, C. Stroppel,
D. Vogan and W. Wang for stimulating discussions related to various parts of this paper. Also I would like to thank the referee
for many useful comments that allowed me to improve the exposition.

\section{W-algebras and finite dimensional modules}\label{S_W_alg}
\subsection{W-algebras via a slice construction}\label{SSS_slice}
Let $G,\g, e,h,f, (\cdot,\cdot),\chi:=(e,\cdot), \U,Q,\Orb$ have the same meaning as in Subsection \ref{SS_not}.
Further, let $\tau$ be either an involutive anti-automorphism of
$G$ with $\tau(e)=e, \tau(f)=f, \tau(h)=-h$ or (for formal technical  reasons)
the identity.  We will mostly use $\tau$ described in  \ref{SSS_setting}.

Let us recall the construction of a W-algebra associated to the pair $(\g,e)$ (or, more precisely, to $\g$
and  the $\sl_2$-triple $e,h,f$). Consider the group $\widehat{Q}:=\ZZ/2\ZZ\ltimes (Q\times \K^\times)$, where
the action of a nontrivial element $\varsigma\in \ZZ/2\ZZ$
on $Q$ is induced by $\tau$, while the action on $\K^\times$ is trivial. The group $\widehat{Q}$ acts on
$\g$ and $\g^*$: for $q\in Q, t\in \K^\times,x\in \g, \alpha\in \g^*$ we set: $$q.x=\operatorname{Ad}(q)x,
q.\alpha=\operatorname{Ad}^*(q)\alpha,\quad t.x=t^2 t^h x, t.\alpha=t^{-2}t^h\alpha,
\quad\varsigma.x=\tau(x), \varsigma.\alpha=\tau(\alpha).$$ Here $t\mapsto t^h$ stands for the one-parameter
subgroup of $G$ corresponding to the semisimple element $h$.  We remark that the $\widehat{Q}$-action fixes $\chi$. Also the
{\it Slodowy slice} $S:=\chi+ [\g,f]^\perp\subset \g^*$ is $\widehat{Q}$-stable (under the isomorphism
$\g^*\cong \g$ given by the Killing form the Slodowy slice becomes $e+\z_\g(f)$, we would like to point out that the isomorphism
is not $\widehat{Q}$-equivariant because $\widehat{Q}$ rescales the Killing form). The restriction of the $\widehat{Q}$-action to
$\K^\times$ is often called the {\it Kazhdan action}.

Consider the universal enveloping algebra $\U$ of $\g$. We endow this algebra with a ``doubled'' PBW filtration
$\F_0\U=\K\subset \F_1\U\subset\ldots\subset \U$ such that $\F_{i}\U$ is the span of all monomials of degree
$\leqslant i/2$. Let $\U_\hbar$ stand for the Rees algebra with respect to this filtration,
$\U_\hbar:=\bigoplus_{i=0}^{+\infty} \F_i\U\cdot \hbar^i$. On $\U_\hbar$ we have an action of the group
$\widetilde{Q}:=\ZZ/2\ZZ\times \widehat{Q}$ given as follows: the action of $\ZZ/2\ZZ$ on $\g$ is trivial
and the action of $\widehat{Q}$ on $\g$ is as before, while $\ZZ/2\ZZ\ltimes Q\subset \widehat{Q}$ acts
trivially on $\hbar$,
the other copy of $\ZZ/2\ZZ$ acts on $\hbar$ by changing the sign, and $t.\hbar=t\hbar$ for
$t\in \K^\times$.

We can view $\chi$ as a homomorphism $\U_\hbar\rightarrow \K$ via $\U_\hbar\xrightarrow{/\hbar}S(\g)\xrightarrow{\chi}\K$.
Let  $I_\chi, \tilde{I}_\chi$ denote the kernels of $\chi$ in $S(\g),\U_\hbar$, respectively.  Since $\chi$ is fixed by $\widetilde{Q}$,
we see that $\tilde{I}_\chi$ is $\widetilde{Q}$-stable and  therefore
$\widetilde{Q}$ acts on the completion $\U_\hbar^{\wedge_\chi}:=\varprojlim_{n\rightarrow +\infty}\U_\hbar/\tilde{I}_\chi^n$.

Set $V:=[\g,f]$. We can view $V$ as a subspace in $T^*_\chi \g^*=I_\chi/I_{\chi}^2$
via the map $V\rightarrow I_{\chi}$ sending $v$ to $v-\langle\chi,v\rangle$.
We claim that  there is an embedding $\iota:V\rightarrow \tilde{I}_\chi^{\wedge_\chi}$
with the following properties:
\begin{enumerate}
\item $\iota$ is $\widetilde{Q}$-equivariant.
\item The composition of $\iota$ with a natural projection $\tilde{I}^{\wedge_\chi}_\chi\twoheadrightarrow I_\chi/I_\chi^2$
is the inclusion $V\subset I_\chi/I_\chi^2$.
\item $[\iota(u),\iota(v)]=\hbar^2 \omega(u,v)$ for all $u,v\in V$, where  $\omega$ denotes the Kostant-Kirillov form on $V$,
i.e. $\omega(u,v)=(e,[u,v])$.
\end{enumerate}
See \cite[2.1]{W-prim} for the proof (the construction there formally covers the $\ZZ/2\ZZ\times \K^\times\times Q$-action, but the proof
with $\widetilde{Q}$ is similar).
So $\iota$ extends to a homomorphism (actually an embedding)
$\W_\hbar^{\wedge_0}\hookrightarrow \U_\hbar^{\wedge_\chi}$, where $\W_\hbar:=\W_\hbar(V)$ is the homogenized Weyl algebra
 $$T(V)[\hbar]/(u\otimes v-v\otimes u-\hbar^2\omega(u,v))$$ and the superscript $\wedge_0$
means the completion at $0$, i.e., $\W_\hbar(V)^{\wedge_0}:=\varprojlim_{n\rightarrow +\infty} \W_\hbar(V)/J^n$,
where we write $J$ for the ideal generated by $V$ and $\hbar$.

Then the algebra $\U^{\wedge_\chi}_\hbar$ decomposes into the tensor product
\begin{equation}\label{eq:slice_decomp}\U^{\wedge_\chi}_\hbar=\W_\hbar^{\wedge_0}\widehat{\otimes}_{\K[[\hbar]]}\Walg'_\hbar,\end{equation}
where $\Walg'_\hbar$ stands for the centralizer of $\iota(V)$ in $\U^{\wedge_\chi}_\hbar$.
This is shown analogously to the proof of \cite[Proposition 3.3.1]{HC}.
We remark that $\Walg'_\hbar/(\hbar)$ is naturally identified with $\K[S]^{\wedge_\chi}$: there is
a natural homomorphism $\Walg'_\hbar\rightarrow \K[S]^{\wedge_\chi}=\K[\g^*]^{\wedge_\chi}/(V)$, this homomorphism
factors through  an isomorphism $\Walg'_\hbar/(\hbar)\xrightarrow{\sim} \K[S]^{\wedge_\chi}$ thanks to (\ref{eq:slice_decomp}).

The subalgebra $\Walg_\hbar$ of all {$\K^\times$-finite} elements (i.e., finite sums of semi-invariants) in $\Walg'_\hbar$ (for the Kazhdan action) is dense in $\Walg'_\hbar$ and $\Walg_\hbar/(\hbar)=\K[S]$. Those claims are direct consequences of two observations: that the isomorphism $\Walg_\hbar'/\hbar \Walg_\hbar'\cong \K[S]^{\wedge_\chi}$ from the previous paragraph is $\K^\times$-equivariant (by the construction) and that the $\K^\times$-action on $S$ contracts $S$ to $\chi$.

By definition, $\Walg:=\Walg_\hbar/(\hbar-1)$. This is a filtered algebra that comes with a $Q\times \ZZ/2\ZZ$-action
by automorphisms and an anti-automorphism (or the identity map) $\tau$,
both preserve the filtration. Moreover, similarly to \cite[2.1]{W-prim},
the construction easily implies that there is a $Q$- and $\tau$-equivariant Lie algebra embedding
(a quantum comoment map) $\q\hookrightarrow \Walg$ with image in $\F_{\leqslant 2}\Walg$ such that the adjoint $\q$-action on $\Walg$
coincides with the differential of the $Q$-action. In more detail, we can consider the natural inclusion
$\q\hookrightarrow \U_\hbar^{\wedge_\chi}$, denote it
by $\varphi_\U$.  Also we have a natural $\operatorname{Sp}(V)$-equivariant Lie algebra embedding
$\mathfrak{sp}(V)\hookrightarrow \W_\hbar^{\wedge_0}$
(as a complement of $\K\hbar$ in the degree 2 component of $\W_\hbar$). Composing this with the Lie algebra
homomorphism $\q\rightarrow \mathfrak{sp}(V)$, we get a map $\varphi_\W:\q\rightarrow \W_\hbar^{\wedge_0}$.
As we have seen in \cite[2.1]{W-prim}, under the identification $\U^{\wedge_\chi}_\hbar\cong \W_\hbar^{\wedge_0}
\widehat{\otimes}_{\K[[\hbar]]}\Walg^{\wedge_\chi}_\hbar$, the map $\varphi_\U$ decomposes into the
sum $\varphi_\W+\varphi_\Walg$, with $\varphi_\Walg$ being a $Q$-equivariant map $\q\rightarrow \Walg^{\wedge_\chi}_\hbar$
with image lying in degree 2 component, and, in particular, in $\Walg_\hbar$. The Lie algebra homomorphism $\q\rightarrow \Walg$
is obtained from $\varphi_\Walg$ by taking the quotient by $\hbar-1$. The reason why this map is an embedding is as follows:
the $Q^\circ$-action on $\K[S]$ and hence on $\Walg$ has discrete kernel. Indeed, the kernel of the $Q$-action on $S$
has to act trivially on $f$ and on $\z_\g(e)$ and hence on the whole algebra $\g$.

As in \cite[2.1]{W-prim}, all choices we have made  differ by an automorphism of the form $\exp(\frac{1}{\hbar^2}\ad f)$, where
$f\in \tilde{I}_\chi^3$ and $\frac{1}{\hbar^2}f$ is $\widetilde{Q}$-invariant.
So, as a filtered algebra with a $\ZZ/2\ZZ\times (\ZZ/2\ZZ\ltimes Q)$-action and with a quantum comoment
map $\q\rightarrow \Walg$, the algebra $\Walg$ is independent of the choice of $\iota$ up to an isomorphism.

An important property of $\Walg$ is that its center is naturally identified with the center $\Centr$
of $\U$. This was first proved by Ginzburg and Premet,  see \cite[2.2]{HC} for details.

In fact, in Section \ref{S_Vfun} we will need a somewhat different choice of $\iota$.

\subsection{Functor $\bullet_\dagger$}\label{SS_dagger_functor}
This is a functor from the category of Harish-Chandra $\U$-bimodules to the category of Harish-Chandra
$\Walg$-bimodules introduced in \cite{HC} (a closely related functor was constructed by Ginzburg
in \cite{Ginzburg}). Let us recall the definitions of the categories, first.

By a Harish-Chandra bimodule over $\U$ (relative to $G$), we mean a finitely generated $\U$-bimodule $M$
such that the adjoint action of $\g$, $\ad(x)m=xm-mx$, is locally finite and integrates
to a $G$-action (the last condition is vacuous when the group is semisimple and simply connected).
On such  a bimodule one can introduce a  {\it good filtration}, i.e., a $G$-stable filtration
$\F_i\M$ that is compatible with the algebra filtration $\F_i\U$ and such that the associated
graded $\gr\M$ is a finitely generated $S(\g)$-module. Using this we can define the associated
variety $\VA(\M)$ of $\M$ as the support of $\gr\M$, this is a conical $G$-stable subvariety
in $\g\cong \g^*$ independent of the choice of a good filtration. The category of Harish-Chandra
(HC, for short) bimodules will be denoted by $\HC(\U)$.

By a $Q$-equivariant Harish-Chandra $\Walg$-bimodule we mean a $\Walg$-bimodule $\Nil$ equipped
with a $Q$-action compatible with the $Q$-action on $\Walg$ (in the sense that the structure
map $\Walg\otimes \Nil\otimes \Walg\rightarrow\Nil$ is $Q$-equivariant) and subject to the following conditions:
\begin{itemize}
\item there is a $Q$-stable filtration $\F_i\Nil$ on $\Nil$ that is compatible with the algebra
filtration on $\Walg$,
\item we have $[\F_i\Walg, \F_j\Nil]\subset \F_{i+j-2}\Nil$ for all $i$ and $j$,
\item $\gr\Nil$ is a finitely generated $\K[S]$-module,
\item and the differential of the $Q$-action on $\Nil$ coincides with the the adjoint $\q$-action.
\end{itemize}
The category of $Q$-equivariant HC $\Walg$-bimodules will be denoted by $\HC^Q(\Walg)$.

A functor $\bullet_\dagger:\HC(\U)\rightarrow \HC^Q(\Walg)$ was constructed in \cite[3.3,3.4]{HC} using the same construction
as was used to construct $\Walg$ in the previous subsection. Namely, take $\M\in \HC(\U)$. Form the
Rees bimodule $\M_\hbar$ with respect to some good filtration. We can complete $\M_\hbar$ with respect
to the (left or right, does not matter) $\tilde{I}_\chi$-adic topology. Denote the resulting $\U_\hbar^{\wedge_\chi}$-bimodule
by $\M^{\wedge_\chi}_\hbar$. This bimodule carries a $Q\times \K^\times$-action compatible with a
$Q\times\K^\times$-action on $\U^{\wedge_\chi}_\hbar$. Then one can show that $\M^{\wedge_\chi}_\hbar$
splits into the completed tensor product $\W_\hbar(V)^{\wedge_0}\widehat{\otimes}_{\K[[\hbar]]}\Nil'_\hbar$,
where $\Nil'_\hbar$ is the centralizer of $V$ in $\M^{\wedge_\chi}_\hbar$ and hence is a $\Walg'_\hbar$-bimodule.
Then again the $\K^\times$-finite part $\Nil_\hbar$ of $\Nil'_\hbar$ is dense. We set $\M_{\dagger}:=\Nil_\hbar/(\hbar-1)\Nil_\hbar$.
This can be shown to be  canonically independent of the choice of a good filtration on $\M$. So $\bullet_\dagger$ is a functor.

The functor $\bullet_\dagger: \HC(\U)\rightarrow \HC^Q(\Walg)$ has the following properties, see \cite[Proposition 3.4.1,Theorem 4.4.1]{HC}:
\begin{itemize}
\item[(i)] $\bullet_\dagger$ is exact.
\item[(ii)] With the choice of filtration as above, the sheaf  $\gr \M_\dagger$ on $S$ is the restriction
of the sheaf $\gr \M$ (on $\g^*=\g$) to $S$ (this not explicitly stated in {\it loc.cit.} but is deduced directly from the construction).
\item[(iii)] In particular, consider the full subcategories $\HC_{\partial \Orb}(\U)\subset\HC_{\overline{\Orb}}(\U)\subset \HC(\U)$
consisting of all HC bimodules whose associated varieties are contained in the boundary $\partial \Orb$ and the
closure $\overline{\Orb}$ of $\Orb$. Then $\bullet_\dagger$ annihilates $\HC_{\partial \Orb}(\U)$ and sends
$\HC_{\overline{\Orb}}(\U)$ to finite dimensional $\Walg$-bimodules. In particular, $\bullet_\dagger$ descends
to the quotient category $\HC_{\Orb}(\U):=\HC_{\overline{\Orb}}(\U)/\HC_{\partial \Orb}(\U)$.
\item[(iv)] For $\M\in \HC_{\overline{\Orb}}(\U)$, the dimension of $\M_\dagger$ coincides with the multiplicity
$\mult_{\Orb}\M$ of $\M$ on $\Orb$ (=generic rank of $\gr\M$ on $\Orb$).
\item[(v)] There is a right adjoint functor $\bullet^\dagger:\HC^Q_{fin}(\Walg)\rightarrow \HC_{\overline{\Orb}}(\U)$. Moreover,
its composition with the quotient $\HC_{\overline{\Orb}}(\U)\twoheadrightarrow \HC_{\Orb}(\U)$ is left inverse
to $\bullet_{\dagger}$. In other words, $\bullet_{\dagger}:\HC_{\Orb}(\U)\rightarrow \HC^Q_{fin}(\Walg)$
is an equivalence onto its image.
\item[(vi)] Let $\M\in \HC(\U)$ and $\mathcal{N}$ be a $Q$-stable subbimodule of finite codimension in $\M_{\dagger}$. Then there
is a unique maximal subbimodule $\M'\subset \M$ with the property $\M'_\dagger=\mathcal{N}$. We automatically have
$\M/\M'\in \HC_{\overline{\Orb}}(\U)$.
\item[(vii)] The image of $\HC_{\overline{\Orb}}(\U)$ under $\bullet_\dagger$ is closed
under taking subquotients (this is a direct corollary of (i) and (vi)).
\end{itemize}

\subsection{Classification of finite dimensional irreducible modules}\label{SS_W_classif}
We are  going to recall the classification of finite dimensional irreducible $\Walg$-modules
with integral central characters (this notion makes sense because the centers of $\U$
and $\Walg$ are identified). This classification was obtained in \cite{W_classif}.

We start by  recalling one of  the main results from \cite{HC}.  By the construction of
$\Walg$, the group $Q$ acts on $\Walg$ by algebra automorphisms. This gives rise to
a $Q$-action on the set $\Irr_{fin}(\Walg)$ of isomorphism
classes of finite dimensional irreducible $\Walg$-modules. Clearly, $Z(G)$
acts trivially. Also recall  that the differential of the  $Q$-action on $\Walg$ coincides with the adjoint
action of $\q\subset \Walg$. Therefore $Q^\circ$ acts trivially on $\Irr_{fin}(\Walg)$ and so we get an action of the
component group $A:=Q/(Q^\circ Z(G))$ on that set. The orbit space
$\Irr_{fin}(\Walg)/A$ gets naturally identified with the set $$\Prim_{\Orb}(\U)=\{\text{primitive }\J|\VA(\U/\J)=\overline{\Orb}\},$$
 this is Premet's conjecture, \cite[Conjecture 1.2.2]{HC} that is a corollary of \cite[Theorem 1.2.3]{HC}.
Namely, to an $A$-orbit $N_1,\ldots,N_k$ we, first, assign $\I:=\bigcap_{i=1}^k
\Ann_{\Walg}N_i$. This intersection is a $Q$-stable ideal of finite codimension. Then we can apply property
(vi) to $\M=\U$ (so that $\M_\dagger=\Walg$), and $\Nil:=\I$. The corresponding ideal $\J:=\M'\subset \U$ can be seen
to be primitive, and this is the ideal we need. The identification  $\Irr_{fin}(\Walg)/A\cong
\operatorname{Pr}_{\Orb}(\U)$ preserves the central characters under our identification of the centers of $\U$
and of $\Walg$, this follows from \cite[Theorem 3.3.1]{HC}.

Now let us recall a result from \cite{W_classif} that explains how to compute the $A$-orbit
lying over a primitive ideal $\J$ in the case when $\J$ has integral central character.
Below we use facts recalled in \cite[6.1,6.2]{W_classif}.
The existence of such $\J$ implies that $\Orb$ is special in the sense of Lusztig.
So to $\Orb$ we can assign a two-sided cell, say $\dcell$, that is a subset of $W$.
To $\dcell$ one assigns  a subset $\Irr^{\dcell}(W)$ (called a {\it family}) in the set $\Irr(W)$
of irreducible representations of $W$, where, recall, $W$ denotes the Weyl group of $\g$.

Let $Y^{\Lambda}$ denote the subset of $\Irr_{fin}(\Walg)$ consisting of all modules with integral
central character. The $A$-action on $Y^{\Lambda}$ factors through  a certain quotient
$\bA$ of $A$ introduced by Lusztig in \cite{orange}. To define this quotient consider the
Springer $W\times A$-module $\Spr(\Orb)$. Consider its $W$-submodule $\Spr(\Orb)^{\dcell}$
that is the sum of all irreducible $W$-submodules belonging to $\Irr^{\dcell}(W)$. Of course, this is
also an $A$-submodule. The group $\bA$ is the quotient of $A$  by the kernel of the $A$-action
on $\Spr(\Orb)^{\dcell}$, \cite[13.1.3]{orange}.

Now let us describe the stabilizers. Pick a left cell $c\subset \dcell$ and let $\lambda$ be
a dominant weight compatible with $c$ (i.e., compatible with any $w\in c$, this condition
is independent of the choice of $w$, see \cite[6.2]{W_classif}). We can view $\lambda$ as a point in
$\h/W$ and hence as a central character for $\U$. The set of the left cells in $\dcell$ compatible with
$\lambda$ is in a bijection with the set $\Prim_{\Orb}(\U_\lambda)$ of the primitive ideals $\J$ with central character
$\lambda$ and $\VA(\U/\J)=\overline{\Orb}$: to $\cell$ we assign the ideal $J(w\lambda)$, the annihilator
of the irreducible highest weight module $L(w\lambda)$ with highest weight $w\lambda-\rho$, where
$w\in \cell$.

According to  \cite[Theorem 1.1]{W_classif}, the stabilizer of the orbit over $J(w\lambda), w\in c,$
is the subgroup $H_c\subset \bA$ defined by Lusztig in \cite{Lusztig_subgroups}.
It can be described as follows. Consider the cell module $\mathsf{M}(c)$ associated to $c$. Then $\Hom_W(\mathsf{M}(c),\Spr(\Orb))=
\Hom_W(\mathsf{M}(c),\Spr(\Orb)^{\dcell})$ is an $\bA$-module. It turns out that there is a unique (up to
conjugacy) subgroup $H_c\subset \bA$ such that the $\bA$-modules $\BQ(\bA/H_c)$
and  $\Hom_W(\mathsf{M}(c),\Spr(\Orb))$ are isomorphic. See \cite{W_classif}, Subsections 6.5-6.8, for  explicit computations
of $H_c$ starting from $\mathsf{M}(c)$.

\subsection{Semisimple HC bimodules}\label{SS_HC_semis}
Now fix  a finite set  $\Lambda'$ of dominant weights such that
the pairwise differences of the elements of $\Lambda'$ lie in the root lattice
and such that there is a regular element $\varrho\in \Lambda'$.  Consider
the subcategory $\HC_{\Orb}(\U)^{ss}_{\Lambda'}\subset\HC_{\Orb}(\U)$ consisting of all
semisimple  objects with left and right central characters lying in $\Lambda'$. According to \cite[Theorem 7.4, Remark 7.7]{W_classif}, this category
is isomorphic to the category $\Coh^{\bA}(Y^{\Lambda'}\times Y^{\Lambda'})$
of $\bA$-equivariant sheaves of finite dimensional vector spaces on $Y^{\Lambda'}\times Y^{\Lambda'}$, where $Y^{\Lambda'}$
is the set of finite dimensional irreducible $\Walg$-modules with central character in $\Lambda'$. Irreducible objects in
the latter category are parameterized by the triples $(x,y,\mathcal{V})$, where $x,y\in Y^{\Lambda'}$ and $\mathcal{V}$
is an irreducible $\bA_{(x,y)}$-module, a triple is defined up to an $\bA$-conjugacy. Namely, the support of an
irreducible sheaf is a single orbit and we take $(x,y)$ from this orbit, for $\mathcal{V}$ we take the fiber
of the sheaf at $(x,y)$. So any irreducible Harish-Chandra bimodule in $\HC_{\Orb}(\U)_{\Lambda}$ gets mapped to some triple $(x,y,\mathcal{V})$. We say that this triple corresponds to this bimodule  (or to $(w,\lambda)$
if the bimodule is $\M_{w}(\lambda):=L(\Delta(\varrho),L(w\lambda))$ or just to $w$ if $\lambda=\varrho$; the triple does
not depend on the choice of $\varrho$). \cite[Theorem 1.3.1]{HC}
implies that $x,y$  lie over the left and right annihilators of $\M_w(\lambda)$, the ideals $J(w\lambda),J(w^{-1}\varrho)$, respectively. Moreover, the construction in \cite{W_classif} implies that, being an idempotent object in
$\HC_{\Orb}(\U)^{ss}_{\Lambda'}$, the quotient $\U/J(w\lambda)$ gets mapped to the sheaf supported on
the diagonal of the $\bA$-orbit corresponding to $J(w\lambda)$, whose fiber is the trivial module.
In particular, if $d$ is the Duflo involution in $c_w$, the left cell containing $w$, then the triple corresponding to $d$
has the form $(x,x,\operatorname{triv})$. This is because $\M_d$ coincides with $\U/J(d\varrho)$ in $\HC_{\Orb}(\U)$.
To see that we first recall that, by the definition of $d$ given in \cite[3.3,3.4]{Joseph_I},
$L(d\varrho)$ is the socle in $\Delta(\varrho)/J(w\varrho)\Delta(\varrho)$ and
the GK dimension of $L(d\varrho)$ is bigger than that of $[\Delta(\varrho)/J(w\varrho)\Delta(\varrho)]/L(d\varrho)$.
Under the Bernstein-Gelfand equivalence, $\Delta(\varrho)/J(w\varrho)\Delta(\varrho)$
corresponds precisely to $\U/J(d\varrho)$ and the equality $\M_d=\U/J(d\varrho)$ in $\HC_{\Orb}(\U)$ follows.

A important corollary from \cite{W_classif}  is a formula for the multiplicity
of an irreducible object $\M$ in $\HC_{\Orb}(\U)_{\Lambda}$, see Remark 7.7 and formula (7.1) in {\it loc. cit}.
Namely, let $(x,y,\mathcal{V})$ be a triple corresponding to  $\M$. Then we have
\begin{equation}\label{eq:mult_equality}
\mult_{\Orb}(\M)=d_xd_y\frac{|\bA|}{|\bA_{(x,y)}|}\dim \mathcal{V}.
\end{equation}
Here $d_x,d_y$ are the dimensions of irreducible $\Walg$-modules lying over
the left and the right annihilators of $\M$. This formula will be one of the
crucial tools to relate the Goldie ranks and the dimensions of $\Walg$-irreducibles.

\subsection{Parabolic induction}\label{SS_parab_ind}
Recall the Lusztig-Spaltenstein induction, \cite{LS}. Take a Levi subalgebra $\underline{\g}\subset \g$
and a nilpotent orbit $\underline{\Orb}\subset \underline{\g}$. One can construct a nilpotent orbit $\Orb\subset \g$
(called {\it induced} from $\underline{\Orb}$)
from this pair as follows. Pick a parabolic subalgebra $\p\subset\g$ with Levi subalgebra
$\underline{\g}$. Let $\n$ stand for the maximal nilpotent subalgebra of $\p$.
For $\Orb$ we take a unique dense orbit in $G(\underline{\Orb}+\n)$. It turns
out that $\Orb$ does not depend on the choice of $\p$. The codimension of $\Orb$ in $\g$
coincides with the codimension of $\underline{\Orb}$ in $\underline{\g}$. The intersection
of $\Orb$ with $\underline{\Orb}+\n$ is a single $P$-orbit, see \cite[Theorem 1.3]{LS}.

Let $\underline{\Walg}$ denote the W-algebra
of the pair $(\underline{\g},\underline{\Orb})$. In \cite[Section 6]{Miura} we have constructed
a dimension preserving exact functor $\varsigma:\underline{\Walg}$-$\operatorname{mod}_{fin}\rightarrow
\Walg$-$\operatorname{mod}_{fin}$ between the categories of finite dimensional modules.
This functor depends on the choice of $P$. Namely, see \cite[6.3]{Miura}, there is a completion $\underline{\Walg}'$ of $\underline{\Walg}$ such that any finite dimensional $\underline{\Walg}$-module extends to
$\underline{\Walg}'$ and an embedding $\Xi:\Walg\hookrightarrow \underline{\Walg}'$. The functor
under consideration is just the pull-back from $\underline{\Walg}'$ to $\Walg$. Furthermore,
we can choose $e\in \Orb\cap (\underline{\Orb}+\n)$ in such a way that a reductive part $\underline{Q}$ of the centralizer
$Z_P(e)$ lies in the Levi subgroup $\underline{G}$ of $P$ corresponding to $\underline{\g}$.
The group $\underline{Q}$ acts on $\underline{\Walg}$ by automorphisms, the action extends to $\underline{\Walg}'$,
and the embedding $\Walg\rightarrow \underline{\Walg}'$ is $\underline{Q}$-equivariant (the latter
can be deduced directly from the construction of $\Xi$ in \cite[Theorem 6.3.2]{Miura}).

\subsection{Isomorphisms of completions}\label{SS_iso_comp}
\subsubsection{Setting}\label{SSS_setting}
Let us fix a setting that will be used  until Section \ref{S_Goldie}.

Let $e\in \g$ be a  nilpotent element. We include $e$ into a Levi subalgebra $\g_0$
so that $e$ is an even nilpotent element in $\g_0$. For example, this is always the case when $e$ is distinguished in
$\g_0$, equivalently, $\g_0$ is a minimal Levi subalgebra containing $e$, see, for example, \cite[Theorem 8.2.3]{CM}.
So such $\g_0$ always exists.

Choose an $\sl_2$-triple $(e,h,f)$ in $\g_0$.
Choose  Cartan and Borel subalgebras $\h\subset \b_0\subset \g_0$ in such a way that $h\in \h$ and is a dominant (for $\g_0$)
element there. Let $\g=\bigoplus_{i\in \ZZ}\g(i)$ stand for the eigendecomposition for $h$.

Pick an integral element $\theta\in \z(\g_0)$ such that $\z_\g(\theta)=\g_0$. Consider the eigen-decomposition
$\g=\bigoplus_{i\in \ZZ}\g_i$. Set $\b:=\b_0\oplus \g_{>0}$, where  $\g_{>0}:=\bigoplus_{i>0}\g_i$,
clearly, $\b$ is a Borel subalgebra in $\g$. Further, set $\p:=\g_0(\geqslant 0)\oplus \g_{>0}$. This is a parabolic
subalgebra in $\g$ containing $\b$. Let $P$ denote the corresponding parabolic subgroup of $G$.

Let $\sigma$ be the anti-involution of $\g$ defined as follows: $\sigma|_{\h}=\operatorname{id}, \sigma(e_i)=f_i, \sigma(f_i)=e_i$.
We claim that one can choose $e$ and $f$ (still in the same orbit $\Orb$) in such a way that $h\in \h$ is still dominant for $\g_0$, $\sigma(e)=f$ and
hence $\sigma(f)=e$. We remark that $h$ is  fixed by $\sigma$ because $\sigma$ is the identity on $\h$.
So a result of Antonyan, \cite{Antonyan}, implies that $e$ is conjugate to a $\sigma$-invariant, say $e'$.
The element $e'$ can be included into an $\sl_2$-triple $(e',h',f')$ in $\g_0$ with $\sigma(e')=e', \sigma(f')=f',
\sigma(h')=-h'$. In other words, on the $\sl_2$-subalgebra with standard basis $e',h',f'$ the antiautomorphism $\sigma$
acts as the transposition with respect to the anti-diagonal. It is conjugate to the usual
transposition.  So we can find the $\sl_2$-triple $e'',h'',f''$ in that $\sl_2$-subalgebra with
$\sigma(e'')=f''$. Now we can replace $(e,h,f)$ with $(e'',h'',f'')$. Further, we can
conjugate such $h$ to $\h$ by an element of $\operatorname{Ad}(\g_0^{-\sigma})$.
The Cartan space and the Weyl group of the symmetric pair $(\g_0,\g_0^{-\sigma})$
are just $\h$ and  $W$ (clearly, $\h\subset \g^{\sigma}$ and the claim about the Weyl group can be checked
for $\g=\mathfrak{sl}_2$, where it is straightforward). So we can assume that $h$ is dominant and we are done.
Clearly, $\sigma$ lifts to $G$.

Let $n$ be the image of $\begin{pmatrix}0&i\\i&0\end{pmatrix}$ under the homomorphism
$\operatorname{SL}_2\rightarrow G$ corresponding to the $\sl_2$-triple $(e,h,f)$. The matrix is
symmetric and therefore $\sigma(n)=n$. It follows that $\tau:g\mapsto n\sigma(g)n^{-1}$
is involutive. This is our choice of $\tau$ from now on and until Section \ref{S_suppl}.

Recall that we write $\t$ for the center $\z(\g_0)$ of $\g_0$. Let $T$ stand for the torus in $G$
corresponding to $\mathfrak{t}$ and $R$ be the centralizer of $T$ in $Q$.

In the next two parts we establish certain isomorphisms of various algebras. In the rest of the paper we will
always assume that the algebras are identified as explained below in this subsection.

\subsubsection{Right-handed completions}\label{SSS_compl_right_hand}
Define a subalgebra $\m\subset \g$ by
$\m:=\g_0(<0)\oplus\g_{>0}$ (in \cite{LOCat} this subalgebra
was denoted by $\widetilde{\m}$, while the notation $\m$ was only used
in the case $\g=\g_0$, but we want to simplify the notation here). Let us point out
that $\chi$ is a character of $\m$. Also we consider the shift of $\m$, the  subspace
$\m_\chi:=\{x-\langle\chi,x\rangle, x\in\m\}\subset \g\oplus \K$.

We will need a completion of  $\U$ considered in \cite[Section 5]{LOCat}: $\U^{\wedge}:=\varprojlim_{n\rightarrow +\infty}
\U/\U\m_\chi^n$. This is a topological algebra as explained  in \cite[3.2]{wquant}.

We can decompose $\U^\wedge$ into a completed tensor product as follows.
As we have noticed in \cite[5.1]{LOCat},  $\v:=\m\cap V$ is a lagrangian subspace in $V$
(note that the notation used there was different).
Thanks to the embedding $\q\hookrightarrow\Walg$
we can view $\theta$ as an element of $\Walg$ and consider the eigen-decomposition
$\Walg=\bigoplus_{i\in \ZZ}\Walg_i$. Then we set $\Walg^{\wedge}:=\varprojlim_{n\rightarrow+\infty}\Walg/\Walg \Walg_{>0}^n,
\W^{\wedge}:=\varprojlim_{n\rightarrow +\infty}\W/\W\v^n$.
We have seen in \cite[Section 5]{LOCat} (and, in a bit different setting, in \cite[Sections 3.2,3.3]{wquant}),  the decomposition $\U_\hbar^{\wedge_\chi}=\W_\hbar^{\wedge_\chi}\widehat{\otimes}_{\K[[\hbar]]}\Walg^{\wedge_\chi}_\hbar$ gives rise
to an isomorphism $\U^\wedge\cong \W(\Walg)^{\wedge}$, where
we write $\W(\Walg)^\wedge$ for $\W^\wedge\widehat{\otimes}\Walg^\wedge$, of course, $\W(\Walg)^\wedge$
is the completion of $\W\otimes \Walg$ with respect to the left ideals $\W\otimes \Walg(\v\otimes 1+ 1\otimes \Walg_{>0})^n$.
The isomorphism $\U^\wedge\cong \W(\Walg)^\wedge$ maps the left ideal
$\U^\wedge\m_\chi$ to $\W(\Walg)^{\wedge}(\v\otimes 1+ 1\otimes \Walg_{>0})$.

Let us recall how the isomorphism is constructed, see \cite[Section 5]{LOCat}.
We embed $\K^\times$ into $\K^\times\times Q$ with differential
$(1,-N\theta)$ for $N$ large enough. Then we can consider the subalgebras $$(\U_\hbar^{\wedge_\chi})_{fin},
(\W_\hbar(V)^{\wedge_0}\widehat{\otimes}_{\K[[\hbar]]}\Walg_\hbar^{\wedge_\chi})_{fin}$$ of $\K^\times$-finite
vectors for this copy of $\K^\times$. These algebras are isomorphic because the isomorphism $\U_\hbar^{\wedge_\chi}\cong
\W_\hbar(V)^{\wedge_0}\widehat{\otimes}_{\K[[\hbar]]}\Walg_\hbar^{\wedge_\chi}$ is $\K^\times$-equivariant.
We mod out $\hbar-1$ and get isomorphic algebras \begin{equation}\label{eq:heart}\U^{\heartsuit}:=(\U_\hbar^{\wedge_\chi})_{fin}/(\hbar-1),
\W(\Walg)^{\heartsuit}:=(\W_\hbar(V)^{\wedge_0}\widehat{\otimes}_{\K[[\hbar]]}\Walg_\hbar^{\wedge_\chi})_{fin}/(\hbar-1).\end{equation}
They are embedded into $\U^{\wedge}, \W(\Walg)^{\wedge}$, \cite[Proposition 5.1, Lemma 5.3]{LOCat}. Then the isomorphism $\U^{\heartsuit}\cong \W(\Walg)^{\heartsuit}$ extends by continuity to an isomorphism $\U^\wedge\cong \W(\Walg)^{\wedge}$ that maps
$\U^\wedge \m_{\chi}$ to $\W(\Walg)^{\wedge}(\v\otimes 1+ 1\otimes \Walg_{>0})$, \cite[Lemma 5.3]{LOCat}.

The isomorphism $\U^\wedge\cong \W(\Walg)^\wedge$ induces another isomorphism. Namely, we write $\U^0$
for $U(\g_0)$. Also we set $\Walg_{\geqslant 0}:=\bigoplus_{i\geqslant 0}\Walg_i$ and $\Walg^0:=\Walg_{\geqslant 0}/(\Walg_{\geqslant 0}\cap \Walg\Walg_{>0})=\Walg_0/(\Walg_0\cap \Walg_{<0}\Walg_{>0})$. As we have noticed in \cite[Section 5]{LOCat},
the identification $\U^\wedge\cong \W(\Walg)^{\wedge}$ induces an isomorphism of $\Walg^0$ and  $U(\g_0,e):=[\U^0/\U^0\m^0_\chi]^{M^0}$,
where $\m^0:=\m\cap \g_0$. The algebra $U(\g_0,e)$ (a W-algebra, as defined by Premet in \cite{Premet}) is identified
with the W-algebra for $\g_0$ in our sense via an isomorphism of completions analogous to $\U^\wedge\cong \W(\Walg)^\wedge$
but taken for $\g_0$ and not for $\g$.

The isomorphism $U(\g_0,e)\xrightarrow{\sim} \Walg^0$ is $R$-equivariant  but does not intertwine the quantum comoment maps
$\mathfrak{r}\rightarrow U(\g_0,e),\Walg^0$.
Here the quantum comoment map $\mathfrak{r}\rightarrow U(\g_0,e)$ is obtained as the composition  $\mathfrak{r}\hookrightarrow
\U^0\twoheadrightarrow\U^0/\U^0\m^0_\chi$, while the quantum comoment map $\mathfrak{r}\rightarrow \Walg^0$ is the composition
$\mathfrak{r}\hookrightarrow \Walg\rightarrow \Walg/\Walg\Walg_{>0}$.
Instead the isomorphism induces a shift on $\mathfrak{r}$ by a certain character $\delta$, as in \cite[Remark 5.5]{LOCat},
i.e., it maps $\xi\in \mathfrak{r}\hookrightarrow U(\g_0,e)$ to $\xi-\langle\delta,\xi\rangle$. Our setting is a bit
different from {\it loc. cit.}, as we consider a larger Lie algebra here. So we are going to provide details.

Let us write $\iota_\U,\iota_\Walg,\iota_\W$ for the quantum comoment maps to the corresponding algebras
(this differs a bit from the conventions of Subsection \ref{SSS_slice}) so that under the isomorphism
$\U^\wedge\cong \W(\Walg)^\wedge$, we have $\iota_\U=\iota_\W+\iota_{\Walg}$.
A key observation is that the map $\mathfrak{r}\rightarrow[\W/\W\v]^\v$ induced by $\iota_\W$ is the character
$\delta$ that, by definition, equals a half of the character of $\bigwedge^{top}\v^*$ (see \cite[Remark 5.5]{LOCat}
for a computation). In particular, the restriction of $\delta$ to $\t$ is the same
character as in {\it loc.cit.} (we remark that $\t$ is naturally represented as a direct summand
of $\mathfrak{r}$). So the map $\mathfrak{r}\rightarrow \Walg^\wedge=
[\W(\Walg)^\wedge/ \W(\Walg)^\wedge\v]^\v$ induced by $\iota_\U$ equals $\iota_\Walg+\delta$.
This implies the claim in the previous paragraph.

\subsubsection{Left-handed completions}
Set $\bar{\m}:=\tau(\m)=\g_{<0}\oplus \g_0(<0)$. Define the completion $\wU:=\varprojlim_{n\rightarrow +\infty} \U/\bar{\m}_\chi^n \U$.
Also set $\bar{\v}:=\bar{\m}\cap V$, this is again a lagrangian subspace in $V$. Set $\wW:=\varprojlim_{n\rightarrow +\infty} \W/\bar{\v}^n \W$,
$\wWalg:=\varprojlim_{n\rightarrow +\infty} \Walg/\Walg_{<0}^n \Walg$, where $\Walg_{<0}:=\bigoplus_{i<0}\Walg_i$. Twisting the isomorphism
$\U^\wedge\cong \W(\Walg)^\wedge$ with $\tau$, we get an isomorphism $\wU\cong\wW(\Walg):=\wW\widehat{\otimes}\wWalg$.

Again, below we will need several isomorphisms induced by $\wU\cong \wW(\Walg)$. We can form the analogous completions
$\wU^0$ of $\U^0$, $\wW^0$ of $\W^0:=\W(V^0)$, where $V^0:=\g_0\cap V=[\g_0,f]$. Also we can consider the eigen-spaces
$\wU_i,\wW_i, \wWalg_i$ for the action of $\ad(\theta)$ (in the case of $\wW$ rather of the corresponding one-dimensional
torus). Then we can define the subalgebras $\wU_{\leqslant 0}:=\bigoplus_{i\leqslant 0}\wU_i,
\wW_{\leqslant 0}, \wWalg_{\leqslant 0}$ and their ideals $\wU_{<0},\wW_{<0}, \wWalg_{<0}$ similarly to
$\Walg_{\geqslant 0},\Walg_{>0}$. We claim that $$\wU_{\leqslant 0}/(\wU_{\leqslant 0}\cap \wU_{<0}\wU)=
\wU_0/(\wU_0\cap \wU_{<0}\wU_{>0})$$ is naturally identified with $\wU^0$ and the similar equalities hold for
the other two algebras (in the $\Walg$-case we have, by definition, $\wWalg^0:=\Walg^0$).

We are going to prove the isomorphism in the $\U$-case, the other two cases are similar.
The algebra $\wU$ can be realized ``explicitly'' as follows. Choose a basis $x_1,\ldots,x_a, y_1,\ldots,y_b, $ $ z_1,\ldots,z_c,
w_1,\ldots,w_d$ of $\g_\chi:=\{x-\langle\chi,x\rangle| x\in \g\}$ such that
\begin{itemize} \item $x_1,\ldots,x_a$ are weight vectors for $\theta$ with negative weights,
$y_1,\ldots,y_b, z_1,\ldots,z_c\in \g_0$, while $w_1,\ldots,w_d$ are weight vectors for $\theta$
with positive weights.
\item $y_1,\ldots,y_b$ form a basis in $\m^0_\chi$.
\end{itemize}
Then $\wU$ consists of all infinite sums $\sum_{\alpha,\beta,\gamma,\delta} n_{\alpha\beta\gamma\delta}x^\alpha y^\beta
z^\gamma w^\delta$, where $\alpha=(\alpha_1,\ldots,\alpha_a), x^\alpha:=x_1^{\alpha_1}\ldots x_a^{\alpha_a}$ etc.,
subject to the condition that, for any given $\alpha,\beta$, only finitely many coefficients $n_{\alpha\beta\gamma\delta}$
are nonzero. The product is induced (=extended by continuity) from $\U$.

The quotient $\wU/\wU\wU_{>0}$ consists of the infinite sums of the form $\sum_{\alpha\beta\gamma}n_{\alpha\beta\gamma}
x^\alpha y^\beta z^\gamma$ with the same finiteness condition as above. So $\wU^0$ consists of the sums  $\sum_{\beta\gamma}n_{\beta\gamma}
y^\beta z^\gamma$ (with product induced from $\U$ or, equivalently, $\U^0$) and is naturally identified with $\wU^0$.

The algebra $\wW(\Walg)_0/(\wW(\Walg)_0\cap \wW(\Walg)_{<0}\wW(\Walg)_{>0})$ is naturally identified with $\wW^0(\Walg^0)=
\wW^0\widehat{\otimes}\Walg^0$. It follows that the isomorphism $\wU\cong \wW(\Walg)$ induces an isomorphism
$\wU^0\cong \wW^0(\Walg^0)$. Moreover, it also induces an isomorphism $\wU/\wU\wU_{>0}\xrightarrow{\sim}\wW(\Walg)/\wW(\Walg)\wW(\Walg)_{>0}$
that is linear both over $\wU\cong \wW(\Walg)$ (acting on the left) and over $\wU^0\cong \wW^0(\Walg^0)$ (acting on the right).

Finally, similarly to \ref{SSS_compl_right_hand}, the isomorphism $\wU\cong \wW(\Walg)$ induces an isomorphism of $U(e,\g_0):=[\U^0/\m^0_\chi\U^0]^{M^0}$ and $\Walg^0$. Again, this isomorphism is $R$-equivariant but does not intertwine the embeddings of $\t$. Rather it again induces a shift by $\tau(\delta)$. The elements $\delta$ and $\tau(\delta)$ are different but they agree on $\t$
 (because $\tau$ is the identity on $\mathfrak{t}$) and their difference is the character of $R$ on $\bigwedge^{top}\g_0(<0)^*$.

%\subsubsection{Isomorphisms}
\section{Categories}\label{S_Cats}
\subsection{Parabolic category $\mathcal{O}$}
\subsubsection{Definition}
Recall that $G$ denotes a connected reductive algebraic group and that we have fixed  a parabolic subgroup  $P\subset G$.
Fix a character $\nu$ of $\p$. Let $\tilde{\OCat}^P_\nu$ denote the full subcategory in the category of $(\g,P)$-modules
consisting of all modules $\M$ where the $\nu$-shifted $\p$-action, i.e., $(x,m)\mapsto xm-\nu(x)m$,
is locally finite and integrates to the action of $P$. We remark that the categories $\tilde{\mathcal{O}}^{P}_\nu$
and $\tilde{\OCat}^P_{\nu'}$ are naturally equivalent provided $\nu'-\nu$ is a character of $P$.

Inside $\tilde{\OCat}^P_\nu$ we consider
the full subcategory $\OCat^P_\nu$ of all modules where all weight spaces
(for $\z(\lf)$, where $\lf$ is a Levi subalgebra of $P$, or, equivalently, for a Cartan subalgebra $\mathfrak{h}$)
are finite dimensional and where the center of $\U$ acts with finitely many eigen-characters.
Equivalently, $\OCat^P_\nu$ consists of all finitely generated modules in
$\tilde{\OCat}^P_\nu$. It is known that all modules in $\OCat^P_\nu$ have finite length.

Consider the category $\OCat^L_\nu$ of all finite dimensional $\mathfrak{l}$- and $L$-modules, where the differential
of the $L$-action coincides with the $\nu$-shifted $\mathfrak{l}$-action. We have the induction
functor $\Delta_P:\OCat^L_\nu\rightarrow \OCat^P_\nu, \M^0\mapsto \U\otimes_{U(\p)}\M^0$. For the
irreducible $\lf$-module $L_{\lf}(\lambda)$ with highest weight $\lambda-\rho$, we write $\Delta_P(\lambda)$ for $\Delta_P(L_{\lf}(\lambda))$,
this is, of course, a parabolic Verma module. It has a unique irreducible quotient to be denoted
by $L(\lambda)$.

\subsubsection{Completed version}
We consider a category $\hat{\OCat}^{P}_\nu$ consisting of all topological $\U$- and $P$-modules $\M$ satisfying the
following conditions:
\begin{enumerate}
\item The weights of $\z(\lf)$ in $\M$ are bounded from above in the sense that
there is an element $\theta$ in $\z(\lf)$ such that $\operatorname{ad}\theta$ has only positive integral eigenvalues
on the nilpotent radical of $\p$ and all eigenvalues of $\theta$ on $\M$ are bounded from above.
\item Any $\z(\lf)$-weight space is finite dimensional and the center of $\U$ acts with finitely many eigen-characters.
\item The $\nu$-shifted $\lf$-action on any weight space coincides with the differential of the  $L$-action.
\item $\M$ (considered as a topological $\U$-module) is the direct product of its  $\z(\lf)$-weight subspaces.
\end{enumerate}
We have a completion functor $\M\mapsto \widehat{\M}:\OCat^P_\nu\rightarrow \hat{\OCat}^P_\nu$ that sends $\M=\bigoplus_{\mu} \M_\mu$,
where $\M_\mu$ is a weight space corresponding to $\mu\in \z(\lf)^*$ to $\prod_\mu \M_\mu$.
This functor is an equivalence of categories.

\subsubsection{Right-handed versions}\label{SSS_right_cat}
We will also consider the analogs $\OCat_\nu^{P,r}, \hat{\OCat}_\nu^{P,r}$ of $\OCat_\nu^P, \hat{\OCat}_\nu^P$
consisting of right modules (we impose the  condition that the $\nu$-shifted right action of $\p$ integrates to a right action
of $P$).

\subsubsection{Duality}
We have a contravariant duality functor $\bullet^\vee: \OCat^P_\nu\rightarrow \OCat^{P}_\nu$.
Namely, recall the anti-automorphism
$\sigma:\g\rightarrow \g$, see \ref{SSS_setting}.
In particular, it sends $P$ to the opposite parabolic subgroup $P^-$. For $\M=\bigoplus_{\mu\in \t^*}\M_\mu$,
the restricted dual $\M^{(*)}:=\bigoplus_{\mu}\M_\mu^*\subset \M_\mu$ is a right $\U$-module that lies in $\OCat_{-\nu}^{P^-,r}$.
The twist $\M^\vee:=\,^\sigma \!\M^{(*)}$ is therefore again an object of $\OCat^P_\nu$. It is known that $L(\lambda)^\vee\cong
L(\lambda)$, while $\Delta_P(\lambda)^\vee$ is a costandard (=dual Verma) module $\nabla_P(\lambda)$.

\subsubsection{Bernstein-Gelfand equivalence}\label{SSS_BG_equiv}
Here we are going to recall the classical Bernstein-Gelfand equivalence relating the sum of suitable blocks of the full  BGG category
$\OCat$ to a certain category $\HC(\U)^\varrho$ of Harish-Chandra bimodules. Then we introduce
its parabolic analog. This analog should be known but we did not find any reference.

Let $\HC(\U)^\varrho$ denote the category of all HC bimodules with right central character $\varrho$.
Assume that $\varrho$ is strictly dominant meaning that $\langle\varrho,\alpha^\vee\rangle\not\in \ZZ_{\leqslant 0}$
for any positive root $\alpha$. Then the functor $X\mapsto X\otimes_\U\Delta(\varrho)$ is an equivalence
$\HC(\U)^\varrho\rightarrow\OCat_\varrho$, see \cite[5.9]{BG}. The quasi-inverse equivalence is given
by $\M\mapsto L(\Delta(\varrho), \M)$, where $L(\bullet,\bullet)$ denotes the space of $\g$-finite maps.

Now let us proceed to a parabolic analog of this. Suppose we are given a parabolic category $\OCat^P_\nu$.
Adding a suitable character of $P$ to $\nu$, we may assume that $\nu+\rho$ is strictly dominant
(we remark that $\nu$ is $0$ on the roots of $\lf$). Let $\J_{P,\nu}$ denote the annihilator of
$\Delta_P(\nu+\rho)$ in $\U$. Consider the subcategory $\HC(\U)^{\J_{P,\nu}}$
of $\HC(\U)^{\nu+\rho}$ consisting of all bimodules annihilated by $\J_{P,\nu}$ on the right.
Since $\Delta_P(\nu+\rho)=\Delta(\nu+\rho)/\J_{P,\nu}\Delta(\nu+\rho)$,
this functor can be written as $X\mapsto X\otimes_\U \Delta_P(\nu+\rho)$.
So the Bernstein-Gelfand equivalence restricts to a functor $\HC(\U)^{\J_{P,\nu}}\rightarrow \OCat^P_\nu$.

For reader's convenience, let us recall the proof of the equality $\Delta_P(\nu+\rho)=\Delta(\nu+\rho)/\J_{P,\nu}\Delta(\nu+\rho)$.
First, we have a natural epimorphism $\Delta(\nu+\rho)\twoheadrightarrow \Delta_P(\nu+\rho)$
that factors through $\Delta(\nu+\rho)/\J_{P,\nu}\Delta(\nu+\rho)$. The latter corresponds to
$\U/\J_{P,\nu}$ under the Bernstein-Gelfand equivalence. The ideal $\J_{P,\nu}$ is primitive.
Indeed, $\Delta_P(\nu+\rho)$ is isomorphic to $U(\g_{<0})$ as a $U(\g_{<0})$-module. So its GK
multiplicity is $1$ and hence the socle of $\Delta_P(\nu+\rho)$ is simple. The latter proves that
$\J_{P,\nu}$ is primitive. Further, any proper quotient of $\U/\J_{P,\nu}$ has GK
dimension smaller than that of $\U/\J_{P,\nu}$ (equal to $\dim\g-\dim \lf$).
It follows that $\Delta(\nu+\rho)/\J_{P,\nu}\Delta(\nu+\rho)$ has simple socle and the quotient by this socle
has GK dimension less then $\frac{1}{2}(\dim\g-\dim\lf)$. Since the GK dimension of $\Delta_P(\nu+\rho)$
equals $\frac{1}{2}(\dim\g-\dim\lf)$, the required equality follows.

Since $\Delta_P(\nu+\rho)$ is the largest quotient of $\Delta(\nu+\rho)$ lying in $\OCat^P_\nu$, we see
that $\M\mapsto L(\Delta(\nu+\rho),\M)=L(\Delta_P(\nu+\rho),\M):\mathcal{O}^P_\nu\rightarrow \operatorname{HC}(\U)^{\J_{P,\nu}}$
is a quasi-inverse functor to $X\mapsto X\otimes_\U \Delta_P(\nu+\rho)$.

Summarizing, the functor $X\mapsto X\otimes_{\U}\Delta_P(\nu+\rho)$ defines an equivalence
$\HC(\U)^{\J_{P,\nu}}\rightarrow \OCat^P_\nu$ with quasi-inverse equivalence $\M\mapsto
L(\Delta_P(\nu+\rho),\M)$.

Now suppose that $\varrho$ is integral dominant, $w$ is compatible with $\varrho$, $L(w\varrho)$ lies $\OCat^P$
and has Gelfand-Kirillov dimension $\frac{1}{2}(\dim\g-\dim \mathfrak{l})$.
Then $w$ lies in the right cell equal to $c^{-1}$, where $c$ is the left
cell corresponding to the primitive ideal $\J_{P,\nu}$. This is a direct corollary of the parabolic Bernstein-Gelfand
equivalence. Conversely, if $w\in c^{-1}$ is compatible with $\varrho$, we see that $L(w\varrho)\in \OCat^P$
and that $L(w\varrho)$ has GK dimension $\frac{1}{2}(\dim \g-\dim\lf)$.

\subsection{Category $\mathcal{O}$ for $\Walg$ and Whittaker modules I: non-completed version}\label{SS_Cats_noncompl}
\subsubsection{Categories $\tilde{\OCat}^\theta(\g,e), \OCat^\theta(\g,e)$}
We will define the full categories $\OCat^\theta(\g,e)_\nu\subset \tilde{\OCat}^\theta(\g,e)_\nu$
in the category of left $\Walg$-modules. Namely, recall that $\q$ and hence $\t=\z(\g_0)\subset \q$
are naturally included into $\Walg$.

Let $\tilde{\OCat}^\theta(\g,e)$ stand for the category of all $\Walg$-modules $\Nil$ such that
\begin{enumerate}
\item $\t$ acts on $\Nil$ diaglonalizably and with  eigenvalues lying in $\mathfrak{X}(P)+\nu+\delta$.
\item The collection of integers $\langle\mu-\nu-\delta,\theta\rangle$, where $\mu$ is a $\t$-weight
of $\Nil$, is bounded from above.
\end{enumerate}

Below we write $\Nil_\mu$ for the weight space with weight $\mu+\delta$.

By definition, the subcategory $\OCat^\theta(\g,e)_\nu\subset \tilde{\OCat}^\theta(\g,e)_\nu$ consists of all modules with
finite dimensional weight spaces and with finitely many eigen-characters for the action of the center of $\Walg$. Equivalently,
$\OCat^\theta(\g,e)_\nu\subset \tilde{\OCat}^\theta(\g,e)_\nu$ consists of all finitely generated modules.
This is, basically, the category that appeared in \cite[4.4]{BGK} (they considered
the case when $\theta$ is generic in $\q$ but the general case is completely
analogous). In particular, every object in $\OCat^\theta(\g,e)_\nu$ has finite length.

To an object $\Nil\in \OCat^\theta(\g,e)_\nu$ we can assign its formal character $\ch(\Nil):=\sum_{\mu}(\dim \Nil_\mu) e^{\mu}$,
where the summation is taken over $\mu\in \t^*$.

\subsubsection{Verma modules and their characters}
Recall (see \ref{SSS_compl_right_hand}) that $\Walg^0$  is identified with the W-algebra $U(\g_0,e)$ for $\g_0$.
The category $\OCat^\theta(\g_0,e)_\nu$ is just the category of finite dimensional
$U(\g_0,e)$-modules, where $\t$ acts diagonalizably with  weights in $\mathfrak{X}(P)+\nu$.

We have the induction functor (to be called a {\it Verma functor})
$\Delta^\theta_{\Walg}: \tilde{\OCat}^\theta(\g_0,e)_\nu\rightarrow \tilde{\OCat}^\theta(\g,e)_\nu$
that maps $\Nil^0\in \tilde{\OCat}^\theta(\g_0,e)_\nu$ to $\Walg\otimes_{\Walg_{\geqslant 0}}\Nil^0$,
where we view $\Nil^0$ as a $\Walg_{\geqslant 0}$-module via the  projection $\Walg_{\geqslant 0}\twoheadrightarrow
\Walg^0$. Of course, $\Delta^\theta_{\Walg}$ maps $\OCat^\theta(\g_0,e)_\nu$ to $\OCat^\theta(\g,e)_\nu$.

The functor $\Delta^\theta_{\Walg}$ is right exact. It has a right adjoint functor that maps $\Nil$
to the annihilator $\Nil^{\Walg_{>0}}$ of $\Walg_{>0}$. It turns out that the functor
$\Delta^\theta_{\Walg}$ is exact and, moreover, one can compute the character of $\Delta_{\Walg}^\theta(\Nil^0)$.

Namely, consider the eigen-decomposition $\z_\g(e)=\bigoplus_{i\in \ZZ}\z_\g(e)_i$ with respect to $\ad\theta$ and set
$\z_\g(e)_{<0}=\bigoplus_{i<0}\z_\g(e)_i$. Pick a basis $f_1,\ldots,f_k$ of $\z_\g(e)_{<0}$
consisting of weight vectors for $\t$. Recall that $\gr\Walg=\K[S]=S(\z_\g(e))$.
Lift the elements $f_1,\ldots,f_k$ to $\t$-weight vectors $\tilde{f}_1,\ldots,\tilde{f}_k$
in $\Walg$.

The following is a straightforward generalization of assertion (1) of \cite[Theorem 4.5]{BGK}.

\begin{Prop}\label{Prop:Verm_basis}
Let $v_1,\ldots,v_m$ be a basis in $\Nil^0$. The elements $\tilde{f}_1^{n_1}\ldots \tilde{f}_k^{n_k} v_i$, where
$i=1,\ldots,m$,  and $n_j\in  \ZZ_{\geqslant 0}$, form a basis in $\Delta^\theta_{\Walg}(\Nil^0)$.
\end{Prop}

\begin{Cor}\label{Cor:Verm_char}
Suppose that $\t$ acts on $\Nil^0$ with a single weight, say $\mu_0$, and let $\mu_1,\ldots,\mu_k$
be the weights of $\tilde{f}_1,\ldots,\tilde{f}_k$, respectively. Then $$\ch\Delta_{\Walg}^\theta(\Nil^0)=e^{\mu_0}\dim \Nil^0\prod_{i=1}^k (1-e^{\mu_i})^{-1}.$$
\end{Cor}

\begin{Cor}\label{Cor:Verm_exact}
The functor $\Delta_{\Walg}^\theta$ is exact.
\end{Cor}

Also we remark that, for $\Nil^0\in \OCat^\theta(\g_0,e)_\nu$, $\Nil^0$ is naturally embedded into  $\Delta^\theta_{\Walg}(\Nil^0)$.
There is the maximal submodule of $\Delta^\theta_{\Walg}(\Nil^0)$ that does not intersect
 $\Nil^0$, the quotient will be denoted by $L^\theta_{\Walg}(\Nil^0)$.
The modules $L^\theta_{\Walg}(\Nil^0)$, for simple $\Nil^0$, form a complete list of simple objects in $\OCat^\theta(\g,e)_\nu$.
Since any object in $\OCat^\theta(\g,e)_\nu$ has finite length, we see that $\OCat^\theta(\g,e)_\nu$
coincides with the Serre subcategory of $\tilde{\OCat}^\theta(\g,e)_\nu$ generated by $\Delta^\theta_\Walg(\Nil^0)$.

\subsubsection{Categories $\tilde{\Wh}^\theta(\g,e),\Wh^\theta(\g,e)$}
Recall the subalgebra $\m=\g_{>0}\oplus \g_0(<0)$ and its character $\chi$. Consider the full
subcategory $\tilde{\Wh}^\theta(\g,e)_\nu$ in the category of left $\U$-modules consisting of all
modules $\M$ such that
\begin{enumerate}
\item The shift $\m_\chi=\{\xi-\langle \xi,\chi\rangle, \xi\in \m\}$ acts locally nilpotently on $\M$.
\item $\t$ acts diagonalizably on $\M$ with weights in $\mathfrak{X}(P)+\nu$.
\end{enumerate}

One can define an analog of a Verma module in $\tilde{\Wh}^\theta(\g,e)$. Namely, take
$\Nil^0\in \tilde{\OCat}^\theta(\g_0,e)_\nu$. Then we have Skryabin's equivalence (see  \ref{SSS_CatO_equi} below)
$\Sk_0: \tilde{\OCat}^\theta(\g_0,e)_\nu\xrightarrow{\sim} \tilde{\Wh}^\theta(\g_0,e)_\nu$. We set $\Delta_\U^\theta(\Nil^0):=\U\otimes_{U(\g_{\geqslant 0})}\Sk_0(\Nil^0)$. The functor
$\Delta^\theta_\U$ admits a right adjoint, the functor $\M\mapsto \M^{\m_\chi}$.

By definition, for $\Wh^\theta(\g_0,e)_\nu$ we take the Serre subcategory of $\tilde{\Wh}^\theta(\g_0,e)_\nu$
generated by the modules $\Delta^\theta_{\U}(\Nil^0)$ with $\Nil^0\in \OCat^\theta(\g_0,e)_\nu$.

\subsubsection{Equivalences}\label{SSS_CatO_equi}
In \cite{LOCat} we have produced an equivalence $\Kfun: \tilde{\Wh}^\theta(\g,e)_\nu\rightarrow \tilde{\OCat}^\theta(\g,e)_\nu$.
In the case when $\g=\g_0$, the equivalence $\Kfun$ becomes an equivalence introduced by Skryabin in an appendix
to \cite{Premet}.

To construct $\Kfun$, recall the isomorphism $\U^\wedge\cong \W(\Walg)^\wedge$ from \ref{SSS_compl_right_hand}.
The category $\tilde{\Wh}^{\theta}(\g,e)_\nu$
is nothing else but the category of topological $\U^\wedge$-modules with respect to the discrete topology,
where $\t$ acts diagonalizably with  weights in $\mathfrak{X}(P)+\nu$. In particular, we can view $\M\in \tilde{\Wh}^\theta(\g,e)_\nu$
as a module over $\W(\Walg)^{\wedge}$. Then we define $\Kfun(\M)$ as $\M^{\v}$, where the lagrangian subspace
$\v\subset V$ was introduced in \ref{SSS_compl_right_hand}. A quasi-inverse functor
is given by $\Nil\mapsto \K[\v]\otimes \Nil$. Here $\W$ acts on $\K[\v]$ via the identification $\W\cong D(\v)$, or equivalently,
$\K[\v]=\W/\W\v$.

According to  \cite[Theorem 4.1]{LOCat}, the functor $\Kfun$ intertwines the functors
$\Delta^\theta_\Walg,\Delta^\theta_\U$.
It follows that $\Kfun$ induces an equivalence of $\Wh^\theta(\g,e)_\nu$ and $\OCat^\theta(\g,e)_\nu$
because both these subcategories are defined  in terms of  $\Delta^\theta_\bullet$.

\subsubsection{Equivariant version}
Set $R:=Q\cap G_0$, this is a reductive subgroup in $G_0$ (and a maximal reductive subgroup
of $Q\cap P$).
We can consider the $R$-equivariant versions of the categories under consideration. For example,
by an $R$-equivairiant object in $\OCat^\theta(\g,e)_\nu$ we mean a module $\Nil\in \OCat^\theta(\g,e)$
equipped with an action of $R$ that makes the structure map $\Walg\times \Nil\rightarrow \Nil$
into an $R$-equivariant map and such that the differential of the $R$-action
coincides with the action of $\mathfrak{r}\subset \Walg$
(shifted by $\nu+\delta$; for Whittaker categories we just consider a shift by $\nu$).
The $R$-equivariant categories will be denoted by $\OCat^\theta(\g,e)^R_\nu$, etc.
Since the isomorphism  $\U^\wedge\cong \W(\Walg)^\wedge$ is $R$-equivariant, we see that
$\Kfun$ upgrades to an equivalence of equivariant categories. Let us also mention that the Verma
module functors are lifted to the equivariant categories, i.e., for, say, $\Nil^0\in \OCat^\theta(\g_0,e)^R_\nu$,
the Verma module $\Delta^\theta_{\Walg}(\Nil^0)$ has a natural $R$-equivariant structure.

\subsubsection{Duality for $\OCat^\theta(\g,e)$}\label{SSS_W_duality}
Here we are going to define a contravariant involutive equivalence $\OCat^\theta(\g,e)_\nu\rightarrow \OCat^\theta(\g,e)_\nu$.
We are going to use conventions of \ref{SSS_setting}. In particular, $\t\subset \h$.

To $\Nil\in \OCat^\theta(\g,e)_\nu$ we can assign its restricted dual $\Nil^{(*)}=\bigoplus_{\mu\in \t^*}\Nil^*_\mu
\subset \Nil^*$. By the construction, the anti-automorphism $\tau$ is the identity on
$\t$. The twist $\Nil^\vee:=\,{^\tau}\!\Nil^{(*)}$ is therefore an object in $\OCat^\theta(\g,e)_\nu$. Since
$\tau^2=1$, we see that  $\bullet^\vee$ is involutive.

Also it follows directly from the construction that $\ch(\Nil)=\ch(\Nil^\vee)$.

A pairing $\Nil_1\times \Nil_2\rightarrow \K$ is called {\it contravariant} (or $\tau$-contravariant)
if $\langle\tau(a) n_1,n_2\rangle=\langle n_1,an_2\rangle$ for all $a\in \Walg,
n_i\in \Nil_i, i=1,2$. For example, there is a natural contravariant pairing
$\Nil^\vee\times \Nil\rightarrow \K$.

The following lemma characterizes the module $\Nil^\vee$ up to an isomorphism.

\begin{Lem}\label{Lem:naive_dual_char}
Let $\Nil\in \OCat^\theta(\g,e)_\nu, \Nil_1\in \tilde{\OCat}^\theta(\g,e)_\nu$.
Suppose that there is a contravariant pairing $\langle\cdot,\cdot\rangle: \Nil\times \Nil_1\rightarrow \K$
with zero left and right kernels. Then there is a unique isomorphism $\Nil^\vee\xrightarrow{\sim} \Nil_1$
that intertwines the pairing $\Nil_1\times \Nil\rightarrow \K$ with a natural pairing
$\Nil^\vee\times \Nil\rightarrow \K$.
\end{Lem}

The proof is straightforward.

Let us finish by noting that $\bullet^\vee$ naturally upgrades to an equivalence of $R$-equivariant categories.

\subsection{Category $\mathcal{O}$ for $\Walg$ and Whittaker modules II: completed version}\label{SS_cats_compl}
\subsubsection{Categories $\widehat{\OCat}^\theta(\g,e), \hat{\OCat}^\theta(\g,e)$}
By definition, the category $\widehat{\OCat}^\theta(\g,e)_\nu$ consists
of all $\Walg$-modules $\M$ satisfying the following conditions:
\begin{enumerate}
\item $\M$ is complete and separated with respect to the $\Walg_{<0}$-adic topology.
\item The eigenvalues of $\t$ on $\M$ are in $\mathfrak{X}(P)+\nu+\delta$ and are bounded from above
(in the sense that $\langle\theta,\mu-\nu-\delta\rangle$ is uniformly bounded for all weights $\mu$
of $\M$).
\item $\M$ (considered as a topological $\Walg$-module with respect to the $\Walg_{<0}$-adic topology)
is the direct product of its  $\t$-weight subspaces (where the latter is considered with the
direct product topology).
\end{enumerate}
Inside $\widehat{\OCat}^\theta(\g,e)_\nu$ we consider the full subcategory $\hat{\OCat}^\theta(\g,e)_\nu$
consisting of all modules with finite dimensional weight spaces and with finitely many eigen-characters for the
action of the center of $\Walg$.
We again have a completion functor $\Nil\mapsto \widehat{\Nil}:\OCat^\theta(\g,e)_\nu\rightarrow \hat{\OCat}^\theta(\g,e)_\nu$
that sends $\Nil=\bigoplus_{\mu} \Nil_\mu$, to $\prod_\mu \Nil_\mu$. The only claim that one needs to check
in order to verify $\widehat{\Nil}\in \hat{\OCat}^\theta(\g,e)_\nu$ is that the $\Walg_{<0}$-adic topology on $\widehat{\Nil}$
is complete and separated. But this is straightforward from the decomposition into the product of weight spaces.

\begin{Lem}\label{Lem:OCatW_compl}
The functor $\widehat{\bullet}$ coincides with the functor of $\Walg_{<0}$-adic completion,
$\Nil\mapsto \varprojlim_{n\rightarrow +\infty}\Nil/\Walg_{<0}^n \Nil$ and is an equivalence of categories.
\end{Lem}
\begin{proof}
Let us check the claim about an equivalence. For $\Nil'\in \hat{\OCat}^{\theta}(\g,e)_\nu$,
let $\Nil'_{fin}$ denote the subspace of $\t$-finite elements. In other words,
$\Nil'_{fin}=\bigoplus_{\mu}\Nil'_\mu$. Clearly, $\Nil'_{fin}$
is an object of $\OCat^\theta(\g,e)_\nu$ and the functor $\bullet_{fin}$ is quasi-inverse
to $\widehat{\bullet}$.

Now let us check that $\widehat{\bullet}$ coincides with the $\Walg_{<0}$-adic completion functor. Clearly, $\widehat{\Nil}$
is complete in the $\Walg_{<0}$-adic topology so we have a natural map $\varprojlim_{n\rightarrow +\infty}\Nil/\Walg_{<0}^n \Nil
\rightarrow \widehat{\Nil}$. But $\Nil$ is generated by finitely many weight vectors as a $\Walg_{\leqslant 0}$-module.
This implies that the homomorphism is actually an isomorphism.
\end{proof}

We have a Verma functor  $\widehat{\Delta}^\theta_\Walg:\tilde{\OCat}^\theta(\g_0,e)_\nu\rightarrow \widehat{\OCat}^{\theta}(\g,e)_\nu$ given by $\Nil^0\mapsto \widehat{\Delta^\theta_{\Walg}(\Nil^0)}$.
On the other hand, (this follows, for example, from Lemma \ref{Lem:OCatW_compl})
$$\widehat{\Delta}^\theta_{\Walg}(\Nil^0)=[\wWalg/\wWalg\wWalg_{>0}]\widehat{\otimes}_{\Walg^0}\Nil^0.$$
Thanks to the equivalence $\hat{\OCat}^\theta(\g,e)_\nu\cong
\OCat^\theta(\g,e)_\nu$ and the claim that the latter is generated by the Verma modules,
we see that $\hat{\OCat}^\theta(\g,e)_\nu$ coincides with the Serre
subcategory of $\widehat{\OCat}^\theta(\g,e)_\nu$ generated by $\widehat{\Delta}^\theta_{\Walg}(\Nil^0)$
with $\Nil^0\in \OCat^\theta(\g_0,e)_\nu$. We will write $\hat{\Delta}^\theta_{\Walg}(\Nil^0)$ for
$\widehat{\Delta}^\theta_{\Walg}(\Nil^0)$ in the case when $\Nil^0\in \OCat^\theta(\g_0,e)_\nu$.

\subsubsection{Category $\widehat{\Wh}^\theta(\g,e)$ and equivalence $\widehat{\Kfun}$}
Now we are going to define a category $\widehat{\Wh}^{\theta}(\g,e)_\nu$ and show that it
is equivalent to $\widehat{\OCat}^\theta(\g,e)_\nu$.

Recall the subalgebra $\bar{\m}:=\g_{<0}\oplus \g_0(<0)=\tau(\m)$. Let $\widehat{\Wh}^{\theta}(\g,e)_\nu$
consist of all $\wU$-modules $\M$ such that
\begin{enumerate}
\item The $\bar{\m}_{\chi}$-adic  topology on $\M$ is separated and complete.
\item The weights of $\t$ on $\M$ are integral after the $\nu$-shift and are bounded from above.
\item As a topological module, $\M$ is the direct product of its $\t$-weight spaces.
\end{enumerate}

We are going to construct an equivalence $\widehat{\Kfun}:\widehat{\Wh}^\theta(\g,e)_\nu\rightarrow \widehat{\OCat}^\theta(\g,e)_\nu$.
This functor will be produced as a restriction of an equivalence between the categories $\wU$-$\operatorname{csMod}$ of $\wU$-modules
that are complete and separated in the $\bar{\m}_\chi$-adic topology and the category $\wWalg$-$\operatorname{csMod}$ of
$\wWalg$-modules that are complete and separated in the $\Walg_{<0}$-adic topology.

Recall the isomorphism $\wU\cong \wW(\Walg)$. Under this isomorphism,
an equivalence $\widehat{\Kfun}:\wU\operatorname{-csMod}\rightarrow \wWalg$-$\operatorname{csMod}$
is given by taking the $\bar{\v}$-coinvariants.
A quasi-inverse equivalence looks as follows. Choose a lagrangian subspace $\bar{\v}^*\subset V$
complimentary to $\bar{\vf}$. Then $\K[[\bar{\v}^*]]=\wW/\wW\bar{\v}^*$ is a $\wW$-module.
A quasi-inverse equivalence sends $\Nil$ to $\K[[\bar{\v}^*]]\widehat{\otimes}\Nil$.

The claim that the two functors are quasi-inverse to each other reduces to the following lemma.

\begin{Lem}\label{Lem:compl_mod}
Let $U$ be a finite dimensional vector space and $M$ be a module over the algebra
$D(U)$ of differential operators on $U$. Suppose that $M$ is complete and separated
with respect to the $U^*$-adic topology. Then $M=\K[[U]]\widehat{\otimes}M_0$,  where
$M_0$ is a complete separated topological vector space  ($D(U)$ acts on the first
factor).
\end{Lem}
The proof is very similar to  (and is a slight generalization
of) the proof of \cite[Lemma 5.14]{W_classif} but we will provide it for reader's convenience.
\begin{proof}
By induction, compare with {\it loc.cit.}, we reduce the proof to the case when $\dim U=1$.
Let $p$ be a basis vector in $U$ and $q\in U^*$ be the dual basis vector so that
$[p,q]=1$. Set $r:=\sum_{i=0}^{\infty} \frac{(-1)^i}{i!}q^i p^i$, this is a well-defined
element of the completion $\,^\wedge\!D(U)$. Moreover,  we have $pr=0, rq=0$ and $\sum_{i=0}^{+\infty}\frac{1}{i!} q^i r p^i=1$
in $\,^\wedge\!D(U)$. It follows that $m=\sum_{i=0}^{+\infty} \frac{1}{i!}q^i r p^i m$ for every $m\in M$.
But this just says that $M= \K[[q]]\widehat{\otimes} r(M)$ and $r(M)$ coincides with the
annihilator of $p$ in $M$ (and is naturally isomorphic to $M/qM$).
\end{proof}

Now let us show that $\widehat{\Kfun}$ restricts to an equivalence $\widehat{\Wh}^\theta(\g,e)\xrightarrow{\sim} \widehat{\OCat}^\theta(\g,e)$.
The functor $\widehat{\Kfun}^{-1}$ is essentially just taking
the tensor product with $\K[[\bar{\v}^*]]$. All eigenvalues of $\theta$ on $\bar{\v}$ are non-positive integers
 by the construction of $\bar{\v}$. So the eigen-values of $\theta$ on $\K[[\bar{\v}^*]]$ are non-positive integers as well.
It follows that properties (1)-(3) in the definitions of the categories are equivalent for $\Nil$
and   $\K[[\bar{\v}^*]]\widehat{\otimes}\Nil$. This shows that $\widehat{\Kfun}^{-1}$
is an equivalence between $\widehat{\OCat}^\theta(\g,e)$ and $\widehat{\Wh}^\theta(\g,e)$.

\subsubsection{Verma functor for $\widehat{\Wh}^\theta(\g,e)$ and definition of $\hat{\Wh}^\theta(\g,e)$}
Once again, we have a Verma functor $\widehat{\Delta}^\theta_\U: \tilde{\OCat}^\theta(\g_0,e)_\nu\rightarrow
\widehat{\Wh}^\theta(\g,e)_\nu$:
$$\widehat{\Delta}^\theta_{\U}(\Nil^0):=\wU/\wU\wU_{>0}\widehat{\otimes}_{\wU^0}\widehat{\Sk}_0(\Nil^0).$$
Here $\widehat{\Sk}_0$ is the $\g_0$-counterpart of  the equivalence $\widehat{\Kfun}^{-1}$.
It is straightforward to check that $\widehat{\Delta}_\U^\theta(\Nil^0)$ is indeed an object of
$\widehat{\Wh}^\theta(\g,e)_\nu$.

We write $\hat{\Delta}^\theta_{\U}(\Nil^0)$ instead of $\widehat{\Delta}^\theta_{\U}(\Nil^0)$
when $\Nil^0\in \OCat^\theta(\g_0,e)_\nu$.
We define the full subcategory $\hat{\Wh}^\theta(\g,e)_\nu\subset \widehat{\Wh}^\theta(\g,e)_\nu$
as the Serre span of $\hat{\Delta}^\theta_\U(\Nil^0)$ with $\Nil^0\in \OCat^\theta(\g_0,e)_\nu$.

We claim that $\widehat{\Kfun}^{-1}$ (and hence $\widehat{\Kfun}$) intertwines the functors $\widehat{\Delta}^\theta_\U$
and $\widehat{\Delta}^\theta_{\Walg}$. We have the following equality of modules over $\wW(\Walg)$,
\begin{align*}
\widehat{\Delta}^\theta_\U(\Nil^0)=&[\wW(\Walg)/\wW(\Walg)\wW(\Walg)_{>0}]\widehat{\otimes}_{\wW^0(\Walg^0)}(\K[[\bar{\v}^{0*}]]\widehat{\otimes}\Nil^0)=\\
&\left(\left[\wW/\wW\wW_{>0}\right]\widehat{\otimes}_{\wW^0}\K[[\bar{\v}^{0*}]]\right)\widehat{\otimes} \left(\left[\wWalg/\wWalg\wWalg_{>0}\right]\widehat{\otimes}_{\Walg^0}\Nil^0\right)=\\
&\K[[\bar{\v}^*]]\widehat{\otimes}\widehat{\Delta}^\theta_\Walg(\Nil^0).
\end{align*}
But this equality precisely means that $\widehat{\Kfun}^{-1}(\widehat{\Delta}^\theta_{\Walg}(\Nil^0))=\widehat{\Delta}^\theta_\U(\Nil^0)$.

In particular, it follows that $\widehat{\Kfun}$ restricts to an equivalence of $\hat{\Wh}^\theta(\g,e)_\nu$
and $\hat{\OCat}^\theta(\g,e)_\nu$, we write $\hat{\Kfun}$ for the restriction.

We remark that we do not know any direct equivalence between $\hat{\Wh}^\theta(\g,e)_\nu$
and $\Wh^\theta(\g,e)_\nu$.

\subsubsection{Equivariant versions}
Again, we can consider the equivariant versions $\hat{\OCat}^\theta(\g,e)_\nu^R,\hat{\Wh}^\theta(\g,e)_\nu^R$
of $\hat{\OCat}^\theta(\g,e)_\nu,\hat{\Wh}^\theta(\g,e)_\nu$  and can upgrade $\hat{\Kfun}$ to an equivalence of these categories.
We remark that in the completed setting the differential of an $R$-action still makes sense so the definition of an $R$-equivariant
object carries over to this setting.

\section{Generalized Soergel functor}\label{S_Vfun}
\subsection{Functor $\bullet_{\dagger,e}$}\label{SS_dagger_e}
\subsubsection{Category $\OCat^K_\nu$}\label{SSS_Cat_OK_def}
Recall that we have fixed an $\sl_2$-triple $e,h,f\in \g$.
 Next, pick a connected algebraic subgroup $K\subset G$  such that
\begin{itemize}
\item $\mathfrak{u}:=\k\cap V$ is a lagrangian subspace in $V:=[\g,f]$ (here $\mathfrak{k}$ denotes the Lie algebra of $K$),
\item $\chi$ vanishes on $\k$
\item and $h\in\k$.
\end{itemize}
Let $R$ denote a maximal reductive subgroup of
the intersection $Q\cap K$.

Let us provide an example of this situation  that is of most interest for us.

Let $\theta, \g=\bigoplus_{i\in \ZZ}\g_i,\p$ be as in \ref{SSS_setting}. Then we set
$K:=P$. Let us check  that our conditions are satisfied. Clearly,
$\chi$ vanishes on $\mathfrak{p}$ and so it remains to prove that $\mathfrak{u}$ is lagrangian.
First, it is easy to see that  $\omega_\chi$ vanishes on $\k$ and hence
on $\k\cap V$. Let us compute the dimension of $\k\cap V$.
We have $\k\cap V=(\g_0(\geqslant 0)\cap [\g_0,f])\oplus [\g_{>0},f]$.
The first summand is isomorphic to $[\g_0,f]/\g_0(<0)$ and so its dimension
is $\dim\g_0(<0)=\frac{1}{2}\dim [\g_0,f]$. The $\sl_2$-modules $\g_{>0}$
and $\g_{<0}$ are dual to each other and hence are isomorphic.
So $\dim [\g_{>0},f]=\frac{1}{2}(\dim [\g,f]-\dim [\g_0,f])$.
We get $\dim \k\cap V=\frac{1}{2}\dim[\g,f]=\frac{1}{2}\dim V$.

%Recall that in Subsection \ref{SSS_setting} we have introduced the anti-automorphism
%$\tau$ of $\g$. We would like to point out that $\p$ is not $\tau$-stable.

Two other examples of $K$ will be introduced in Section \ref{S_suppl}.

\subsubsection{Choice of $\iota: V\rightarrow \tilde{I}_\chi$}\label{SSS_iota_choice}
Now we are going to prove that there is a $\ZZ/2\ZZ\times R\times \K^\times$-equivariant embedding $\iota: V\hookrightarrow \tilde{I}_\chi$ with properties (2)-(3) from Subsection \ref{SSS_slice}  such that the image of $\mathfrak{u}$ lies in $\U^{\wedge_\chi}_\hbar\k$ (let us recall that here $\ZZ/2\ZZ\times \K^\times$ is a factor of $\widetilde{Q}$ and $R$ is viewed as a subgroup of $Q$, see Subsection \ref{SSS_slice}, the $\ZZ/2\ZZ\times R\times \K^\times$-actions are the restrictions of the $\widetilde{Q}$-actions).
%We will also see that in the main case of interest such an embedding can be made, in addition, $\tau$-equivariant.
Recall that we already have an embedding $V\hookrightarrow \tilde{I}_\chi$ that satisfies (1)-(3) from
Subsection \ref{SSS_slice} but we do not know yet the claim about the image of $\mathfrak{u}$.

Let us choose a basis in $\g_\chi:=\{x-\langle\chi,x\rangle, x\in \g\}$ as follows.
Let $x_1,\ldots,x_n$ be a basis  in $\mathfrak{u}$.
Let $y_1,\ldots,y_k\in V$ be such that $x_1,\ldots,x_n,y_1,\ldots,y_n$ is a Darboux basis, i.e, for the
symplectic form $\omega$ on $V$ (given by $\omega(u,v)=\langle\chi, [u,v]\rangle$) we have $\omega(x_i,x_j)=\omega(y_i,y_j)=0,
\omega(y_i,x_j)=\delta_{ij}$. Further, complete
$x_1,\ldots,x_n,y_1,\ldots,y_n$ to a basis in $\g$ with vectors $z_1,\ldots,z_m\in \z_\g(e)\cap \k, w_1,\ldots,w_{m'}\in \z_\g(e)$. We view $x_i, y_i, z_j, w_k$ as elements of $\g_\chi$ via the isomorphism $x\mapsto x-\langle \chi,x\rangle$.
Being the completion of $\K[\g_\chi]$ at the origin,  $\K[\g^*]^{\wedge_\chi}$ coincides with $\K[[x_1,\ldots,x_n,y_1,\ldots,y_n,z_1,\ldots,z_m,w_1,\ldots,w_{m'}]]$ as an algebra and
$\U^{\wedge_\chi}_\hbar$ coincides with $\K[[x_1,\ldots,x_n,y_1,\ldots,y_n,z_1,\ldots,z_m,w_1,\ldots,w_{m'},\hbar]]$ as a vector space.

\begin{Lem} There are elements $\tilde{x}_i, \tilde{y}_i, \tilde{z}_j, \tilde{w}_k\in \U^{\wedge_\chi}_\hbar$
with $i=1,\ldots,n, j=1,\ldots,m, k= 1,\ldots,m'$ such that
\begin{itemize}
\item[(i)] $\tilde{x}_i-x_i,\tilde{y}_i-y_i, \tilde{z}_j-z_j, \tilde{w}_k-w_k\in \tilde{I}_\chi^2$,
\item[(ii)] $[\tilde{x}_i,\tilde{x}_{i'}]=[\tilde{y}_i,\tilde{y}_{i'}]=0, [\tilde{x}_i,\tilde{y}_{i'}]=\delta_{ii'}\hbar^2$,
and the $\tilde{x}$'s and $\tilde{y}$'s commute with $\tilde{z}$'s and $\tilde{w}$'s.
\item[(iii)] $\tilde{x}_i, \tilde{z}_j\in \U^{\wedge_\chi}_\hbar\k$.
\item[(iv)] The map $\iota: V\rightarrow \tilde{I}_\chi$ given by $x_i\mapsto \tilde{x}_i,y_i\mapsto \tilde{y}_i$ is $\ZZ/2\ZZ\times R\times\K^\times$-equivariant.
\end{itemize}
\end{Lem}
\begin{proof}
We will construct such elements order by
order. More precisely, set $\tilde{x}_i^{(1)}:=x_i, \tilde{y}_i^{(1)}:=y_i$ etc.  For each positive integer $\ell>1$, we will produce elements $\tilde{x}^{(\ell)}_i, \tilde{y}^{(\ell)}_i, \tilde{z}^{(\ell)}_j, \tilde{w}^{(\ell)}_k$ with the following properties:
\begin{itemize}
\item[(i$^{(\ell)}$)] $\tilde{x}_i^{(\ell)}-\tilde{x}_i^{(\ell-1)}, \tilde{y}_i^{(\ell)}-\tilde{y}_i^{(\ell-1)}, \tilde{z}_j^{(\ell)}-\tilde{z}_j^{(\ell-1)},
\tilde{w}_k^{(\ell)}-\tilde{w}_k^{(\ell-1)}\in \tilde{I}_\chi^{\ell}$.
\item[(ii$^{(\ell)}$)] The commutation relations for the elements $\tilde{x}_i,\tilde{y}_i, \tilde{z}_j,\tilde{w}_k$ specified above
hold for the elements $\tilde{x}_i^{(\ell)},\tilde{y}_i^{(\ell)}, \tilde{z}_j^{(\ell)},\tilde{w}_k^{(\ell)}$ modulo $\tilde{I}_\chi^\ell$.
\item[(iii$^{(\ell)}$)] $\tilde{x}^{(\ell)}_i, \tilde{z}^{(\ell)}_j\in \U^{\wedge_\chi}_\hbar \k$.
\item[(iv$^{(\ell)}$)] The map $\iota^{(\ell)}: V\rightarrow \tilde{I}_\chi$ given by $x_i\mapsto \tilde{x}_i^{(\ell)},y_i\mapsto \tilde{y}_i^{(\ell)}$ is $\ZZ/2\ZZ\times R\times\K^\times$-equivariant.
\end{itemize}

We remark, modulo (i$^{(\ell-1)}$)-(iv$^{(\ell-1)}$), the conditions (i$^{(\ell)}$)-(iii$^{(\ell)}$)
\begin{itemize}
\item are preserved by averaging over $\ZZ/2\ZZ\times R\times \K^\times$,
\item are preserved by adding summands from $\tilde{I}_\chi^{\ell+1}$ (to $\tilde{y}_i^{(\ell)},\tilde{w}^{(\ell)}_k$)
or summands from $\U_\hbar^{\wedge_\chi}\k\cap \tilde{I}_\chi^{\ell+1}$ (to $\tilde{x}_i^{(\ell)},\tilde{z}^{(\ell)}_j$),
\item and cut an affine subspace in
$$\left(\U_\hbar^{\wedge_\chi}/\tilde{I}_\chi^{\ell+1}\right)^{\oplus n+m'}\oplus \left(\U^{\wedge_\chi}_\hbar\k/[\tilde{I}_\chi^{\ell+1}\cap \U^{\wedge_\chi}_\hbar\k]\right)^{\oplus n+m}.$$
\end{itemize}
and so we can automatically assume that (iv$^{(\ell)}$) holds as well.

Set $y_1^{(\ell)}:=y_1^{(\ell-1)}$.
Construct an element $\tilde{x}^{(\ell)}_1$ as follows. Set $a:=\frac{1}{\hbar^2}[y_1^{(\ell)},x^{(\ell-1)}_1]-1$
so that $a\in \tilde{I}_\chi^{(\ell-1)}$. Expand $a$ in the form
$a=\sum_{i=0}^\infty f_i\cdot (\tilde{x}_1^{(\ell-1)})^i$, where each $f_i$ is a series in $\tilde{x}^{(\ell-1)}_2,\ldots,\tilde{x}^{(\ell-1)}_n,\tilde{y}^{(\ell-1)}_1,\ldots,
\tilde{y}^{(\ell-1)}_n,\tilde{z}^{(\ell-1)}_1,\ldots,\tilde{z}^{(\ell-1)}_m, \tilde{w}^{(\ell-1)}_1,\ldots,\tilde{w}^{(\ell-1)}_{m'},\hbar$
(with the variables written in this order, although this does not really matter). For $i=0,\ldots,\ell-1$
we have $f_i\in \tilde{I}_\chi^{\ell-i-1}$.
Set $\tilde{x}_1^{(\ell)}:=\tilde{x}^{(\ell-1)}_1-\int a \,d\tilde{x}^{(\ell-1)}_1$, where $\int a \,d\tilde{x}^{(\ell-1)}_1:=\sum_{i=0}^\infty \frac{1}{i+1}f_i\cdot (\tilde{x}_1^{(\ell-1)})^{i+1}$.
We remark that $\int a \,d\tilde{x}^{(\ell-1)}_1\in \tilde{I}_\chi^\ell\cap \U_\hbar^{\wedge_\chi}\k$.
We have $$\frac{1}{\hbar^2}[\tilde{y}_1^{(\ell)},\tilde{x}_1^{(\ell)}]=
1+a-\sum_{i=0}^\infty \frac{1}{(i+1)\hbar^2}[\tilde{y}^{(\ell)}_1,f_i](\tilde{x}_1^{(\ell-1)})^{i+1}-
\sum_{i=0}^\infty f_i \frac{1}{(i+1)\hbar^2}[\tilde{y}^{(\ell)}_1,(x_1^{(\ell-1)})^{i+1}]$$
We have $\frac{1}{\hbar^2}[\tilde{y}^{(\ell)}_1,f_i]\in \tilde{I}_\chi^{\ell-1-i}$ because of the form
of $f_i$. So the third summand of the right hand side is in $\tilde{I}_\chi^{\ell}$. On the other hand, $\frac{1}{(i+1)\hbar^2}[\tilde{y}^{(\ell)}_1,(x_1^{(\ell-1)})^{i+1}]$ is congruent to $(\tilde{x}_1^{(\ell-1)})^i$
modulo $\tilde{I}_\chi^\ell$.
From here we deduce  that $\frac{1}{\hbar^2}[\tilde{y}_1^{(\ell)},\tilde{x}_1^{(\ell)}]-1\in \tilde{I}_\chi^\ell$.

We can actually continue the above procedure and get $\frac{1}{\hbar^2}[\tilde{y}_1^{(\ell)}, \tilde{x}_1^{(\ell)}]=1$
(and $x_1^{(\ell)}$ is still in $\U_\hbar^{\wedge_\chi}\k$). We make this choice of $\tilde{x}_1^{(\ell)}$ but this
is not our final choice for $\tilde{x}_1$ because we need to guarantee the equivariance.
We remark that $\K[[\tilde{y}_1^{(\ell)},\tilde{x}^{(\ell)}_1,\hbar]]\widehat{\otimes}_{\K[[\hbar]]}
Z=\U^{\wedge_\chi}_\hbar$, where $Z$ stands for the centralizer of $\tilde{y}^{(\ell)}_1,\tilde{x}^{(\ell)}_1$ in
$\U^{\wedge_\chi}_\hbar$, analogously to the proof of   \cite[Proposition 3.3.1]{HC}.

We can similarly correct other basis elements $b=\tilde{x}_i^{(\ell-1)}, \tilde{y}_i^{(\ell-1)}, i>1,
\tilde{z}_j^{(\ell-1)},\tilde{w}_k^{(\ell-1)}$ by elements from $\tilde{I}_\chi^{\ell}$ so that (i$^{(\ell-1)}$)-(iii$^{(\ell-1)}$) still hold and the bracket of $b$ with $\tilde{y}_1^{(\ell)}$ vanishes.

Our goal now is to show that we can modify the elements $b$ (again by adding  elements from
$\tilde{I}_\chi^\ell$) so that (ii)$^{(\ell)}$ holds for the brackets with $\tilde{y}_1^{(\ell)}, \tilde{x}_1^{(\ell)}$,
(i)$^{(\ell)}$ holds for $b$, and (iii)$^{(\ell)}$ holds if $b=\tilde{x}_i^{(\ell-1)},\tilde{z}_j^{(\ell-1)}$.
If $b=\tilde{y}_i^{(\ell-1)},\tilde{w}_i^{(\ell-1)}$, then $b$ can be corrected
similarly to $\tilde{x}_1^{(\ell-1)}$ above. Now let us consider the case when $b=\tilde{x}_i^{(\ell-1)},\tilde{z}_i^{(\ell-1)}$: a priori it is unclear why the correction of $b$ commuting with $\tilde{x}_1^{(\ell)}$ produced by a construction similar to the above still lies in $\U_\hbar^{\wedge_\chi}\k$.

Since $[\tilde{y}_1^{(\ell)},b]=0$ we have $b=\sum_{i=0}^\infty (\tilde{y}_1^{(\ell)})^i b_i$
with $b_i\in Z$. We also know that $\frac{1}{\hbar^2}[\tilde{x}_1^{(\ell)}, b]\in \tilde{I}_\chi^{\ell-1}$.
This means that $b_i\in \tilde{I}_\chi^{\ell-i}$ for $i=1,\ldots,\ell-1$. In particular, $b-b_0\in \tilde{I}_\chi^{\ell}$.
We can take $b_0$ for a lift of $b$ if we know that $b_0\in \U_\hbar^{\wedge_\chi}\k$. Let us check that
$(\tilde{y}_1^{(\ell)})^i b_i\in \U_\hbar^{\wedge_\chi}\k$ for all $i$.
First of all, let us point out that
$\frac{1}{\hbar^2}[\U_\hbar^{\wedge_\chi}\k, \U_\hbar^{\wedge_\chi}\k]\subset \U_\hbar^{\wedge_\chi}\k$, which follows
from the observation that $\k$ is a subalgebra of $\g$. Consider the operator $E:=-\frac{1}{\hbar^2}\tilde{y}^{(\ell)}_1[\tilde{x}^{(\ell)}_1,\cdot]$.
It preserves the centralizer of $\tilde{y}^{(\ell)}_1$ and also $\U^{\wedge_\chi}_\hbar\k$
and sends an element of the form $\sum_{i=0}^\infty (\tilde{y}_1^{(\ell)})^i b_i$ to $\sum_{i=0}^\infty i (\tilde{y}_1^{(\ell)})^i b_i$.
Also $\U^{\wedge_\chi}_\hbar\k$ is closed in the $\tilde{I}_\chi$-adic topology,
see \cite[Lemma 2.4.4]{HC} for a more general result. Since all elements $E^i b$ are in $\U^{\wedge_\chi}_\hbar \k$,
we see that  indeed $(\tilde{y}_1^{(\ell)})^i b_i\in \U^{\wedge_\chi}_\hbar\k$ for any $i$.

In this way we achieve that (i)$^{(\ell-1)}$-(iii)$^{(\ell-1)}$ are still satisfied and all basis elements
commute with $\tilde{y}_1^{(\ell)},\tilde{x}_1^{(\ell)}$. We can proceed in the same way with modified $\tilde{y}_2^{(\ell-1)},x_2^{(\ell-1)}$, etc. We achieve that (i)$^{(\ell)}$-(iii)$^{(\ell)}$ are satisfied
and then average the map $\iota^{(\ell)}$ with respect to $\ZZ/2\ZZ\times R\times \K^\times$.
\end{proof}

Also if $\k$ is $\tau$-stable,
then, using the same argument as before, we see that $\iota$ can be made, in addition, $\tau$-stable.
One can also show that, although $\k=\p$ is not $\tau$-stable, one can choose $\tau$-stable $\iota$,
but we will not need that result.

%In the case of most interest for us, $\k=\p$, the subalgebra is not $\tau$-stable, unless $\g=\g_0$. However,
%in this case one can still choose $\iota$ to be $\tau$-equivariant.  Indeed, $\k\cap \g_0$ is $\tau$-stable.
%Also recall that in this case we assume that $R$ is contained in $G_0$. So averaging over $\ZZ/2\ZZ\ltimes (\ZZ/2\ZZ\times %\K^\times\times R)$ preserves the property that $\iota^{(\ell)}(\mathfrak{u}\cap \g_0)\subset \U^{\wedge_\chi}_\hbar\k$. %On the other hand,
%since the averaged $\iota^{(\ell)}$ is still $[\theta,\cdot]$-equivariant, we get $\iota^{(\ell)}(V\cap\g_{>0})\subset %\U^{\wedge_\chi}_\hbar\g_{>0}
%\subset \U^{\wedge_\chi}_\hbar\k$.

In general, it seems that we cannot make $\iota$ both $Q$-equivariant (as in Subsection \ref{SSS_slice}) and mapping $\mathfrak{u}$
to $\U_\hbar^{\wedge_\chi}\mathfrak{k}$. Having such a property is desirable: this would imply that the functor
$\bullet_{\dagger}$ from Subsection \ref{SS_dagger_functor} defined using $\iota$ maps
HC $\U$-bimodules to $Q$-equivariant bimodules (a priori, we only get $R$-equivariance).
However, it turns out that the bimodules in the image of $\bullet_\dagger$ defined using $\iota$
have a $Q$-equivariant structure for our present choice of $\iota$ as well. Let us explain why this is the case.
We still have an action of $\widetilde{Q}$ on $\Walg_\hbar^{\wedge_\chi}$ that restricts to the
$\ZZ/2\ZZ\times R\times \K^\times$-action coming from the splitting constructed from our present
choice of $\iota$. This is because any two $\ZZ/2\ZZ\times R\times \K^\times$-equivariant $\iota$'s
differ by an automorphism of the form $\exp(\frac{1}{\hbar^2}\ad(a))$, where $\frac{1}{\hbar^2}a$ is a $\ZZ/2\ZZ\times R\times \K^\times$-invariant element and $a\in \tilde{I}_\chi^3$ as in Subsection \ref{SSS_slice}. The target category for $\bullet_{\dagger}$ is still $\HC^Q(\Walg)$ because the transformation
$\exp(\frac{1}{\hbar^2}\ad(a))$ acts on $\M_\hbar^{\wedge_\chi}$ for any HC bimodule $\M$.

\subsubsection{Construction of $\bullet_{\dagger,e}$}\label{SSS_dag_constr}
Pick  a module $\M\in \OCat^K_\nu$. We can choose a $K$-stable increasing exhaustive filtration
$\Fi_0 \M\subset \Fi_1 \M\subset\ldots$ on $\M$ such that this filtration is compatible with the filtration
on $\U$ and $\gr \M$ is a finitely generated $S(\g)$-module. In fact, since the filtration is $K$-stable,
$\gr \M$ is a $S(\g/\k)=\K[\k^\perp]$-module. Consider the Rees $\U_\hbar$-module
$\M_\hbar:=\bigoplus_{i=0}^\infty \Fi_i \M\hbar^i$.

The space $\K[[\mathfrak{u},\hbar]]$ has a natural structure of a $\W_\hbar^{\wedge_0}$-module,
 $\K[[\mathfrak{u},\hbar]]=\W^{\wedge_0}_\hbar/\W^{\wedge_0}_\hbar \mathfrak{u}$ (an element $u\in \mathfrak{u}$ acts by $\hbar^2\partial_u$).

\begin{Lem}\label{Lem:N_decomp}
Let $\M'_\hbar$ be the annihilator of $\mathfrak{u}$ in $\M_\hbar^{\wedge_\chi}$.
The natural homomorphism $$\K[[\mathfrak{u},\hbar]]\widehat{\otimes}_{\K[[\hbar]]}\M'_\hbar\rightarrow
\M^{\wedge_\chi}_\hbar$$ is bijective.
\end{Lem}
\begin{proof} The proof is similar to that of Lemma \ref{Lem:compl_mod} and to other related statements
such as \cite[Proposition 3.3.1]{HC} but we provide a proof for reader's convenience.

The kernel of the homomorphism is an $\hbar$-saturated (meaning that the quotient is $\K[[\hbar]]$-flat)
$\W_\hbar^{\wedge_0}$-submodule in $\K[[\mathfrak{u},\hbar]]\widehat{\otimes}_{\K[[\hbar]]}\M'_\hbar$. Any such submodule can be shown
to intersect $\M'_\hbar$ by an argument similar to \cite[Lemma 3.4.3]{wquant}. It follows that the kernel is zero. It remains to
prove that the homomorphism is surjective.

Recall that, thanks to the choice of a filtration on $\M$, we have
$\k \M_\hbar\subset \hbar^2\M_\hbar$. It follows that  $\mathfrak{u}\M_\hbar^{\wedge_\chi}\subset
\hbar^2 \M_\hbar^{\wedge_\chi}$.

Choose a basis $x_1,\ldots,x_m$ in $\mathfrak{u}$ and vectors $y_1,\ldots,y_m\in V$ such that
$y_1,\ldots,y_m, x_1,\ldots,x_m$ form a Darboux basis of $V$.
For each $i=1,\ldots,m$ and any $v\in \M^{\wedge_\chi}_\hbar$
the sum $\rho_i(v):=\sum_{j=0}^{+\infty} \frac{(-1)^j}{j}y_i^j \frac{1}{\hbar^{2j}}x_i^j v$ converges.
Moreover, this sum is annihilated by $x_i$. Also, by the construction, $v-v_0=y_i v'$, where $v_0:=\rho_i(v)$,
for some $v'\in \M^{\wedge_\chi}_\hbar$.  We can repeat the same argument with $v'$ and get a decomposition
$v-v_0-y_i v_1=y_i^2 v_0''$. Repeating this procedure we represent $v$ as the infinite sum
$\sum_{j=0}^{+\infty} y_i^j v_j$ with $x_i v_j=0$ for all $j$. The element $v_j$ in this expression
 has to be given by $v_j=\rho_i(\frac{1}{j!\hbar^{2j}}x_i^j v)$.
The operator $\rho_i$ commutes with $x_{i'},y_{i'}, \rho_{i'}$ for $i'\neq i$. So we can first decompose
$v$ into the sum  $\sum_{j=0}^{+\infty} y_1^j v_j$, then decompose each $v_j$ into the sum
$\sum_{k=0}^{+\infty} y_2^k v_{jk}$ as above. But now each $v_{jk}$ is annihilated by both
$x_1,x_2$. Proceeding in this way, we get a  decomposition of $v$ as of an element in $\K[[\mathfrak{u},\hbar]]\widehat{\otimes}_{\K[[\hbar]]}\M'_\hbar$.
\end{proof}

With this lemma, we can construct the functor $\bullet_{\dagger,e}$ completely analogously to $\bullet_\dagger$.
Namely, let $\Nil_\hbar$ denote the subspace of $\K^\times$-stable elements in $\M_\hbar'$. Again, similarly to
\cite[Proposition 3.3.1]{HC}, using the fact that $\Walg_\hbar$ is positively graded, one shows that $\Nil_\hbar^{\wedge_\chi}=\M'_\hbar$. Then we set $\M_{\dagger,e}:=\Nil_\hbar/(\hbar-1)\Nil_\hbar$. This is a finitely generated $\Walg$-module that is (canonically) independent
of the choice of a good filtration on $\M$ (as in the proof of a similar claim for
$\bullet_\dagger$ in \cite[3.4]{HC}).

The group $R$ naturally acts on $M_{\dagger,e}$ and this action is rational. However, the differential of the $R$-action coincides with
the action of $\mathfrak{r}\subset \Walg$ only up to a shift by $\nu+\delta_K$, where $\delta_K$
is half the character of $K$ on $\bigwedge^{top}\mathfrak{u}^*$, as in the end of \ref{SSS_compl_right_hand},
for the reason similar to that situation. In particular, when $K=P$, we have $\delta_P=\tau(\delta)$.
This is because $\mathfrak{u}$ and $\bar{\v}$ are complimentary lagrangian subspaces in $V$ so that the characters
of the $R$-action on $\bigwedge^{top}\mathfrak{u}^*\cong \bigwedge^{top}\bar{\v}$ coincide.

We also can consider the right-handed analog $\OCat^{K,r}_\nu$ of $\OCat^K_\nu$ in the category of
right $\U$-modules. A straightforward ramification of the construction above produces a functor
$\bullet_{\dagger,e}^r$ from $ \OCat^{K,r}_\nu$ to the category of $R$-equivariant right $\Walg$-modules.

\subsubsection{General properties of $\bullet_{\dagger,e}$}\label{SSS_dag_prop}
It  follows from the construction that the functor $\bullet_{\dagger,e}:\OCat^K_\nu\rightarrow \Walg$-$\operatorname{mod}^R_{\nu}$ is exact. Here $\Walg$-$\operatorname{mod}^R_{\nu}$ is the category of finitely generated $\Walg$-modules that are
$R$-equivariant after the shift of the $\mathfrak{r}$-action by $\nu+\delta_K$. Indeed, if we have an exact sequence
$0\rightarrow M_1\rightarrow M_2\rightarrow M_3\rightarrow 0$ of modules in $\OCat^K_\nu$, then we can choose
a good filtration on $M_2$ and induce good filtrations on $M_1,M_3$. The sequence $0\rightarrow M_{1\hbar}
\rightarrow M_{2\hbar}\rightarrow M_{3\hbar}\rightarrow 0$ becomes exact. We deduce the claim about exactness
of $\bullet_{\dagger,e}$ from the exactness of the completion functor.

Further, for $\M\in \OCat^K_\nu$ and  a Harish-Chandra
bimodule $X$ we have $X\otimes_{\U}\M\in \OCat^K_\nu$ and we have a natural isomorphism
$(X\otimes_{\U}\M)_{\dagger,e}\xrightarrow{\sim} X_{\dagger}\otimes_{\Walg}\M_{\dagger,e}$.
This is proved by tracking the constructions of $\bullet_{\dagger}$ and $\bullet_{\dagger,e}$,
compare with the proof of \cite[Proposition 3.4.1,(2)]{HC}. In particular, if
$\overline{\Orb}$ is an irreducible component of $\VA(\U/\Ann_{\U}(\M))$, then $\M_{\dagger,e}$ is finite dimensional.
Indeed, in this case $\operatorname{Ann}_{\U}(\M)_\dagger$ has finite codimension in $\Walg$, and
$\operatorname{Ann}_{\U}(\M)_\dagger\M_{\dagger,e}=0$.

Next, since $\Walg$ is a filtered algebra with $\gr\Walg=\K[S]$, one can define the associated variety $\VA(N)$
of a finitely generated $\Walg$-module $N$, this will be  a conical (with respect to the Kazhdan action)
subvariety of $S$. Tracking the construction of $\bullet_{\dagger,e}$, we have $\VA(M_{\dagger,e})=\VA(M)\cap S$.

Further, we claim that there is a functor $\bullet^{\dagger,e}:\Walg$-$\operatorname{mod}^R_{\nu}\rightarrow \tilde{\OCat}^K_\nu$
with the property that $\Hom_\U(\bullet,?^{\dagger,e})\cong\Hom_{\Walg,R}(\bullet_{\dagger,e},?)$.
This functor is constructed similarly to the functor $\bullet^\dagger$ in \cite[3.3,3.4]{HC}. Namely, we pick $\Nil\in \Walg$-$\operatorname{mod}^R_{\nu}$. Choose some good filtration and form the Rees module
$\Nil_\hbar$. Then set $\M_\hbar':=\K[[\mathfrak{u},\hbar]]\widehat{\otimes}_{\K[[\hbar]]}\Nil_\hbar^{\wedge_\chi}$.
Choose the maximal subspace $\M_\hbar$, where $\K^\times$ acts locally finitely and the action of
$\k$ shifted by $\nu$ (i.e., given by $(x,m)\mapsto \frac{1}{\hbar^2}xm-\langle \nu,x\rangle m$)
integrates to a $K$-action. This subspace is $\g$- and $R$-stable and hence we have two $R$-actions:
one obtained by restricting the action from $\M_\hbar'$ and the other restricted from a $K$-action.
Similarly to \cite[3.3]{HC}, the structure map $\g\otimes \M_\hbar\rightarrow \M_\hbar$ is $R$-equivariant
for both actions and the restrictions of the two $R$-actions to $R^\circ$ coincide. As in \cite[3.2,3.3]{HC}, we deduce that the difference
of the two actions is an $R/R^\circ$-action commuting with $\g$. We set $\Nil^{\dagger,e}:=\M_\hbar^{R/R^\circ}/(\hbar-1)$.
This is an object in $\tilde{\OCat}^K_\nu$.  We remark that $\Nil^{\dagger,e}$ comes equipped with a filtration (induced
from the grading on $\M_\hbar^{R/R^\circ}$). By the construction, for any $\mathcal{L}\in \OCat^K_{\nu}$ and any good
filtration on $\mathcal{L}$, we have $\Hom_{\Walg_\hbar}(R_\hbar(\mathcal{L}_{\dagger,e}), \Nil_\hbar)^R\cong
\Hom_{\U_\hbar}(\mathcal{L}_\hbar, \M_\hbar)$ (an isomorphism of $\K[\hbar]$-modules with $\K^\times$-action).
We have $\Hom_{\Walg}(\mathcal{L}_{\dagger,e}, \Nil)^R\cong \Hom_{\Walg_\hbar}(R_\hbar(\mathcal{L}_{\dagger,e}), \Nil_\hbar)^R/(\hbar-1)$
as any homomorphism $\mathcal{L}_{\dagger,e}\rightarrow \Nil$ can be made filtration preserving after a shift of
a filtration on $\mathcal{L}_{\dagger,e}$. Here we only use that the filtration on $\mathcal{L}_{\dagger,e}$ is good.
Similarly,  $\Hom_{\U_\hbar}(\mathcal{L}_\hbar, \M_\hbar)/(\hbar-1)\cong\Hom_\U(\mathcal{L}, \Nil^{\dagger,e})$.
So $\Hom_{\Walg}(\mathcal{L}_{\dagger,e}, \Nil)^R\cong
\Hom_{\U}(\mathcal{L}, \Nil^{\dagger,e})$. So we indeed get a functor $\bullet^{\dagger,e}$ with the properties
required in the beginning of the paragraph.

For a conical $K$-stable
subvariety $Z\subset \k^\perp$ we can form the full subcategory $\OCat^K_{\nu,Z}\subset \OCat^K_\nu$
consisting of all modules $M$ with $\VA(M)\subset Z$.

Until the end of the subsection we make an additional assumption. Namely, we assume that
$\k^\perp\cap \Orb$ has finitely many $K$-orbits. Set $Y:=Ke$. This is an irreducible component
of $\k^\perp\cap \Orb$ and its dimension equals $\frac{1}{2}\dim V$.

Consider the subvariety $Z:=\bigcup_{\Orb'} \Orb'\cap \k^\perp$,
where the union is taken over all nilpotent orbits $\Orb'$ such that $\Orb\not\subset \overline{\Orb'}$. We can form the subcategories
$\OCat^K_{\nu,Z},\OCat^K_{\nu,Z\cup Y}\subset \OCat^K_\nu$ and their quotient $\OCat^K_{\nu,Y}$.
The subcategory $\OCat^K_{\nu,Z}$ is precisely the kernel of $\bullet_\dagger$, while $\OCat^K_{\nu,Z\cup Y}$ is the preimage of
the category of finite dimensional representations.
In particular,  the functor $\bullet_{\dagger,e}$ descends to a functor from $\OCat^K_{\nu,Y}$ to the category
$\Walg$-$\operatorname{mod}^R_{\nu,fin}$ of finite dimensional $\nu$-shifted $R$-equivariant $\Walg$-modules.

Now consider the special case when $\operatorname{codim}_{\overline{Y}}(\partial Y)>1$. We can restrict
$\bullet_{\dagger,e}$ to the subcategory $\OCat^K_{\nu,\overline{Y}}$. Then we can show analogously
to \cite[Proposition 3.3.4]{HC} that $\Nil^{\dagger,e}$ lies in $\OCat^K_{\nu,\overline{Y}}$ for any
$\Nil\in \Walg$-$\operatorname{mod}^R_{\nu,fin}$.

Also under the assumption that $\operatorname{codim}_{\overline{Y}}(\partial Y)>1$ we have an analog of
\cite[Theorem 4.1.1]{HC}. Namely, pick $\M\in \OCat^K$ and let $\Nil\subset \M_{\dagger,e}$
be an $R$-stable submodule of finite codimension. Then there is a unique maximal submodule $\M'\subset \M$
such that $\M'_{\dagger,e}=\Nil$. Moreover, the quotient $\M/\M'$ lies in $\OCat^K_{\overline{Y}}$.
The module $\M'$ coincides with the preimage of $\Nil^{\dagger,e}\subset (\M_{\dagger,e})^{\dagger,e}$
under the adjunction morphism $\M\rightarrow (\M_{\dagger,e})^{\dagger,e}$.
The proof repeats that of \cite[Theorem 4.1.1]{HC}.

We remark however that sometimes the codimension condition is not necessary: for  $\OCat^P_\nu$ the conclusions of the two previous paragraphs are still true, as we will see in Subsection \ref{SS_Vfun_prop}.

\subsection{Functor $\Vfun$}\label{SS_Vfun_constr}
\subsubsection{Three definitions}\label{SSS_Vfun_def}
Our goal is to define a functor $\Vfun:\OCat^P_\nu\rightarrow \OCat^\theta(\g,e)^R_\nu$. We start by giving three
different definitions. Below in this subsection we will see that all three functors are isomorphic.

We claim that  $\M_{\dagger,e}\in \OCat^\theta(\g,e)_\nu$ for $\M\in \OCat^P_\nu$.
By the construction, the action of $\mathfrak{r}\subset R$ integrates to $R$ after the $\nu+\delta$-shift.
The other claims that we need to check (that the weight spaces are finite dimensional
and the weights are bounded by above)  will follow from similar claims about $\gr\M_{\dagger,e}$. In turn, those will
follow if we show that the one-parametric subgroup $\K^\times\rightarrow Q$
with differential $\theta$ contracts $\VA(\M_{\dagger,e})$ to $e$. But $\VA(\M_{\dagger,e})=\VA(\M)\cap S$.
Since $\VA(\M)\subset \p^{\perp}$ and $\p^{\perp}=(\g_0\cap \p^{\perp})\oplus \g_{>0}$, we only need to check
that $(\g_0\cap \p^{\perp})\cap (e+\z_{\g_0}(f))=\{e\}$. Since $\g_0\cap \p^\perp=\g_0(>0)$ and $\z_{\g_0}(f)\subset \g_0(\leqslant 0)$,
our claim follows. So we can set $\Vfun_1(\M):=\M_{\dagger,e}$.

Let us now define $\Vfun_2$.  Pick  $\M\in \OCat^P_\nu$. Set $\M^{\wedge}:=\varprojlim_{n\rightarrow +\infty} \M/\bar{\m}_\chi^n \M$.
We claim that $\M^{\wedge}=\prod_\mu \M_\mu^{\wedge_0}$, where $\bullet^{\wedge_0}$ is the analog of
$\bullet^\wedge$ for $\g_0$. Indeed, $\prod_\mu \M_\mu^{\wedge_0}$ is complete and separated in the
$\bar{\m}_\chi$-adic topology. This gives a  map $\iota:\M^{\wedge}\twoheadrightarrow \prod_\mu \M_\mu^{\wedge_0}$.
Let us point out that the $\U$-module $\prod_\mu \M_\mu$
coincides with the $\g_{<0}$-adic completion on $\M$. So we get a map
$\iota':\prod_\mu \M_\mu^{\wedge_0}\rightarrow \M^\wedge$ that is the inverse of $\iota'$
because $\iota,\iota'$ are the identity on $\M$ that is dense in both modules.

So $\M^\wedge$ is an object in $\widehat{\Wh}^\theta(\g,e)_\nu^R$.
Therefore $\widehat{\Kfun}(\M^{\wedge})_{fin}\in \tilde{\OCat}^\theta(\g,e)^R_\nu$. We set
$\Vfun_2(\M):=\widehat{\Kfun}(\M^{\wedge})_{fin}$ so that $\Vfun_2$ is a functor $\OCat^P_\nu\rightarrow
\tilde{\OCat}^\theta(\g,e)^R_\nu$. Below we will see that the image is actually in $\OCat^\theta(\g,e)_\nu^R$.

Finally, let us define a functor $\Vfun_3$. Again, pick $\M\in \OCat^P_\nu$ and consider the completion  $\widehat{\M}$.
This is an object in $\hat{\OCat}^P_\nu$. Recall the element $n\in G$ defined in \ref{SSS_setting}, it is the image
of $\begin{pmatrix}0&i\\i&0\end{pmatrix}\in \operatorname{SL}_2(\K)$ under the homomorphism $\operatorname{SL}_2(\K)\rightarrow G$ induced
by the $\sl_2$-triple $(e,h,f)$.
We twist the $\g$-action on $\widehat{\M}$ with $\operatorname{Ad}(n^{-1})$,
we get an object $\,^{n^{-1}}\!\widehat{\M}\in \hat{\OCat}^{n^{-1}Pn}_{\nu}$. The nil-radical of the parabolic subgroup
$n^{-1}Pn$ coincides with $\m$.  For a module $\M'\in \hat{\OCat}^{n^{-1}Pn}_{\nu}$ we let $\operatorname{Wh}_\nu(\M')$
be the subspace in $\M'$ spanned by all $\t$-weight vectors that are  nilpotent for the action of
$\m^0_\chi$ (and automatically  nilpotent for $\g_{>0}$). Hence $\Wh_\nu(\M')\in \tilde{\Wh}^\theta(\g,e)^R_\nu$.
So $\Vfun_3(\M):=\Kfun(\Wh_\nu(\,^{n^{-1}}\!\widehat{\M}))$ is an object of  $\tilde{\OCat}^\theta(\g,e)^R_\nu$.

Below in this subsection we will prove that the three functors $\Vfun_i$ are isomorphic. The scheme of a proof
is as follows. First, we show that $\Vfun_1\cong \Vfun_2$, this is quite easy. After that, it remains
to prove that $\Vfun_2\cong\Vfun_3$. We start proving this by  showing that $\Vfun_3$
is dual to $\Vfun_2$, i.e. $\Vfun_3(\bullet)\cong\Vfun_2(\bullet^\vee)^\vee$. Next, we show that $\Vfun_i(X\otimes_\U \M)\cong X_{\dagger}\otimes_{\Walg}\Vfun_i(\M)$ for all $i$ (a bi-functorial isomorphism). After that, we show that
the functors $\Vfun_i$ basically intertwine the Verma module functors. Finally, we use the Bernstein-Gelfand
equivalence to establish an isomorphism $\Vfun_2\cong \Vfun_3$.

\subsubsection{Isomorphism of $\Vfun_1$ and $\Vfun_2$}
We use the identification $\wU\cong \wW(\Walg)$
to view $\M^\wedge=\wU\otimes_{\U}\M$ as a module over $\wW(\Walg)$. This module is isomorphic to
$\K[[\bar{\v}^*]]\widehat{\otimes}\Nil$, where $\Nil$ is a $\wWalg$-module.
Taking $\theta$-finite elements in $\Nil$, we get $\Vfun_2(\M)$. Recall that by $\bar{\v}^*$ we mean
any lagrangian subspace in $V$ complimentary to $\bar{\v}$, in particular, we can take $\bar{\v}^*:=\mathfrak{u}(=V\cap \p)$.

We can define $\bullet_{\dagger,e}$ on the dehomogenized level (we still need
to fix a good filtration on $\M$). Namely, we can set $\M^{\heartsuit}:=\U^{\heartsuit}\otimes_{\U}\M$,
where $\U^{\heartsuit}$ is defined by (\ref{eq:heart}).
Alternatively, $\M^\heartsuit=(\M_\hbar^{\wedge_\chi})_{\K^\times-fin}/(\hbar-1)$. The algebra $\U^\heartsuit$
is naturally included into $\wU$ and we  have
$$\M^\wedge=\wU\otimes_{\U}\M=\wU\otimes_{\U^{\heartsuit}}(\U^\heartsuit\otimes_\U \M)=\wU\otimes_{\U^\heartsuit}\M^\heartsuit.$$ On the other hand, applying the $\heartsuit$-construction (i.e., taking  $\K^\times$-finite elements and modding out $\hbar-1$)
to the decomposition $\M^{\wedge_\chi}_\hbar=
\K[[\mathfrak{u},\hbar]]\widehat{\otimes}_{\K[[\hbar]]}\Nil_\hbar^{\wedge_\chi}$, we get
$$\M^\heartsuit= \W(\Walg)^{\heartsuit}\otimes_{\W(\Walg)}(\K[\mathfrak{u}]\otimes \M_{\dagger,e}).$$ The module $\M_{\dagger,e}$
can be recovered from $\M^{\heartsuit}$ by taking the quotient by $\bar{\v}$ and then taking $\theta$-finite vectors (in fact,
the latter is not necessary).
We get
$$\M^{\wedge}=\wW(\Walg)\otimes_{\W(\Walg)^\heartsuit}\M^\heartsuit=
\wW(\Walg)\otimes_{\W(\Walg)}(S(\bar{\v})\otimes \M_{\dagger,e})=\K[[\mathfrak{u}]]\widehat{\otimes}\widehat{\M_{\dagger,e}}.$$
So $\Nil\cong \widehat{\M_{\dagger,e}}$ and hence $\Vfun_2(\M)\cong\M_{\dagger,e}$. Tracking the construction,
we see that it is functorial.

It remains to prove that $\Vfun_3\cong \Vfun_2$. The proof will occupy the rest of the subsection,
its various parts will also be used later.

\subsubsection{$\Vfun_2,\Vfun_3$ and duality}\label{SSS_V_dual}
Here we are going to show that $\Vfun_2(\bullet^\vee)^\vee\cong \Vfun_3(\bullet)$.
In particular, this will imply that $\Vfun_3$ is exact and its image is in
$\OCat^\theta(\g,e)_\nu^R$.

We are going to show that there is a natural non-degenerate pairing between $\Vfun_3(\M)$
and $\Vfun_2(\M^\vee)$. As we have seen, Lemma \ref{Lem:naive_dual_char}, this implies the existence of a functorial
isomorphism $\Vfun_3(\M)\cong \Vfun_2(\M^\vee)^\vee$.

Recall that we have a $\sigma$-contravariant pairing $\M^\vee\times \widehat{\M}\rightarrow \K$.
It can be regarded a $\tau$-contravariant pairing $\M^\vee\times \,^{n^{-1}}\!\widehat{\M}$  that identifies $\,^{n^{-1}}\!\widehat{\M}$ with the full dual of $\M^\vee$. It follows that there is a $\tau$-contravariant
pairing \begin{equation}\label{eq:pairing}(\M^\vee)^{\wedge}\times \Wh_\nu(\,^{n^{-1}}\!\widehat{\M})\rightarrow \K\end{equation}
that identifies $\Wh_\nu(\,^{n^{-1}}\!\widehat{\M})$ with the continuous dual of $(\M^\vee)^{\wedge}$.
Recall that the identifications $\wU\cong \wW(\Walg)$ and $\U^\wedge\cong \W(\Walg)^\wedge$
are obtained from one another by a $\tau$-twist.
So (\ref{eq:pairing}) gives rise to a non-degenerate (in the sense that both left and right kernels are zero)
contravariant pairing between $(\M^\vee)^{\wedge}/\bar{\v}(\M^\vee)^\wedge$ and $\Wh_\nu(\,^{n^{-1}}\!\widehat{\M})^{\v}$.
The latter module is $\Vfun_3(\M)$. The former module is $\widehat{\Vfun_2(\M^\vee)}$. So we get a natural non-degenerate
contravariant pairing between $\Vfun_2(\M^\vee)$ and $\Vfun_3(\M)$. An isomorphism $\Vfun_3(\bullet)\cong \Vfun_2(\bullet^\vee)^\vee$ is therefore proved.

\subsubsection{Products with Harish-Chandra bimodules}\label{SSS_HC_products}
We claim that all three functors $\Vfun_i$ satisfy $\Vfun_i(X\otimes_{\U} \M)\cong X_{\dagger}\otimes_{\Walg}\Vfun_i(\M)$
(functorially in $X$ and $\M$), where $X$ is a Harish-Chandra bimodule. We have already established this property
for $\Vfun_1$ in \ref{SSS_dag_prop}. For $\Vfun_2$, the property follows from the already proved isomorphism $\Vfun_1\cong \Vfun_2$.
So it remains to prove the functorial isomorphism for $\Vfun_3$.

First of all, we claim that there is a natural transformation $X_{\dagger}\otimes_{\Walg}\Vfun_3(\M)
\rightarrow \Vfun_3(X\otimes_{\U}\M)$. There is a natural homomorphism $X\otimes_{\U} \,^{n^{-1}}\!\widehat{\M}\rightarrow \,^{n^{-1}}\!\widehat{X\otimes_{\U} \M}$
extending $X\otimes_{\U}\,^{n^{-1}}\!\M=\,^{n^{-1}}\!(X\otimes_{\U}\M)\rightarrow \,^{n^{-1}}\!\widehat{X\otimes_{\U}\M}$.
This isomorphism restricts to a functorial homomorphism
\begin{equation}\label{eq:natur_embed} X\otimes_{\U} \Wh_\nu(\,^{n^{-1}}\!\widehat{\M})\rightarrow \Wh_\nu(\,^{n^{-1}}\!(\widehat{X\otimes_\U\M})).\end{equation}
But according to \cite[Theorem 5.11]{W_classif}, $\Kfun(X\otimes_{\U}\M_1)\cong X_{\dagger}\otimes_{\Walg}\Kfun(\M_1)$
for any  $\M_1\in \tilde{\Wh}^\theta(\g,e)_\nu$ and this homomorphism is bifunctorial. Applying this to $\M_1:=\Wh_\nu(\,^{n^{-1}}\!\widehat{\M})$ and using (\ref{eq:natur_embed}),
we get
$$X_{\dagger}\otimes_{\Walg}\Vfun_3(\M)\cong X_{\dagger}\otimes_{\Walg}\Kfun(\M_1)=\Kfun(X\otimes_{\U}\M_1)\rightarrow
\Kfun\circ\Wh_\nu(\,^{n^{-1}}\!(\widehat{X\otimes_\U\M}))\cong \Vfun_3(X\otimes_\U\M).$$
The resulting homomorphism  $X_{\dagger}\otimes_{\Walg}\Vfun_3(\M)\rightarrow \Vfun_3(X\otimes_\U\M)$ is bi-functorial
and provides a natural transformation of interest.

\ref{SSS_V_dual} shows that $\Vfun_3$ is exact. Thanks to the 5-lemma, it is enough to show that the natural transformation is an isomorphism for $X:=L\otimes \U$, where $L$ is a finite dimensional $G$-module (here $\g$ acts on the right in a naive
way, while the left action is given by $x(l\otimes m)=x.l\otimes m+ l\otimes xm$). It is clear that
$\widehat{L\otimes \M}=L\otimes \widehat{\M}$.  Thanks to \cite[Theorem 5.11]{W_classif},
it is enough to show that taking the tensor product with $L$ commutes with taking $\Wh_\nu$.
This easily reduces to the claim that the conditions for $x\in L\otimes \M_1$ being a generalized eigenvector
with zero eigenvalue are equivalent for the following actions:
\begin{itemize}
\item the diagonal $\m_\chi$-action,
\item the $\m\times\m_\chi$-action,
\item the $\m_\chi$-action on the second factor.
\end{itemize}
This observation is a formal corollary of $\dim L<\infty$ and $\m$ acting on $L$ by nilpotent endomorphisms.
Here $\M_1$ is an arbitrary $\g$-module.

So we have proved that $X_{\dagger}\otimes_{\Walg}\Vfun_3(\M)\cong\Vfun_3(X\otimes_{\U}\M)$.

\subsubsection{Images of induced modules}\label{SSS_ind_image}
Let $\Vfun_i^0$ be the functor defined analogously to $\Vfun_i$ for $\g_0$, where $i=2,3$.
Consider the parabolic category $\OCat^{P_0}_\nu$ for $(\g_0,P_0:=P\cap G_0)$. We have the induction
functor $\Delta^0:\U^0$-$\operatorname{Mod}\rightarrow \U$-$\operatorname{Mod}, \Delta^0(\M^0):=\U\otimes_{U(\g_{\geqslant 0})}\M^0$
that maps $\OCat^{P_0}_\nu$ to $\OCat^P_\nu$.
Our goal now is to show that
the functors $\Vfun_i(\Delta^0(\bullet)),\Delta_\Walg^\theta(\Vfun_i^0(\bullet)):\OCat^{P_0}_\nu\rightarrow
\OCat^\theta(\g,e)_\nu^R$ are isomorphic. This boils down to checking an isomorphism of bifunctors
$$\operatorname{Hom}_{\OCat^\theta(\g,e)^R_\nu}(\Vfun_i(\Delta^0(\bullet)),?)\cong \Hom_{\OCat^\theta(\g,e)_\nu^R}(\Delta_{\Walg}^\theta(\Vfun_i^0(\bullet)),?).$$

Consider the case $i=2$ first. We have $\Vfun_2(\M)=\hat{\Kfun}(\M^\wedge)_{fin}$. Both
$\bullet_{fin}:\hat{\OCat}^\theta(\g,e)_\nu^R \rightarrow \OCat^\theta(\g,e)_\nu^R$ and $\hat{\Kfun}:\hat{\Wh}^\theta(\g,e)^R_\nu
\rightarrow \hat{\OCat}^\theta(\g,e)_\nu^R$
are category equivalences.
So, for $\M\in \OCat^P_\nu, \Nil\in \OCat^{\theta}(\g,e)_\nu$, we have a bifunctorial isomorphism
\begin{equation}\label{eq:coinc_Verma1}
\Hom_{\OCat^{\theta}(\g,e)^R_\nu}(\Vfun_2(\M),\Nil)\cong\Hom_{\hat{\Wh}^\theta(\g,e)^R_\nu}(\M^\wedge, \hat{\Kfun}^{-1}(\widehat{\Nil})).
\end{equation}
Clearly, $\Delta^0(\M^0)^{\wedge}=\wU\otimes_{\U}\Delta^0(\M^0)=\wU\otimes_{\U(\g_{\geqslant 0})}\M^0=
[\wU/\wU\wU_{>0}]\widehat{\otimes}_{\wU^0}\M^{0\wedge_0}=\hat{\Delta}^\theta_\U(\Vfun_2^0(\M^0))$.
Recall that $\hat{\Kfun}$ intertwines the functors $\hat{\Delta}^\theta_{\U},\hat{\Delta}^\theta_{\Walg}$,
while $\bullet_{fin}$ intertwines $\hat{\Delta}^\theta_{\Walg}$ and $\Delta^\theta_{\Walg}$.
So
\begin{equation}\label{eq:coinc_Verma2}
\Hom_{\hat{\Wh}^\theta(\g,e)^R_\nu}(\Delta^0(\M^0)^{\wedge},
\hat{\Kfun}^{-1}(\Nil))\cong\Hom_{\OCat^{\theta}(\g,e)^R_\nu}(\Delta^\theta_{\Walg}(\Vfun^0_2(\M^0)),\Nil).
\end{equation}
Combining (\ref{eq:coinc_Verma1}) and (\ref{eq:coinc_Verma2}), we see that $\Vfun_2(\Delta^0(\M^0))\cong
\Delta^\theta_{\Walg}(\Vfun^0_2(\M^0))$.

Let us proceed to $\Vfun_3$. We are going to use a similar argument. We have
\begin{equation}\label{eq:coinc_Verma3}
\Hom_{\OCat^{\theta}(\g,e)^R_\nu}(\Vfun_3(\M),\Nil)\cong\Hom_{\Wh^{\theta}(\g,e)^R_\nu}(\Wh(\,^{n^{-1}}\!\widehat{\M}), \Kfun^{-1}(\Nil)).
\end{equation}
We have a natural map \begin{equation}\label{eq:nat_map}\Delta^0(\Wh^0_\nu(\,^{n^{-1}}\!\widehat{\M^0}))\rightarrow
\Wh_\nu(\widehat{\Delta^0(\,^{n^{-1}}\!\M^0)})
\end{equation}
induced by the inclusion  $\Wh^0_\nu(\,^{n^{-1}}\!\widehat{\M}^0)\hookrightarrow \Wh_\nu(\widehat{\Delta^0(\,^{n^{-1}}\!\M^0)})$
(that comes from the inclusion $\,^{n^{-1}}\!\widehat{\M}^0\subset \widehat{\Delta^0(\,^{n^{-1}}\!\M^0)}$;
here $\Wh^0_\nu$ is an analog of $\Wh_\nu$ for $\g_0$).
Assume for a moment that (\ref{eq:nat_map}) is an isomorphism. The functors $\M^0\mapsto \,^{n^{-1}}\!\Delta^0(\M^0),
\M^0\rightarrow\Delta^0(\,^{n^{-1}}\!\M^0)$ are isomorphic via $a\otimes m\mapsto \Ad(n)a\otimes m$.
As in the case of $\Vfun_2$, $\Delta^0(\Wh_0(\,^{n^{-1}}\!\widehat{\M^0}))\cong\Delta^\theta_\U(\Vfun_3^0(\M^0))$
and, combining (\ref{eq:coinc_Verma3}) with our assumption on (\ref{eq:nat_map}), we see that
$$\Hom_{\OCat^{\theta}(\g,e)_\nu^R}(\Vfun_3(\Delta^0(\M^0)),\Nil)\cong\Hom_{\OCat^{\theta}(\g,e)^R_\nu}(\Delta^\theta_{\Walg}(\Vfun^0_3(\M^0)),\Nil).$$
Hence $\Vfun_3(\Delta^0(\M^0))\cong\Delta^\theta_{\Walg}(\Vfun^0_3(\M^0))$.

So it remains to show that (\ref{eq:nat_map}) is an isomorphism.  We start by showing that it is injective.

The composition of (\ref{eq:nat_map}) with the inclusion $\Wh_\nu\left(\widehat{\Delta^0(\,^{n^{-1}}\!\M^0)}\right)\subset \widehat{\Delta^0(\,^{n^{-1}}\!\M^0)}$
is injective. The resulting map $\Delta^0(\Wh^0_\nu(\,^{n^{-1}}\!\widehat{\M^0}))\rightarrow \widehat{\Delta^0(\,^{n^{-1}}\!\M^0)}$
is the composition of $\Delta^0(\Wh^0_\nu(\,^{n^{-1}}\!\widehat{\M^0}))\rightarrow \Delta^0(\,^{n^{-1}}\!\widehat{\M^0})$
and $\Delta^0(\,^{n^{-1}}\!\widehat{\M^0})\rightarrow \widehat{\Delta^0(\,^{n^{-1}}\!\M^0)}$. Both maps are embeddings.
So (\ref{eq:nat_map}) is injective.

To show that (\ref{eq:nat_map}) is an isomorphism it is enough to check that the characters of the images of both
modules under $\Kfun$ are the same. The argument above shows that the coincidence of characters is equivalent
to saying that the characters of $\Vfun_3(\Delta^0(\M^0))$
and $\Delta^\theta_{\Walg}(\Vfun_3^0(\M^0))$ are the same. Since both functors $\Vfun_3\circ \Delta^0, \Delta_{\Walg}^\theta\circ \Vfun_3^0$
are exact it is enough to consider the case when $\M^0$ is simple. Recall that $\Vfun_3(\bullet)\cong\Vfun_2(\bullet^\vee)^\vee$.
So $\Vfun_3\circ\Delta^0(\M^0)\cong\Vfun_2(\nabla^0(\M^0))^\vee$, where $\nabla^0(\bullet):=\Delta^0(\bullet)^\vee$.
The classes of $\nabla^0(\M^0),\Delta^0(\M^0)$ in the Grothendieck group of $\OCat^P_\nu$ coincide because the duality preserves
the simples. Since $\Vfun_2$ is exact, the characters
of $\Vfun_2\circ\nabla^0(\M^0)$ and of $\Vfun_2\circ \Delta^0(\M^0)$ coincide. Also the duality $\bullet^\vee$
for the W-algebra does not change the character, see \ref{SSS_W_duality}. We see that the characters
of $\Vfun_2(\Delta^0(\M^0))$ and of $\Vfun_3(\Delta^0(\M^0))$ are the same. Now we claim that the characters
of $\Delta^\theta_{\Walg}(\Vfun_2^0(\M^0)),\Delta^\theta_{\Walg}(\Vfun_3^0(\M^0))$ are the same, this will finish the proof.
To show the coincidence of those characters one needs to show that $\Vfun_2^0(\M^0),\Vfun_3^0(\M^0)$
are isomorphic as $\t$-modules. This again follows from $\Vfun_2(\M^{0\vee_0})^{\vee_0}\cong \Vfun_3^0(\M^0)$.
Indeed, by construction, $\bullet^{\vee_0}$ does not change the $\t$-module structure.

\subsubsection{Isomorphism of $\Vfun_2$ and $\Vfun_3$}\label{SSS_V23_iso}
First, we will show that the images of a certain parabolic Verma module under $\Vfun_2$ and
$\Vfun_3$ are isomorphic. Then we will use \ref{SSS_HC_products} and the Bernstein-Gelfand
equivalence recalled in \ref{SSS_BG_equiv} to show that $\Vfun_2\cong \Vfun_3$.

The parabolic Verma module we are going to consider  is $\Delta_P(\nu+\rho)$. Recall that $\J_{P,\nu}$
denotes its annihilator in $\U$.

Clearly, $\Delta_P(\nu+\rho)=\Delta^0(\Delta_{P_0}(\nu+\rho))$.
Thanks to \ref{SSS_ind_image}, it is enough to show that $\Vfun^0_2(\Delta_{P_0}(\nu+\rho))\cong
\Vfun^0_3(\Delta_{P_0}(\nu+\rho))$. So it is enough to assume that $\g=\g_0$.

In this case, the dimension of $\Vfun_2(\Delta_P(\nu+\rho))\cong\Vfun_1(\Delta_P(\nu+\rho))$ equals
to the multiplicity of $\Delta_P(\nu+\rho)$ on $Pe$, the dense orbit of $P$ in $\p^{\perp}$.
As we have recalled in \ref{SSS_BG_equiv}, this multiplicity equals $1$. Since the duality does not change the multiplicity
(it fixes all simples), \ref{SSS_V_dual} implies that $\dim \Vfun_3(\Delta_P(\nu+\rho))=1$.

So to show the isomorphism it remains to prove that the annihilators of both modules
coincide. The ideal $\J:=\J_{P,\nu}$ coincides with the kernel of the epimorphism $\U\rightarrow D_\nu(G/P)$,
where the target algebra is the algebra of $\nu$-twisted differential operators on $G/P$.
So $\J$ is prime (and even completely prime) and hence primitive, and $\VA(\U/\J)=\overline{\Orb}$.
Moreover, the element $e$ is even and hence the morphism $T^*(G/P)\twoheadrightarrow \overline{\Orb}$
is birational. It follows that the multiplicity of $\U/\J=D_\nu(G/P)$ on $\Orb$ is $1$. So $\J_{\dagger}$
is an ideal of codimension $1$ in $\Walg$. It annihilates $\Vfun_2(\Delta_P(\nu+\rho))$ because of
the isomorphism $\Vfun_1\cong \Vfun_2$. Let us show that $\J_\dagger$ annihilates $\Vfun_3(\Delta_P(\nu+\rho))$.
We recall that $\Vfun_3(\Delta_P(\nu+\rho))\cong\Kfun(\operatorname{Wh}_\nu(\,^{n^{-1}}\!\hat{\Delta}_P(\nu+\rho)))$.
The ideal $\J$ annihilates $\,^{n^{-1}}\!\hat{\Delta}_P(\nu+\rho)$ and hence  $\operatorname{Wh}_\nu(\,^{n^{-1}}\!\hat{\Delta}_P(\nu+\rho))$.
Thanks to \cite[Theorem 5.11]{W_classif}, $\Vfun_3(\Delta_P(\nu+\rho))\cong\Kfun(\operatorname{Wh}_\nu(\,^n\!\hat{\Delta}_P(\nu+\rho)))$ is annihilated by $\J_{\dagger}$.

The proof of the isomorphism $\Vfun_3(\Delta_P(\nu+\rho))\cong \Vfun_2(\Delta_P(\nu+\rho))$ is now complete
(for $\g=\g_0$ and hence for general $\g$, as well).

\begin{Rem}\label{Rem:right_cell}
Let $c_{\p}$ be the left cell corresponding to the primitive ideal $\J_{P,0}$. Since $\operatorname{mult}_{\Orb}(\U/\J_{P,0})=1$,
formula (\ref{eq:mult_equality}) together with the observation that $\U/\J_{P,0}$ corresponds to the triple of the form $(x,x,\operatorname{triv})$, see Subsection \ref{SS_HC_semis}, imply that the Lusztig subgroup $H_{c_{\p}}$ coincides with $\bA$.
\end{Rem}

Now we are ready to complete the proof of an isomorphism $\Vfun_2\cong \Vfun_3$. Recall the parabolic
Bernstein-Gelfand equivalence  $X\mapsto X\otimes_{\U}\Delta_P(\nu+\rho): \HC(\U)^{\J_{P,\nu}}\rightarrow \OCat^P_\nu$.
Under the identification of $\OCat^P_\nu$ with $\HC(\U)^{\J_{P,\nu}}$, thanks to \ref{SSS_HC_products},
we have $\Vfun_i(X\otimes_\U \Delta_P(\nu+\rho))\cong X_{\dagger}\otimes_{\Walg}\Vfun_i(\Delta_P(\nu+\rho))$. Since
$\Vfun_3(\Delta_P(\nu+\rho))\cong \Vfun_2(\Delta_P(\nu+\rho))$, we are finally done.

From now on, all three isomorphic functors will be denoted by $\Vfun$.

\subsection{Further properties of $\Vfun$}\label{SS_Vfun_prop}
Here we will show that $\Vfun$ is a quotient onto its image and identify the modules annihilated by $\Vfun$.
The next important property we are going to prove is that $\Vfun$ satisfies the double centralizer property,
i.e., is fully faithful on projective objects. Then we are going to establish a sufficient condition
for $\Vfun$ to be fully faithful on standardly (=parabolic Verma) filtered objects. Finally, we summarize the properties
we have proved in Theorem \ref{Thm:Vfun}.

\subsubsection{Quotient property}\label{SSS_quot}
\begin{Prop}\label{Thm:quot_prop_gen}
The following is true.
\begin{enumerate}
\item The modules killed by $\Vfun$ are precisely those
whose all weight spaces for $\t$ have GK dimension less then $\dim\g_0(<0)$ (the maximal
GK dimension of a module in $\OCat^{P_0}_\nu$). Let $\ker\Vfun$ denote the full subcategory of such modules.
\item $\Vfun:\OCat^{P}_\nu\rightarrow \OCat^\theta(\g,e)_\nu^{R}$ admits a right adjoint functor to be denoted by $\Vfun^*$.
\item $\Vfun^*$ is a left inverse of the functor $\OCat^\theta(\g,e)_\nu^{R}/\ker\Vfun\rightarrow \OCat^\theta(\g,e)_\nu^{R}$
induced by $\Vfun$.
\item The image of $\Vfun$ is closed under taking subquotients. Equivalently, if $N$ is a subobject of $\Vfun(M)$, then
$\Vfun(\Vfun^*(N))=N$.
\end{enumerate}
\end{Prop}
Let $\operatorname{im} \Vfun\subset \OCat^\theta(\g,e)_\nu^{R}$ be the essential image. By (4), $\operatorname{im}\Vfun$ is an abelian category. Now (3) implies that $\Vfun$
induces an equivalence of $\OCat^P_\nu/\ker\Vfun\xrightarrow{\sim} \operatorname{im} \Vfun$ of abelian categories. In other words,
$\Vfun$ is a quotient functor onto its image.
\begin{proof}
First, let us consider the case $\g=\g_0$.

Recall the identification $\OCat^P_\nu\cong \HC(\U)^{\J}$, where we set $\J:=\J_{P,\nu}$,
and that under this identification the functor $\Vfun$ becomes
$X\mapsto X_{\dagger}\otimes_{\Walg}\Nil$, where $\Nil$ is a unique 1-dimensional $\Walg$-module annihilated
by $\J_\dagger$. The functor $Y\mapsto Y\otimes_{\Walg}\Nil$ is an equivalence between the category
$\HC_{fin}^Q(\Walg)^{\J_\dagger}$ of the bimodules annihilated by $\J_\dagger$ from the right and the category
of $Q$-equivariant finite dimensional $\Walg$-modules.
Under the identification  $\OCat^P_\nu\cong \HC(\U)^{\J}$,
modules with non-maximal GK dimension in $\OCat^P_\nu$ correspond to $\HC(\U)^{\J}\cap \HC_{\partial\Orb}(\U)$
because this identification preserves the left annihilators. That all such modules in $\OCat^P_\nu$ are annihilated by $\Vfun$
follows from (iii) in Subsection \ref{SS_dagger_functor}, that no other module gets killed follows from (iv) there.
This shows (1).

Let us prove (2). This will follow if we show that the functor $\bullet^{\dagger}$ from (v) in Subsection \ref{SS_dagger_functor} maps $\HC^Q_{fin}(\Walg)^{\J_\dagger}$ to $\HC(\U)^{\J}$. Let us, first, show that $\J$ coincides with the kernel $(\J_{\dagger})^{\dagger_\U}$ of $\U\rightarrow (\Walg/\J_{\dagger})^{\dagger}$. For this,  we note  that $\J\subset (\J_{\dagger})^{\dagger_\U}$,
$\VA(\U/\J)=\VA(\U/(\J_{\dagger})^{\dagger_\U})=\overline{\Orb}$ and, by (v) of Subsection \ref{SS_dagger_functor}, $\VA((\J_{\dagger})^{\dagger_\U}/\J)\subset \partial \Orb$. Since $\J$ is primitive, see \ref{SSS_BG_equiv}, this implies the equality $\J=(\J_{\dagger})^{\dagger_\U}$ thanks to \cite[Corollar 3.6]{BK}.

So if $\mathcal{B}\in \HC_{fin}^Q(\Walg)$ is annihilated by $\J_\dagger$ from the right, then, by the construction of $\bullet^{\dagger}$ in \cite[3.3,3.4]{HC},
$\mathcal{B}^{\dagger}$ is annihilated by $(\J_\dagger)^{\dagger_\U}$. So $\bullet^{\dagger}:\HC^Q_{fin}(\Walg)\rightarrow \HC_{\overline{\Orb}}(\U)$
restricts to $\HC^Q_{fin}(\Walg)^{\J_{\dagger}}\rightarrow \HC(\U)^\J$. It follows that
$\Vfun$ possesses a right adjoint functor.

(3) now follows from (v) in Subsection \ref{SS_dagger_functor}, while (4) follows from (vii).

Now proceed to the case of a general $\theta$. Using the realization of $\Vfun$ as $\Vfun_2$, we see that
it is enough to prove the direct analogs of (1)-(4) for
$\mathcal{F}:\M\mapsto \M^{\wedge}:\OCat^P_\nu\rightarrow \hat{\Wh}^{\theta}(\g,e)^R_\nu$. As we have noted in
\ref{SSS_Vfun_def},
$\M^\wedge=\prod_{\mu\in \t^*}\M_\mu^{\wedge_0}$, where $\M_\mu$
is the $\mu$-weight space for the action of $\t$. (1)
follows now from (1) for $\Vfun^0$ already established above in this proof
(an object in $\OCat^{P_0}_\nu$ is annihilated
by $\Vfun_1^0$ if and only if its support does not contain $e$; this is equivalent to
the condition on the GK dimension).

Let us prove (2), i.e., that there is a right adjoint functor $\mathcal{G}$ for $\mathcal{F}$. First recall that
the functor $\mathcal{F}^0:\M^0\mapsto \M^{0\wedge_0}$ has a right
adjoint functor because it becomes $\Vfun^0$ under a suitable category equivalence.
The right adjoint  functor is realized as follows. For $\Nil^0\in \hat{\Wh}^\theta(\g_0,e)^R$
let $\tilde{\mathcal{G}}^0(\Nil^0)$ be the sum of all submodules of $\Nil^0$ belonging to
$\OCat^{P_0}_\nu$. Similarly to \ref{SSS_dag_prop}, on $\tilde{\mathcal{G}}^0(\Nil^0)$ we have two $R$-actions, one restricted
from $\Nil^0$ and one coming from the $P_0$-action. They agree on $R^\circ$ and their difference is an $R/R^\circ$-action
commuting with $\g_0$. Let $\mathcal{G}^0(\Nil^0)$ denote the subspace of $R/R^\circ$-invariants in
$\tilde{\mathcal{G}}^0(\Nil^0)$. From the construction, $$\operatorname{Hom}_{\OCat^{P_0}_\nu}(\mathcal{M}^0, \mathcal{G}^0(\Nil^0))\cong
\operatorname{Hom}_{\tilde{\OCat}^\theta(\g_0,e)^R_\nu}(\mathcal{F}^0(\mathcal{M}^0),\Nil^0).$$
So $\mathcal{G}^0$ is indeed a right adjoint functor to $\mathcal{F}^0$.

Now take $\Nil\in \hat{\Wh}^{\theta}(\g,e)^R_\nu$. Then its weight subspaces $\Nil_\mu$ are objects of
$\hat{\Wh}^\theta(\g_0,e)^R_\nu$. Clearly, $\bigoplus_\mu \Nil_\mu\subset \Nil$ is  $\g$-stable. Further,
it is easy to see that $\bigoplus_\mu \tilde{\mathcal{G}}^0(\Nil_\mu)$ is a $\g$-submodule.
The action map  $\g\times \bigoplus_\mu \tilde{\mathcal{G}}^0(\Nil_\mu)\rightarrow
\bigoplus_\mu \tilde{\mathcal{G}}^0(\Nil_\mu)$ is equivariant with respect to both $R$-actions.
So $\mathcal{G}(\Nil):=\bigoplus_\mu \mathcal{G}^0(\Nil_\mu)$ is a $\g$-submodule in $\Nil$.
It follows easily from the construction that $\Hom_{\U}(\M^\wedge,\Nil)^R\cong\Hom_{\U}(\M, \mathcal{G}(\Nil))$.
Since all $\U^0$-modules $\mathcal{G}(\Nil)_\mu=\mathcal{G}^0(\Nil_\mu)$ are in $\mathcal{O}^{P}_0$ (because
$\mathcal{G}^0$ is right adjoint to $\mathcal{F}^0$), all $\h$-weight spaces in $\mathcal{G}(\Nil)$
are finite dimensional. Since the center of  $\U$ acts on $\Nil$ with finitely many eigen-characters, the same is
true also for $\mathcal{G}(\Nil)$. Therefore $\mathcal{G}(\Nil)\in \OCat^P_\nu$. So $\mathcal{G}$
is a right adjoint functor for $\mathcal{F}$ and (2) is proved.

Let us prove (3) that amounts to showing that the kernel and the cokernel of the adjunction homomorphism $\M\rightarrow
\mathcal{G}\circ \mathcal{F}(\M)$ lie in the kernel of $\mathcal{F}$. But this homomorphism
has the form $\bigoplus_\mu \M_\mu\rightarrow \bigoplus_\mu \mathcal{G}^0(\mathcal{F}^0(\M_\mu))$,
where the maps $\M_\mu\rightarrow \mathcal{G}^0(\mathcal{F}^0(\M_\mu))$ are the adjunction maps.
Our claim follows from (3) in the case $\g=\g_0$ that was proved above.

Let us prove that the image of $\mathcal{F}$ is closed under taking taking subquotients
(equivalently, taking subobjects). Namely, let us take a subobject  $\Nil'\subset \Nil$
(in the category $\hat{\Wh}^\theta(\g,e)_\nu^R$).
It has the form $\prod_{\mu\in \t^*}\Nil'_\mu$. So
$\mathcal{G}(\Nil')=\bigoplus_{\mu}\mathcal{G}^0(\Nil'_\mu)$ and $\mathcal{F}(\mathcal{G}(\Nil'))=
\prod_\mu \mathcal{F}^0(\mathcal{G}^0(\Nil'_\mu))$. But $\mathcal{F}^0(\mathcal{G}^0(\Nil'_\mu))=\Nil'_\mu$
by (4) in the case $\g=\g_0$.
\end{proof}

\subsubsection{$\Vfun$ vs homological duality}
Recall that we have a derived equivalence $\dual_\U: D^b(\OCat^P_\nu)\rightarrow D^b(\OCat_{-\nu}^{P,r})^{opp}$, where
$\OCat_{-\nu}^{P,r}$ was defined in \ref{SSS_right_cat},  given
by $$\dual_\U(\bullet):=\operatorname{RHom}(\bullet,\U)[\dim \p].$$
The proof can be found, for example, in \cite[4.1]{GGOR}, the techniques there can be applied to our
case as explained in \cite[footnote 1]{GGOR}.

Clearly, if $\Nil$ is a finitely generated $R$-equivariant left $\Walg$-module, then
$\Hom_{\Walg}(\Nil,\Walg)$ is also finitely generated and $R$-equivariant (as a right $\Walg$-module).
Therefore we can form a homological duality functor for the W-algebra, $\dual_\Walg: D^b(\Walg\operatorname{-mod}^R_\nu)
\rightarrow D^b(\Walg^{opp}\operatorname{-mod}^R_{-\nu})^{opp},
\dual_{\Walg}(\bullet)=\operatorname{RHom}(\bullet,\Walg)[\dim \z_\g(e)_{\geqslant 0}]$.

We can define the analog of $\Vfun_1(\bullet)=\bullet_{\dagger,e}$ for the categories of right
modules completely analogously to the above, see \ref{SSS_dag_constr}. This functor will be denoted by $\Vfun^r$.

Then we have the following statement.

\begin{Prop}\label{Prop:homol_dual}
The functors  $\Vfun^r(H^i(\dual_\U(\bullet))), H^i(\dual_{\Walg}(\Vfun(\bullet))):\OCat^P_\nu\rightarrow
(\Walg^{opp}\operatorname{-mod}^R_{-\nu})^{opp}$
are isomorphic.
\end{Prop}
In fact, the functors $\Vfun^r(\dual_\U(\bullet))$ and $\dual_{\Walg}(\Vfun(\bullet))$
(from $D^b(\OCat^P_\nu)$ to  the $(R,-\nu)$-equivariant derived category of $\Walg^{opp}$)
are  isomorphic but we do not want to provide a  proof of this because
we will not need this fact. To prove the isomorphism one needs a derived version of $\bullet_{\dagger,e}$,
see \cite[Remark 5.12]{BL}.
\begin{proof}
Pick $\M\in \OCat^P_\nu$ and fix a good $G_0$-stable filtration of $\M$ so that $\gr \M$ is a $G_0$-equivariant finitely generated $S(\g)$-module.  Then we can pick a graded free $G_0$-equivariant  resolution $\ldots\rightarrow A^1\rightarrow A^0
\rightarrow \gr \M$ and lift it to a free $G_0$-equivariant resolution $\ldots\rightarrow\A^1\rightarrow \A^0\rightarrow \M$
such that $\A^i$ is the sum of several copies of $\U$ each equipped with a shift of the PBW filtration and
all differentials are strictly compatible with filtrations. Then we get a graded resolution
$\ldots\A^1_\hbar\rightarrow \A^0_\hbar\rightarrow \M_\hbar$. Let us tensor the resolution with $\U_\hbar^{\wedge_\chi}$.
Since $\U_\hbar^{\wedge_\chi}$ is a flat $\U_\hbar$-module, see \cite[2.4]{HC}, we get an $R$-equivariant
resolution of $\M_\hbar^{\wedge_\chi}$:
$$\ldots\A^{1\wedge_\chi}_\hbar\rightarrow \A_\hbar^{0\wedge_\chi}\rightarrow \M^{\wedge_\chi}_\hbar.$$
There is also another way to obtain a free resolution of $\M_\hbar^{\wedge_\chi}=\K[[\mathfrak{u},\hbar]]\widehat{\otimes}_{\K[[\hbar]]}\Nil^{\wedge_\chi}_\hbar$, where $\Nil_\hbar$
is the Rees module of $\Vfun(\M)$ (with the filtration induced from $\M$). We can produce an $R$-equivariant graded free resolution
$\ldots\rightarrow\underline{\A}^0_\hbar\rightarrow\Nil_{\hbar}$ as before and then complete at $\chi$.
Also we can take the Koszul resolution for the $\W_\hbar^{\wedge_0}$-module
$\K[[\mathfrak{u},\hbar]]=\W_\hbar^{\wedge_0}/\W_\hbar^{\wedge_0}\mathfrak{u}$. This resolution is obtained
by taking the Koszul resolution $\ldots\rightarrow \K[\mathfrak{u}]\otimes \bigwedge^i \mathfrak{u}\rightarrow \K[\mathfrak{u}]\otimes \bigwedge^{i-1}\mathfrak{u}\rightarrow\ldots$
of $\K[\mathfrak{u}]/(\mathfrak{u})$ and applying the functor $\W_{\hbar}^{\wedge_0}\otimes_{\K[\mathfrak{u}]}\bullet$ to it.
Multiplying the resolutions of $\K[[\mathfrak{u},\hbar]]$ and $\Nil^{\wedge_\chi}_\hbar$, we get another
resolution of $\M^{\wedge_\chi}_\hbar=\K[[\mathfrak{u},\hbar]]\widehat{\otimes}_{\K[[\hbar]]}\Nil^{\wedge_\chi}_\hbar$.

We are going to show that the cohomology of the dual of the first resolution produces $H^i(\Vfun^r(\dual_\U(\M)))$
(in the same way as was used in \ref{SSS_dag_constr} to pass from $\U^{\wedge_\chi}_\hbar$-modules
to $\Walg$-modules: we need to, first, factor out $\K[[\mathfrak{u},\hbar]]$,  take $\K^\times$-finite elements, then take $R/R^\circ$-invariants, and then mod out $\hbar-1$), while the cohomology of the dual of the second resolution  similarly produce
$H^i(\dual_{\Walg}(\Vfun(\M)))$.

Consider the resolution $...\rightarrow \A^2_\hbar\xrightarrow{d_1} \A^1_\hbar\xrightarrow{d_0} \A^0_\hbar$.
The complex \begin{equation}\label{eq:compl_right}\A^0_\hbar\xrightarrow{d_0^*}\A^1_\hbar\xrightarrow{d_1^*}\A^2_\hbar\ldots\end{equation}
of right modules computes
$\operatorname{RHom}(\M_\hbar,\U_\hbar)$. The complex $\A^{0\wedge_\chi}_\hbar\xrightarrow{d_0^*}
\A_\hbar^{1\wedge_\chi}\xrightarrow{d_2^*}\A^{2\wedge_\chi}_\hbar\ldots$ computes
$\operatorname{RHom}(\M_\hbar^{\wedge_\chi}, \U^{\wedge_\chi}_\hbar)$. Its cohomology is obtained
by completing the cohomology of (\ref{eq:compl_right}). This means that the cohomology
of $\Vfun^r(\dual_{\U}(\M))$ is  obtained from the first resolution in the way explained in the previous
paragraph. The proof for the second resolution is similar.

The isomorphisms of the two $H^i$'s we have constructed  do not depend on the choice of a filtration on $\M$
for the same reason as for $\bullet_{\dagger,e}$ to be a functor. Also the construction is functorial in $\M$ for the standard homological
algebra reasons (two free resolutions are homotopic to each other).
\end{proof}

\subsubsection{Double centralizer property}\label{SSS_double_centr}
Our goal here is to prove the double centralizer property: that $\Hom_{\OCat^P_\nu}(\mathcal{P}_1,\mathcal{P}_2)\cong\Hom_{\OCat^\theta(\g,e)_\nu^R}(\Vfun(\mathcal{P}_1),\Vfun(\mathcal{P}_2))$
for any projective objects $\mathcal{P}_1,\mathcal{P}_2\in \OCat^P_\nu$. Our proof closely follows that of \cite[Theorem 5.16]{GGOR}.

The socle of $\Delta_P(\lambda)$ has maximal GK dimension
and hence is not annihilated by $\Vfun=\bullet_{\dagger,e}$. It follows that the natural homomorphism
$\M\rightarrow \Vfun^*(\Vfun(\M))$ is injective for any standardly filtered (=admitting a filtration
such that the subsequent quotients are parabolic Verma modules) module $\M$.
Similarly, a costandard object (=dual Verma) has no simple quotients annihilated by $\bullet_{\dagger,e}$.
Since the quotient $\Vfun^*(\Vfun(\M))/\M$ is annihilated by $\Vfun$ for any $\M$, we see
that the natural morphism $\Hom_{\OCat^P_\nu}(\M_1,\M_2)\rightarrow \Hom_{\OCat^\theta(\g,e)_\nu^R}(\Vfun(\M_1),\Vfun(\M_2))$
is an isomorphism provided $\M_2$ is standardly filtered, while $\M_1$ is costandardly filtered.

Similar results hold for $\OCat^{P,r}_{-\nu}$ and the functor $\Vfun^r$. Pick a projective $\mathcal{P}\in \OCat^P_\nu$.
Then $\dual_\U(\mathcal{P})$ is a tilting object in $(\OCat^{P,r}_{-\nu})^{opp}$ (see \cite[Proposition 4.2]{GGOR})
or, equivalently, in $\OCat^{P,r}_{-\nu}$. In particular, $\dual_\U(\mathcal{P})$ is a standardly filtered object in $\OCat^{P,r}_{-\nu}$. Also if $\M$ is a standardly filtered
object in $\OCat^P_\nu$, then $\dual_\U(\M)$ is a standardly filtered object in $(\OCat^{P,r}_{-\nu})^{opp}$ (again, see
\cite[Proposition 4.2]{GGOR}) and so a costandardly filtered object in $\OCat^{P,r}_{-\nu}$.
So $$\Hom_{\OCat^P_\nu}(\M,\mathcal{P})\cong\Hom_{\OCat^{P,r}_{-\nu}}(\dual_\U(\mathcal{P}),\dual_\U(\M))\cong
\Hom_{\Walg^{opp}\operatorname{-mod}^{R}}(\Vfun\circ \dual_\U(\mathcal{P}), \Vfun\circ \dual_\U(\M)).$$
As we have seen, $H^i\circ\Vfun\circ \dual_\U\cong H^i\circ\dual_\Walg\circ \Vfun$. In particular,
$\dual_\Walg\circ \Vfun(\mathcal{P}), \dual_\Walg\circ\Vfun(\M)$ are objects of $\Walg^{opp}$-$\operatorname{mod}^{R}_{-\nu}$
isomorphic to $\Vfun\circ \dual_\U(\mathcal{P}), \Vfun\circ \dual_\U(\M)$. It follows that
$$\Hom_{\Walg^{opp}\operatorname{-mod}^{R}_{-\nu}}(\Vfun\circ \dual_\U(\mathcal{P}), \Vfun\circ \dual_\U(\M))\cong
\Hom_{\Walg^{opp}\operatorname{-mod}^{R}_{-\nu}}(\dual_\Walg\circ \Vfun(\mathcal{P}), \dual_\Walg\circ\Vfun(\M)).$$
But the right hand side is just $\Hom_{\Walg\operatorname{-mod}^R_\nu}(\Vfun(\M),\Vfun(\mathcal{P}))=\Hom_{\OCat^{\theta}(\g,e)^R_\nu}(\Vfun(\M),\Vfun(\mathcal{P}))$.
From this argument we deduce that the spaces
$\Hom_{\OCat^P_\nu}(\M,\mathcal{P})$ and $\Hom_{\OCat^\theta(\g,e)^R_\nu}(\Vfun(\M),\Vfun(\mathcal{P}))$, at least, have the same dimension.
Since $\M$ is standardly filtered,  the former $\Hom$ is included into the latter, so they coincide.
For $\M$ we can take another projective, since any projective in $\OCat^P_\nu$ is standardly filtered.

\subsubsection{0-faithfulness}
Our goal here is to prove that under a certain assumption on $P$ the functor $\Vfun$ satisfies an even
stronger property that the double centralizer one, namely that $\Vfun$ is $0$-faithful meaning
that $\Hom_{\OCat^P_\nu}(\M_1,\M_2)\cong\Hom_{\OCat^\theta(\g,e)_\nu^R}(\Vfun(\M_1),\Vfun(\M_2))$
for any standardly filtered objects $\M_1,\M_2$. The condition
on $P$ is that the complement to the open orbit of $P_0$ in $\g_0(>0)$ has codimension (in $\g_0(>0)$)
bigger than $1$. We will elaborate on this condition for $\g\cong \sl_n$ in Subsection \ref{SS_parab_top_quot}.

\begin{Prop}\label{Prop:0_faith}
Suppose that $P$ satisfies the condition of the previous paragraph. Then the functor
$\Vfun$ is $0$-faithful.
\end{Prop}
\begin{proof}
The 0-faithfulness condition is equivalent to $\mathcal{G}\circ \mathcal{F}(\M)\cong\M$ for any standardly filtered
$\M\in \OCat^P_\nu$, where  $\mathcal{F}$ is the completion functor
$\M\mapsto \M^\wedge$ and $\mathcal{G}$ is the right adjoint constructed in  the proof of Proposition \ref{Thm:quot_prop_gen}.

First, we reduce to the $\g=\g_0$ case.  Recall
that $\mathcal{F}(\M)=\prod_{\mu\in \t^*}\mathcal{F}^0(\M_\mu), \mathcal{G}(\Nil)=\bigoplus_{\mu} \mathcal{G}^0(\Nil_\mu)$
and therefore \begin{equation}\label{eq:0_faith_eq} \mathcal{G}(\mathcal{F}(\M))=\bigoplus_\mu \mathcal{G}^0(\mathcal{F}^0(\M_\mu)).\end{equation} So it is enough to
check that $\mathcal{G}^0(\mathcal{F}^0(\M_\mu))\cong\M_\mu$ for any $\t$-weight space $\M_\mu$ of $\M$.

We claim that if $\M$ is standardly filtered (in $\OCat^P_\nu$), then any $\t$-weight space $\M_\mu$
is standardly filtered in $\OCat^{P_0}_\nu$. It is enough to check this when $\M$ is a parabolic
Verma module. We have $\Delta_P(\mu)=\Delta^0(\Delta_{P_0}(\mu))$. So as a $\g_0$-module,
$\Delta_P(\mu)$ is $U(\g_{<0})\otimes \Delta_{P_0}(\mu)$. So any $\t$-weight space of $\Delta_P(\mu)$
is the tensor product of $\Delta_{P_0}(\mu)$ with some finite dimensional $G_0$-module and hence
is standardly filtered.

The previous paragraph together with (\ref{eq:0_faith_eq}) reduces the proposition to the case
when  $\g_0=\g$ (and $e$ is even). We consider that case until the end of the proof. We remark that, thanks the 5-lemma,
in order to check the isomorphism $\M\cong \mathcal{G}(\mathcal{F}(\M))$ for any standardly filtered module $\M$, it suffices
to consider the case when $\M$ is a parabolic Verma module.

Let us equip $\M=\Delta_P(\lambda)=\U\otimes_{U(\mathfrak{p})}L_{00}(\lambda)$ with a good filtration induced from the trivial filtration on $L_{00}(\lambda)$.
Using the realization of $\Vfun$ as $\bullet_{\dagger,e}$ and $\Vfun^*$ as $\bullet^{\dagger,e}$,
we see that it is enough to prove $\mathcal{G}_\hbar(\M_\hbar^{\wedge_\chi})=\M_\hbar$,
where $\mathcal{G}_\hbar$  is the composition of
\begin{itemize}
\item the functor $\bullet_{fin}$ of taking the maximal subspace where the $\nu$-shifted $\p$-action integrates
to $P$ and the Kazhdan action of $\K^\times$ is locally finite
\item and the functor of taking $R/R^\circ$-invariants.
\end{itemize}
We remark that the orbit $Pe$ is simply connected because it is the complement of a codimension 2 subvariety in an affine space.
On the other hand, we have a covering $P/Z_P(e)^\circ\twoheadrightarrow Pe$ so $Z_P(e)^\circ=Z_P(e)$ and $R/R^\circ$ is trivial.
So $\mathcal{G}_{\hbar}(\bullet)=\bullet_{fin}$.
 To show
that $\M_\hbar=(\M_\hbar^{\wedge_\chi})_{fin}$ it is enough to check (compare to \cite[3.3]{HC}) the analogous
equality modulo $\hbar$: i.e., that $M=(M^{\wedge_\chi})_{fin}$, where $M:=\M_\hbar/\hbar\M_\hbar$.
Thanks to our choice of a good filtration on  $\M$, we see that $M$ is a free $\K[\g(>0)]$-module, say
$\K[\g(>0)]^{\oplus n}$. Then we can apply results of \cite[3.2]{HC} to see that $(M^{\wedge_\chi})_{P-fin}=\K[Pe]^{\oplus n}$. The codimension condition implies that $\K[Pe]=\K[\g(>0)]$ which yields $M=(M^{\wedge_\chi})_{fin}$.
\end{proof}

\begin{Rem}
The proof of the double centralizer property can also be deduced from results of Stroppel, \cite{Stroppel_ger},\cite[Theorem 10.1]{Stroppel_eng}.
Namely, her results imply that $\Vfun^0$ (that is a quotient functor killing all simples with
non-maximal GK dimension) has double centralizer property. We are going to deal with integral categories.
Similarly to the proof of Proposition
\ref{Prop:0_faith}, it is enough to check that if $\mathcal{P}$ is projective in $\OCat^P$, then all weight
spaces $\mathcal{P}_\mu$ are projectives in $\OCat^{P_0}$. If $\mathcal{P}$ is a dominant Verma, then this is checked
similarly to the parallel part of the proof of Proposition \ref{Prop:0_faith}. In general, $\mathcal{P}$ is a direct
summand in the tensor product of a dominant Verma and a finite dimensional $\g$-module. The claim that
all $\mathcal{P}_\mu$ are projective easily follows.
\end{Rem}

\subsubsection{Summary of properties}\label{SS_summary}
Here is the summary of our results describing the properties of the functor $\Vfun$.

\begin{Thm}\label{Thm:Vfun}
There is an exact functor $\Vfun:\OCat^P_\nu\rightarrow \OCat^\theta(\g,e)_\nu^R$ with the following properties:
\begin{itemize}
\item[(i)] The essential image of $\Vfun$ is closed under taking subquotients and   $\Vfun$ is a quotient onto its image.
The functor $\Vfun$ annihilates precisely
the modules  whose all $\t$-weight spaces have GK dimension less than $\dim\g_0(<0)$.
\item[(ii)] The functor $\Vfun$ intertwines the products with HC bimodules:
$\Vfun(X\otimes_\U \M)\cong X_\dagger\otimes_{\Walg}\Vfun(\M)$.
\item[(iii)] Let $\Delta^0$ be the induction functor $U(\g)\otimes_{U(\g_{\geqslant 0})}\bullet$.
Then $\Vfun(\Delta^0(\M^0))\cong\Delta^\theta_{\Walg}(\Vfun^0(\M^0))$ for $\M^0\in \OCat^{P_0}_\nu$. Here
$\Vfun^0: \OCat^{P_0}_\nu\rightarrow \OCat(\g_0,e)^R_\nu$ is an analog of $\Vfun$ for $(\g_0,e)$.
\item[(iv)] The dimension of $\Vfun^0(\M^0)$ coincides with the multiplicity of $\M_0$ (on $P_0e$).
In particular, the character of $\Vfun(\Delta_P(\mu))$ equals $\dim L_{00}(\mu)e^{\mu-\rho}\prod_{i=1}^k (1-e^{\mu_i})^{-1}$.
Here $\mu_1,\ldots,\mu_k$ are the weights of $\t$ on $\z_{\g}(e)_{<0}$.
\item[(v)] $\Vfun$ commutes with the naive duality: $\Vfun(\M^\vee)\cong\Vfun(\M)^\vee$.
\item[(vi)] The image of the simple $L(\mu)\in \mathcal{O}^P_\nu$ under $\Vfun$ equals $L^\theta_{\Walg}(\Vfun^0(L_0(\mu)))$.
In the case when $\nu=0$, the module $\Vfun^0(L_0(\mu))$ is computed as follows. Let $w$ be the element of the Weyl group $W_0$ corresponding to $\mu$. To $w$ we can assign the subgroup $H_0$ in the Lustzig quotient
$\bA_0$ associated to $(\g_0,e)$ and also an irreducible  $H_0$-module $\mathcal{V}$. Then
$\Vfun^0(L_0(\mu))$ is the homogeneous bundle over $\bA_0/H_0$ with fiber $\mathcal{V}\otimes \Nil^0$,
where $\Nil^0$ is the irreducible $\Walg^0$-module corresponding to the point $H_0\in \bA_0/H_0$.
\item[(vii)] The functor $\Vfun$ has double centralizer property, i.e., is fully faithful on the projective objects.
\item[(viii)] Assume that the codimension of $\g_0(>0)\setminus P_0e$ in $\g_0(>0)$ is bigger than $1$. Then
$\Vfun$ is $0$-faithful.
\end{itemize}
\end{Thm}

Everything but (vi) has already been proved. Let us prove (vi). The description of $\Vfun^0(L_0(\mu))$
follows from  \cite[Remark 7.7]{W_classif}, (ii), the Bernstein-Gelfand
equivalence, and the special case of (iii) describing  $\Vfun^0(\Delta_{P_0}(\rho))$ from
the proof of Proposition \ref{Thm:quot_prop_gen}. Namely, let $X^0$ be the Harish-Chandra $\U^0$-bimodule
corresponding to $L_0(\mu)$ under the parabolic Bernstein-Gelfand equivalence from \ref{SSS_BG_equiv}.
We have $\Vfun^0(L_0(\mu))\cong X^0_{\dagger^0}\otimes \Vfun^0(\Delta_{P_0}(\rho))$ by \ref{SSS_HC_products}.
But $\Vfun^0(\Delta_{P_0}(\rho))$ is the one-dimensional module $\Walg^0/(\J_{P_0})_{\dagger^0}$ as we have seen
in \ref{SSS_V23_iso}. The object $X^0$ is annihilated by $\J_{P_0}$ from the right and so $\Vfun^0(L_0(\mu))\cong X^0_{\dagger^0}$
as a left $\Walg^0$-module. Now we are in position to use \cite[Remark 7.7]{W_classif} and get the description of
$\Vfun^0(L_0(\mu))$ in (vi).

To get the description of $\Vfun(L(\mu))$
one can argue as follows. The module $\Vfun(L(\mu))$ is a quotient of $\Vfun(\Delta^0(L_0(\mu)))\cong\Delta^\theta_{\Walg}(\Vfun^0(L_0(\mu)))$.
The object $\Vfun(L(\mu))\in \OCat^\theta(\g,e)^R$ is simple thanks to (i). But $\Vfun^0(L_0(\mu))$
is a simple object in $\OCat^\theta(\g_0,e)^R$. So the object   $\Delta^\theta_{\Walg}(\Vfun^0(L_0(\mu)))
\in \OCat^\theta(\g,e)^R$ has simple head, $L^\theta_{\Walg}(\Vfun^0(L_0(\mu)))$, and therefore this head has to be isomorphic to $\Vfun(L(\mu))$.

We recall that it is parts (iv) and (vi) that allow us to compute the characters of the modules
$L^\theta_{\Walg}(\Nil^0)\in \OCat^\theta(\g_0,e)$. We have mentioned already that these characters
do not depend on the choice of $\Nil^0$ in the $A_0(e)=R/R^\circ$-orbit. This is because the
$A_0(e)$-action on the classes of irreducibles in $\OCat^\theta(\g_0,e)_\nu$ does not change the
character and the functor $L^\theta_{\Walg}$ intertwines the $A_0(e)$-actions on the classes of simples
in $\OCat^\theta(\g_0,e)_\nu$ and in $\OCat^\theta(\g,e)_\nu$.

Let us now summarize how the computation of the character $\operatorname{ch}L^\theta_{\Walg}(\Nil^0)$ works, step by step.

1) Let $W_0$ be the Weyl group of $\g_0$. Let $\varrho_0$ be the $\g_0$-dominant element representing the central character of $\Nil^0$ and let $c^0_{l}$ be the left cell in $W_0$ corresponding to $\Nil^0$, see Subsection \ref{SS_W_classif}.
The choice of a parabolic subalgebra $\p_0=\bigoplus_{i\geqslant 0}\g_0(i)$ defines a right cell $c^0_{r}$, see \ref{SSS_BG_equiv}.
Pick $w\in c^0_l\cap c^0_r$.

2)  Decompose the class of $L(w\varrho_0)$ in $K_0(\OCat_\nu^P)$ via the classes of parabolic Verma modules:
\begin{equation}\label{eq:char_par} L(w\varrho_0)=\sum_{u}c_{wu}\Delta^P(u\varrho_0).\end{equation}
Here  $u$ runs
over all elements in $W/W_{\varrho_0}$ such that $u\varrho_0$ is strictly dominant for $\g_0(0)$.
The numbers $c_{wu}$ are the parabolic Kazhdan-Lusztig coefficients.

3) Let $\bA_0$ be the Lusztig quotient for the two-sided cell in $W_0$ containing $c^0_l,c^0_r$ and $H_0$ be the
subgroup of $\bA_0$ corresponding to the left cell $c^0_l$. Then to $w$ viewed as an element in $W_0$
we can assign a pair $(x,\mathcal{V})$ of a point $x\in \bA_0/H_0$ and an irreducible $H_0$-module $\mathcal{V}$, see Subsection \ref{SS_HC_semis} (there we were dealing with triples $(x,y,\mathcal{V})$ but in the present situation $y$
is uniquely determined, see Remark \ref{Rem:right_cell}).

4) Let $\mathfrak{t}$ stand for the center of $\g_0$. Let $\mu_1,\ldots,\mu_k\in \mathfrak{t}^*$
denote the weights of $\mathfrak{t}$ in $\g_{<0}\cap \z_\g(e)$ (counted with multiplicities).
Then we have the following formula for $\operatorname{ch}L^\theta_{\Walg}(\Nil^0)$
$$\operatorname{ch}L^\theta_{\Walg}(\Nil^0)=\frac{|H_0|}{|\bA_0|\dim \mathcal{V}}\left(\sum_{u} c_{wu}e^{u\varrho_0-\rho}\dim L_{00}(u\varrho_0)\right)\prod_{i=1}^k(1-e^{\mu_i})^{-1}.$$
Here the range of summation is the same as in (\ref{eq:char_par}) and $L_{00}(u\varrho_0)$ is the
irreducible finite dimensional $\g_0(0)$-module with highest weight $u\varrho_0-\rho$.

We remark that the module $L^\theta_{\Walg}(\Nil^0)$ is finite dimensional if and only if $w\varrho_0$
has the form $w' \varrho$ for $w'$ lying in the two-sided cell corresponding to $\Orb$ and compatible
with a dominant weight $\varrho$.

\section{Goldie ranks}\label{S_Goldie}
In this section we fix a special orbit $\Orb\subset \g$ and take the W-algebra  $\Walg$ for that orbit.
Let $\dcell$ be the two-sided cell corresponding to $\Orb$. By $c_w$ we denote the left cell in $\dcell$
containing an element $w$.
\subsection{Reminders on Goldie ranks}\label{SUBSECTION_Goldie_rem}
In this subsection we will recall a few classical facts about Goldie ranks proved by Joseph.

First of all, following Joseph, we will consider some new algebras.
Namely, consider the algebra $L(L(w  \lambda),L(w \lambda))$ of $\g$-finite linear endomorphisms of
$L(w\lambda)$. It is known, see \cite[2.5]{Joseph_I} that
this algebra is prime and noetherian, so  has Goldie rank.
For a dominant integral weight $\lambda$ and compatible $w\in W$,  let $\gol_w(\lambda)$
denote the Goldie rank of $L(L(w\lambda),L(w\lambda))$.

Each left cell $c$ has a unique distinguished involution $d_c$ called a Duflo involution.
As Joseph proved in \cite[3.4]{Joseph_I}, for each dominant $\lambda$ compatible with $d_c$,
we have \begin{equation}\label{eq:Goldie_Duflo}\Goldie(\U/J(d_c \lambda))=\gol_{d_c}(\lambda).\end{equation} So the collection $\gol_w(\lambda)$
contains all Goldie ranks that we have originally wanted to compute.

According to \cite[Corollary 5.12]{Joseph_I} there is a polynomial
$q_w(\lambda)$ such that $\gol_w(\lambda)=q_w(\lambda)$ for all dominant $\lambda$ compatible with $w$
(we remark that our notation here is different from Joseph's; we write $q_w(\lambda)$ for what Joseph in
\cite{Joseph_II} would denote $\tilde{q}_w(\lambda)$; for Joseph $q_w(\lambda)=\gol_w(w^{-1}\lambda)$).
In particular, we have $\Goldie(\U/J(w\lambda))=p_w(\lambda)$, where $p_w(\lambda):=q_{d_c}(\lambda), w\in c$. The polynomials $p_w(\lambda)$ are called {\it Goldie rank polynomials}.
Furthermore, the quotient $q_w(\lambda)/p_w(\lambda)$
is a positive integer independent of $\lambda$. This positive integer is known as
Joseph's {\it scale factor} and is denoted by $z_w$. In fact, below we will only need
to know that $z_w$ is a number bigger than or equal to $1$ (the equality $\Goldie(\U/J(w\lambda))\leqslant \gol_w(\lambda)$
is a consequence of the inclusion $\U/J(w\lambda)\subset L(L(w\lambda),L(w\lambda))$
that is provided by the $\U$-action on $L(w\lambda)$ -- such an inclusion implies the inequality, see
\cite{Warfield}; but the claim that the ratio
is independent of $\lambda$ is not so easily seen). A crucial property of the polynomials
$p_w(\lambda)$ with $w\in \dcell$ is that their span is an irreducible $W$-submodule of $\K[\h^*]$,
see \cite[Theorem 5.5]{Joseph_II}. This submodule  is isomorphic to the special $W$-module
in $\Irr(W)^\dcell$, \cite[14.15]{Jantzen}. Moreover, if we choose elements $w_1,\ldots,w_k$, one in each left cell of $\dcell$,
then the polynomials $p_{w_1},\ldots, p_{w_k}$ form a basis in the submodule.

In \cite[Theorem 5.1]{Joseph_II}  Joseph determined the polynomial $p_w$ up to a scalar multiple.
So to complete the Goldie ranks computation one needs to determine a collection of scalars, say $s_c$, one for each
left cell $c$. As Joseph, basically, pointed out in \cite[Remark 1 in 5.5]{Joseph_II}, if one knows the scale factors
$z_w$ for all $w\in \dcell$, then one can, in principle, determine the scalars $s_c$ for all left cells $c\subset \dcell$
up to a common scalar multiple. Let us explain how this works.

In \cite[5.5, Remark 1]{Joseph_II}, Joseph finds a formula expressing $y.q_w, y\in W, w\in \dcell$ as a linear
combination of elements $q_{w'}$ with $w'\in \dcell$. The coefficients are expressed in terms of the multiplicities
in the BGG category $\mathcal{O}$ and so are known. If one knows the coefficients $z_w$ for all $w\in\dcell$,
then one can express $y.p_{w_i}$ in terms of the $p_{w_j}, j=1,\ldots,k$, say
\begin{equation}\label{eq:lin_comb}
y.p_{w_i}=\sum_{j=1}^k b_j^i(y) p_{w_j}.
\end{equation}
But the elements $p_{w_j}$ are linearly independent and their span is an irreducible $W$-module
so the basis $(p_{w_i})_{i=1}^k$ is determined uniquely from (\ref{eq:lin_comb}) up to a scalar multiple.

That single multiple can be determined uniquely  if one knows that there is $w\in \dcell$ such that
$p_w(\lambda)=1$ for some integral weight $\lambda$. The latter happens if there are compatible $w\in \dcell$
and dominant integral $\lambda$  such that  $J(w\lambda)$ is completely prime, i.e.,  $\Goldie(\U/J(w\lambda))=1$.
In fact, there is a conjecture of Joseph saying that for each $w\in \dcell$ there is $\lambda$
with $p_w(\lambda)=1$ (of course, generally, $\lambda$ will not be dominant).

To finish this subsection let us mention that Joseph also had a conjecture computing the scale factors
$z_w$, see \cite[5.3]{Joseph_Duflo}.  It is unclear to us whether his conjecture is compatible
with Theorem \ref{Thm:main_appendix}.

\subsection{Scale factors: Joseph vs Premet}\label{SS_scale}
Let $w\in \dcell$. The goal of this subsection is to provide a formula for Joseph's scale factors $z_w$, see Subsection \ref{SUBSECTION_Goldie_rem},
in  terms of the triple $x,y,\mathcal{V}$ corresponding to $w$, see Subsection \ref{SS_HC_semis},
and certain numbers that we call {\it Premet's scale factors}\footnote{The attribution to Premet
is made because of his beautiful result saying that this scale factor is always integral.
The first version of our proof below used that result. In fact --
as was communicated to the author by Premet -- one does not need  that
in the proof, it is enough to use the fact that the scale factor is bigger or equal
than 1, which was established by the author. It is pleasant when somebody else
understands your work better than you do...} that are given by
$$\operatorname{pr}_w(\lambda)=\frac{d_w(\lambda)}{p_w(\lambda)}.$$
Here $\lambda$ is a dominant weight compatible with $w$, $p_w(\lambda)$ is the Goldie rank of $J(w\lambda)$,
and $d_w(\lambda)$ is the dimension of the irreducible $\Walg$-module lying over $J(w\lambda)$.
Of course, $\operatorname{pr}_w(\lambda)$ depends only on the left cell $c_w$ containing $w$ and on $\lambda$
(below we will see that it is actually independent of $\lambda$).

Pick a regular element $\varrho\in \Lambda^+$.
Recall that we represent $w\in \dcell$ as a triple $(x,y,\mathcal{V})$, see Subsection \ref{SS_HC_semis}.
Below we write $\K$ for the trivial $\bA_{(x,y)}$-module.
Let $Y$ denote the subset of $Y^{\Lambda}$ consisting of all irreducibles with central character $\varrho$.
Below for $x\in Y$ lying over a primitive ideal $\J=J(w\varrho)$ we write $d_x:=d_w(\varrho),\gol_x:=p_w(\varrho),
\operatorname{pr}_x:=\operatorname{pr}_w(\varrho)$.

\begin{Prop}\label{Thm:scale}
We have $$z_w=\frac{\operatorname{pr}_x}{\operatorname{pr}_y}\cdot\frac{|\bA_y|}{|\bA_{x,y}|}\dim \mathcal{V}.$$
\end{Prop}

\begin{proof}
Set $$\M_w(\lambda):=L(\Delta(\varrho),L(w\lambda)),
m_w(\lambda):=\operatorname{mult}_{\Orb}\M_{w}(\lambda),$$
and $\gol_w(\lambda):=\Goldie(L(L(w\lambda), L(w\lambda)))$.
Below we only consider $\lambda$ compatible with $w$.

\begin{Lem}\label{Prop:proportionality}
The ratio
$\frac{m_w(\lambda)}{\gol_w(\lambda)}$ depends only on the right cell
containing $w$ (equivalently, on $c_{w^{-1}}$).
\end{Lem}
\begin{proof}[Proof of Lemma \ref{Prop:proportionality}]
We have $m_w(\lambda)=e(\M_w(\lambda))/e(\K[\Orb])$, where $e$ is the Gelfand-Kirillov multiplicity,
for the definition of $e(\bullet)$ see, say, \cite[Kapitel 8]{Jantzen}. Now the result follows
from \cite[12.5]{Jantzen}.
\end{proof}

Below we write $m_w,\gol_w$
for $m_w(\varrho),\gol_w(\varrho)$. We write $z_{x,y,\mathcal{V}}, m_{x,y,\mathcal{V}}$, etc.,
instead of $z_w,m_w$ etc.  if $w$ corresponds to a triple $(x,y, \mathcal{V})$.

By the definition of $z_\bullet$ we have
\begin{equation}\label{eq:z}
z_{x,y,\mathcal{V}}=\frac{\gol_{x,y,\mathcal{V}}}{\gol_x}.
\end{equation}
(\ref{eq:Goldie_Duflo}) together with the observation that a Duflo involution
corresponds to a triple of the form $(x,x,\K)$, see Subsection \ref{SS_HC_semis},
imply $\gol_x=\gol_{(x,x,\K)}$.
By Lemma \ref{Prop:proportionality}, the ratio $\frac{m_{x,y,\mathcal{V}}}{\gol_{x,y,\mathcal{V}}}$ depends only on the left cell corresponding to $y$. In particular,
\begin{equation}\label{eq:proportion}
\frac{m_{x,y,\mathcal{V}}}{\gol_{x,y,\mathcal{V}}}=\frac{m_{y,y,\K}}{\gol_{y,y,\K}}.
\end{equation}
Plugging (\ref{eq:mult_equality}),(\ref{eq:z}) into (\ref{eq:proportion}) we get
\begin{align*}
&\frac{d_x d_y |\bA|\dim \mathcal{V}}{z_{x,y,\mathcal{V}}\gol_x|\bA_{x,y}|}=\frac{d_y^2 |\bA|}{\gol_y|\bA_y|}\Rightarrow\\
&z_{x,y,\mathcal{V}}=\frac{d_x d_y |\bA|\dim \mathcal{V} \gol_y |\bA_y|}{d_y^2 |\bA|\gol_x|\bA_{x,y}|}=
\frac{\operatorname{pr}_x}{\operatorname{pr}_y}\frac{|\bA_y|}{|\bA_{x,y}|}\dim \mathcal{V}.
\end{align*}

\end{proof}

\subsection{Proof of Theorem \ref{Thm:main_appendix} modulo Conjecture \ref{Conj:main}}\label{SS_Thm_Goldie_proof}

\begin{Prop}\label{Prop:scale_equal}
Let $w,\lambda$ be compatible.
Then  $\pr_w(\lambda)$ depends only on the left cell containing $w$.
\end{Prop}
\begin{proof}
The multiplicity $\operatorname{mult}_{\Orb}(\U/J(w\lambda))$
is equal to  some polynomial $P_w(\lambda)$ and this polynomial is proportional to $p_w(\lambda)^2$, see \cite[12.7]{Jantzen}.
But, on the other hand, according to (\ref{eq:mult_equality}), we have $\operatorname{mult}_{\Orb}(\U/J(w\lambda))=
|\bA/\HL_{\cell_w}|d_w(\lambda)^2$. So $d_w(\lambda)=\tilde{p}_w(\lambda)$
for some polynomial $\tilde{p}_w(\lambda)$ proportional to $p_w(\lambda)$. This implies the claim.
\end{proof}

For a cell $\cell:=\cell_{w}$ we will write $\operatorname{pr}_{\cell}:=\pr_w(\lambda)$.

\begin{proof}[Proof of Theorem \ref{Thm:main_appendix} modulo Conjecture \ref{Conj:main}]
Consider the case when the orbit $\Orb$ is not one of the three exceptional orbits.

Suppose that $x,y\in Y$ are such that
$\bA_x\supset \bA_y$ and let  $\dim \mathcal{V}=1$. We have $|\bA_y|=|\bA_{x,y}|$ and so Proposition \ref{Thm:scale} implies $z_w=\frac{\operatorname{pr}_x}{\operatorname{pr}_y}$. In particular,
since $z_w\geqslant 1$, we have
$\operatorname{pr}_x\geqslant \operatorname{pr}_y$. According to \cite[Proposition 3.4.6]{wquant},
$\operatorname{pr}_y\geqslant 1$. So if $\pr_x=1$ and $\bA_x=\bA$, then $\pr_y=1$
for all $y\in Y$.

Thanks to Conjecture \ref{Conj:main}, there is a (possibly singular) dominant integral weight $\lambda$ and $w\in \dcell$ as
in Theorem \ref{Thm:main_appendix} such that $d_w(\lambda)=1$ (and hence $\operatorname{pr}_w(\lambda)=1$) and the
$\bA$-orbit corresponding to $J(w\lambda)$ is a single point. Then the  $\bA$-orbit corresponding to $J(w\varrho)$
is also a single point. Let $\cell$ be the left cell containing $w$. Then, thanks
to Proposition \ref{Prop:scale_equal}, $\operatorname{pr}_{\cell}=\operatorname{pr}_w(\lambda)=1$. Now it remains to use the result
of the previous paragraph.

If $\Orb$ is one of the three exceptional orbits, then $\bA_x=\{1\}$ for all $x\in Y$,
this follows from \cite[6.7,Theorem 1.1]{W_classif}.
One carries over the argument above to this case without any noticeable modifications.
\end{proof}

\subsection{Proof of Theorem \ref{Thm:classical}}\label{SECTION_main_conj}
\subsubsection{Reduction to weakly rigid orbits}\label{SUBSECTION_weak_reduction}
We start the proof of Theorem \ref{Thm:classical} by introducing a certain induction
procedure related to the Lusztig-Spaltenstein induction.

Let us  define a certain class of special nilpotent orbits:
{\it weakly rigid} ones\footnote{These are close to  so called {\it birationally
rigid} orbits.}. Then we will show that it is enough to prove Conjecture
\ref{Conj:main} for weakly rigid orbits only.

Recall that $\Orb$ is called {\it rigid} if it cannot be induced from a nilpotent orbit in
a proper Levi subalgebra.

\begin{defi}\label{def:weakly_rigid}
We say that a special orbit $\Orb$ is {\it strongly induced} if there is a proper parabolic subalgebra
$\p\subset \g$, Levi subalgebra $\underline{\g}\subset\p$ and a nilpotent orbit $\underline{\Orb}\subset \underline{\g}$ such that
$\Orb$ is induced from $\underline{\Orb}$ in the sense of Lusztig and Spaltenstein and moreover:
\begin{enumerate}
\item $\underline{\Orb}$ is special in $\underline{\g}$.
\item Pick $e\in \Orb\cap(\underline{\Orb}+\n)$. Then the projection
of $Z_P(e)$ to the Lusztig quotient $\bA$ (independent of the choice of $e$) for $\Orb$ is surjective.
\end{enumerate}
We say that $\Orb$ is {\it weakly rigid}, if it is not strongly induced.
\end{defi}

We are going to use the parabolic induction for W-algebras recalled in Subsection \ref{SS_parab_ind}.
In particular, let $\underline{N}$ be a 1-dimensional $\underline{\Walg}$-module with integral
central character. It follows from \cite[Corollary 6.3.3]{Miura} that $\varsigma(\underline{N})$ also
has integral central character.  Now assume in addition that $\underline{N}$ is
$A(\underline{e})$-stable, where $\underline{e}\in \underline{\Orb}$
is the projection of $e$ along $\n$ and $A(\underline{e})$ stands for the component
group of the centralizer of $\underline{e}$ in $\underline{G}$.
Then both  $\underline{N},\varsigma(\underline{N})$ are $\underline{Q}$-stable. So
we see that if condition (2) holds, then $\varsigma(\underline{N})$ is $A(e)$-stable.

Summarizing, we see that   it is enough to prove Conjecture \ref{Conj:main} only for weakly rigid orbits.
We will provide a complete proof only in type $B$. Then we explain modification necessary
for types $C,D$.

\subsubsection{Type $B$}
We start by recalling a few standard facts about nilpotent orbits in $\so_{2n+1}$.

First of all, nilpotent orbits in $\so_{2n+1}$ are parameterized by partitions $\lambda$
of $2n+1$ having type $B$, i.e.,  such that every even part appears with even multiplicity, see, for example, \cite[5.1]{CM}.
In general, for any partition $\lambda$ of $2n+1$ there is the largest (with respect
to the dominance) partition $\lambda_B$ of type $B$ smaller than or equal to $\lambda$. The partition
$\lambda_B$ is called the {\it $B$-collapse} of $\lambda$.

A partition $\lambda$ corresponds to a special orbit if and only if the transposed partition $\lambda^t$ is again of type $B$,
see \cite[Proposition 6.3.7]{CM}.
Explicitly, this means that there is an even number of odd parts between any two consecutive
even parts or smaller than the smallest even part, but there is an odd number of odd parts
larger than the largest even part.

Now let us recall what the Lusztig-Spaltenstein induction does on the level of partitions, see \cite[7.3]{CM}.
Any Levi subalgebra in $\so_{2n+1}$ has the form $\gl_{n_k}\times \gl_{n_{k-1}}\times\ldots \gl_{n_1}\times \so_{2n_0+1}$ with $\sum_{i=0}^k n_i=n$. The induction procedure is associative so it is enough to see what happens when we induce from an orbit $\underline{\Orb}=(0,\Orb_0)\subset \gl_{m}\times \so_{2(n-m)+1}$. Let $\mu=(\mu_1,\ldots,\mu_l)$ with $l\geqslant m+1$ be the partition of $\Orb_0$ (we add zero parts
if necessary). Then the partition $\lambda$ corresponding to the induced orbit $\Orb$  is
\begin{itemize}
\item[(a)]
either $(\mu_1+2,\ldots, \mu_{m}+2,\mu_{m+1},\ldots,\mu_{l})$ if the latter orbit is of type $B$,
\item[(b)]
or $(\mu_1+2,\ldots,\mu_{m-1}+2,\mu_{m}+1,\mu_{m+1}+1,\mu_{m+2},\ldots,\mu_l)$ otherwise.
\end{itemize}
In other words, $\lambda$ is always the $B$-collapse of the partition in (a).

\begin{Lem}\label{Lem:cond2}
Suppose that the partition of $\Orb$ is obtained as in (a). Then the conditions of Definition
\ref{def:weakly_rigid} are satisfied.
\end{Lem}
\begin{proof}
The claim that (1) is satisfied is straightforward from the combinatorial description of special orbits.

We will specify the choice of a parabolic $P$ and then choose an element $e$ as in (2) to compute
$Z_P(e)$ and see that actually $Z_P(e)$ projects surjectively to $A(e)$.

We represent  $\so_{2n+1}$ as the Lie algebra of matrices that are skew-symmetric with
respect to the main anti-diagonal (so that the symmetric form used to define $\so_{2n+1}$
is $(x,y)=\sum_{i=1}^{2n+1}x_i y_{2n+2-i}$). Choose the parabolic subalgebra $\p$ with Levi
subalgebra $\gl_m\times \so_{2(n-m)+1}$ in a standard way, i.e., $\p$ is the stabilizer of the span
of the first $m$ basis elements.

Let us specify an element $e_0\in \Orb_0$. Consider the numbers $\mu_1,\ldots,\mu_m$ and split them
into $q$ pairs of equal numbers and $p$ pairwise different numbers. Then take the remaining
numbers $\mu_{m+1},\ldots, \mu_l$ and do the same getting $q'$ pairs and $p'$ pairwise different numbers.
E.g., for $m=6, \mu=(5^3,4^2,3^4,2^2,1^2)$ (as usual the superscripts are the multiplicities) we have $q=2$ (with pairs $(5,5),(4,4)$), $p=2$ (with numbers $5$ and $3$),
$q'=3$ (the pairs $(3,3),(2,2)$ and $(1,1)$) and $p'=1$ (corresponding to $3$).

Take the subspace $\K^{2(n-m)+1}\subset \K^{2n+1}$, where $\so_{2(n-m)+1}$ acts, and represent it as a direct sum of subspaces $$\bigoplus_{i=1}^p V_i\oplus \bigoplus_{i=1}^q(U_i\oplus U_i^*)\oplus \bigoplus_{i=1}^{p'}V'_i\oplus \bigoplus_{i=1}^{q'}(U'_i\oplus U_i'^*),$$ where $V_1,\ldots, V_p, V_1',\ldots, V'_{p'}$ are orthogonal subspaces of dimensions equal to single $\mu_i$'s ($5,3,3$ in our example), and $U_1,\ldots, U_q, U_1',\ldots, U'_{q'}$ are isotropic subspaces
of dimensions equal to  $\mu_i$'s from pairs ($5,4,3,2,1$ in our example), $U_1^*,\ldots, U_q^*, U_1'^*,\ldots,
U'^*_{q'}$ are dual isotropic subspaces. Below we denote by $v_1(i),\ldots,v_{2d+1}(i)$ a basis
in  $V_i$, where the form is written as above, by $u_1(j),\ldots,u_{d}(j)$ a basis in $U_j$, and by
$u^*_1(j),\ldots, u^*_d(j)$ the dual basis in $U_j^*$. By $v_{m}^n(i)$ we denote the matrix unit sending $v_n(i)$
to $v_m(i)$. The notation $u_m^n(i), u_m^{*n}(i)$ has a similar meaning.

For $e_0$ we take $$\sum_{k=1}^p e_{0}(k)+\sum_{k=1}^q f_{0}(k)+
\sum_{k=1}^{p'} e'_{0}(k)+\sum_{k=1}^{q'} f'_{0}(k).$$ The element $e_{0}(k)$ is given
by the matrix  $\sum_{i=1}^{d} (v_i^{i+1}(k)-v_{i+d}^{i+d+1}(k))$, where $\dim V_k=2d+1$.
The element $f_{0}(k)$ is given by the matrix $\sum_{i=1}^d u_{i}^{i+1}(k)-\sum_{i=1}^d {u^*}_{i}^{i+1}(k)$, where $d=\dim U_i$. The operators
$e_{0}'(k)\in \mathfrak{so}(V_k'),f_{0}'(k)\in \mathfrak{so}(U_k'\oplus U_k'^*)$ are defined similarly. It is clear from the construction that $e_0\in \underline{\Orb}$.

Now let us specify $e$. Set $\widetilde{V}_i:=V_i\oplus \K^2$, where $\K^2$ is viewed as an orthogonal
space with isotropic basis $v_0(i),v_{2d+2}(i)$, and $\widetilde{U}_i:=U_i\oplus \K^2$, where $\K^2$ is viewed as an isotropic space with basis $u_0(i),u_{d+1}(i)$.
Then, of course $$\K^{2n}=\bigoplus_{k=1}^p \widetilde{V}_k\oplus \bigoplus_{k=1}^q (\widetilde{U}_k\oplus \widetilde{U}_k^*)\oplus \bigoplus_{k=1}^{p'}V_k'\oplus \bigoplus_{k=1}^{q'}(U_k'\oplus U_k'^*).$$
Now we set $$e:=\sum_{k=1}^p e(k)+\sum_{k=1}^q f(k)+
\sum_{k=1}^{p'} e'_{0}(k)+\sum_{k=1}^{q'} f'_{0}(k),$$
where the matrices $e(k),f(k)$ are defined as follows. We set $e(k):=\sum_{i=0}^{d}(v_{i}^{i+1}(k)-v_{i+d+1}^{i+d+2}(k))$. Further, $f(k)=\sum_{i=0}^{d+1}u_{i}^{i+1}(k)-\sum_{i=0}^{d+1}{u^*}_{i}^{i+1}(k)$. It is clear that $e\in \Orb$.

Let $P$ be the maximal parabolic subgroup in $\operatorname{O}(2n+1)$ (note that we take a disconnected group)
stabilizing the $m$-dimensional subspace spanned by all basis vectors $v_0(k),u_0(k),u^*_0(k)$.
With this choice of $P$ we have $e\in e_0+\n$.  Now let us produce elements in the Levi subgroup $\operatorname{GL}_m
\times \operatorname{O}_{2(n-m)+1}$ centralizing $e$ whose images span the component group of $Z_{\operatorname{O}_{2n+1}}(e)$
(this will imply condition (2)). This component group is the sum of several copies of
$\ZZ/2\ZZ$, one for  each different odd part of $\lambda$, see, for example, \cite[Theorem 6.1.3]{CM}. More precisely, let $\lambda=(\lambda_1^{n_1},\lambda_2^{n_2},\ldots, \lambda_s^{n_s})$. Then the reductive part of $Z_{\Ort_n}(e)$ is $\prod_{i=1}^k G_{n_i}$,
where $G_{n_i}$ means $\Ort_{n_i}$ if $\lambda_i$ is odd, and $\operatorname{Sp}_{n_i}$ if $\lambda_i$
(and then automatically $n_i$) is even. We have one involution in the basis of the component group for each
$G_{n_i}=\Ort_{n_i}$.

Let us produce elements $g_i\in
\Ort(\widetilde{V}_i)\cap (\GL_m\times \Ort_{2(n-m)+1})$ centralizing $e(i)$, and
$h_j\in \Ort(\widetilde{U}_j\oplus \widetilde{U}^*_j)\cap (\GL_m\times \Ort_{2(n-m)+1})$ centralizing
$f(j)$. For $g_i$ we just take $-\operatorname{id}_{\widetilde{V}_i}$. If $\dim U_j$ is even, then
we set $h_j=\operatorname{id}_{\widetilde{U}_j}$. Finally, suppose $\dim U_j$ is odd.
Then define $h_j$ by $h_j(u_k(j)):=\sqrt{-1}u^*_k(j), h_j(u^*_k(j))=-\sqrt{-1}u_k(j)$.
It is easy to see that $h_j$ defined in this way lies in the required subgroup and centralizes $f(j)$.

From the description of the component group given above we see that
the elements $g_i,h_j$ generate the component group.
\end{proof}

From Lemma \ref{Lem:cond2} we deduce that the partition of a weakly rigid orbit
has the form \begin{equation}\label{eq:weakly_rigid_partn}((2k+1)^{2d_{2k+1}-1},(2k)^{2d_{2k}},(2k-1)^{2d_{2k-1}},\ldots,2^{2d_2}, 1^{2d_1}),\end{equation}
where $d_{2k+1},d_{2k},\ldots,d_1$ are positive integers. It seems that any partition like
this indeed corresponds to a weakly rigid orbit, but we will not need this.

Let us produce an explicit  $\mu\in \Lambda$ such that $\operatorname{mult}_{\Orb}\U/J(\mu)=1$.
According to (\ref{eq:mult_equality}) for a corresponding irreducible $\Walg$-module $x$
we will have $\bA_x=\bA, d_x=1$ as needed in Conjecture \ref{Conj:main}.
This will be a so called Arthur-Barbasch-Vogan  weight, see, for example,
\cite[(1.15)]{BV_unip}.  Let us recall how this weight is constructed in general.

There is a duality for nilpotent orbits in $\g$ and in the Langlands dual Lie algebra $\g^\vee$
($B_n$ and $C_n$ are dual to each other, while all other simple algebras are self-dual). This duality (called Barbasch-Vogan-Spaltenstein duality) is an order reversing bijection between the sets of
special orbits in $\g$ and in $\g^\vee$. Take a special orbit  $\Orb\subset\g$ and let $\Orb^\vee\subset\g^\vee$ be the corresponding
dual orbit. Let $(e^\vee,h^\vee,f^\vee)$ be the corresponding $\sl_2$-triple.
Recall that we have fixed Cartan and Borel subalgebras $\h\subset\bfr\subset\g$. Then we can
take $\h^\vee:=\h^*$ for a Cartan subalgebra in $\g^\vee$ and also we have a preferred choice of a Borel
subalgebra $\bfr^\vee\subset \g^\vee$. Conjugating $h^\vee$
we may assume that $h^\vee$ is dominant in $\h^*$. This determines $h^\vee$ uniquely.
The weight of interest is $\mu:=\frac{1}{2}h^\vee$.

For the classical Lie algebras the duality between nilpotent orbits can be described
on the level of partitions, see, for example, \cite[6.3]{CM}. For example, take $\Orb\subset \so_{2n+1}$ and let $\lambda$
be the corresponding partition. For a partition $\mu$ of $2n+1$, let $l(\mu)$ denote the
partition of $2n$ obtained from $\mu$ by decreasing the smallest part of $\mu$ by $1$.
Recall that the nilpotent orbits in $\sp_{2n}$ are parameterized by partitions of type $C$,
i.e., such that the multiplicity of each odd part is even. For any partition $\mu'$ of $2n$
we can define its $C$-collapse $\mu'_C$ that is a partition of type $C$ similarly to
the $B$-collapse. Now the duality sends the orbit $\Orb$ with partition $\lambda$
to the orbit $\Orb^\vee$ with partition $[l(\lambda^t)]_C$, see, for instance, \cite[Theorem 5.1]{McGovern},
and the discussion after it. It is easy to see that, for $\lambda$ of the form (\ref{eq:weakly_rigid_partn}),
the partition $[l(\lambda^t)]_C$ consists
of even parts and so the orbit $\Orb^\vee$ is even, meaning that all eigenvalues of
$\operatorname{ad}(h^\vee)$ are even. In particular, $\frac{1}{2}h^\vee\in \Lambda^+$.

According to \cite[Proposition 5.10]{BV_unip}, $\VA(\U/J(\frac{1}{2}h^\vee))=\overline{\Orb}$.
Further, \cite[Corollary 5.19]{McGovern} implies that the multiplicity of $\U/J(\frac{1}{2}h^\vee)$
on $\Orb$ is 1.

\subsubsection{Type $C$}
The proofs of Conjecture \ref{Conj:main} in types $C$ and $D$ are very similar. So we will only explain
the necessary modifications.

As we have already mentioned, the nilpotent orbits in type $C$ are parameterized by partitions of type $C$.
A partition $\lambda$ corresponds to a special orbit if and only if $\lambda^t$ is again of type $C$.
Explicitly this means that there is an even number of odd parts between two consecutive even parts
and also an even number of even parts larger than the largest odd part. On the level of partitions
the Lusztig-Spaltenstein induction is described completely analogously to type $B$. Also
one can prove a direct analog of Lemma \ref{Lem:cond2} and the proof basically repeats
the original one. From here we see that the partition of a weakly rigid orbit has the form
$(n^{d_n},(n-1)^{d_{n-1}},\ldots 1^{d_1})$, where $d_1,\ldots,d_n$ are positive even integers.

On the level of partitions the duality is described as follows. For a partition $\lambda$ of $2n$
let $r(\lambda)$ be the partition of $2n+1$ obtained from $\lambda$ by increasing the largest part by $1$.
Now a partition $\lambda$ of a special orbit $\Orb\subset \sp_{2n}$ is sent to $[r(\lambda^t)]_B$,
where, recall, the subscript
$B$ means the $B$-collapse. It is easy to see that all parts of $[r(\lambda^t)]_B$ are odd, meaning
that the corresponding orbit is even. So it remains to take the ideal $J(\frac{1}{2}h^\vee)$ and use the results of Barbasch-Vogan and of McGovern, exactly as in type B.

\subsubsection{Type $D$}
Nilpotent orbits in $\so_{2n}$ are parameterized by partitions of $2n$ of type $D$:
each even part occurs even number of times (in fact, this is only true for the $\Ort_{2n}$-action,
each partition with all parts even produces exactly two $\operatorname{SO}_{2n}$-orbits,
while all other partitions correspond to a single $\operatorname{SO}_{2n}$-orbit).
A partition $\lambda$ corresponds to a special orbit  if and only if $\lambda^t$ is of type $C$ (not $D$!), i.e., there is an even number of odd parts between any two consecutive even parts, larger than the largest even part, and smaller than
the smallest even part.

The Lusztig-Spaltenstein induction is described in the same way as above, and an analog of Lemma
\ref{Lem:cond2} with the same proof holds. From here we deduce that the partition of a weakly
rigid orbit has the form $(n^{d_n},(n-1)^{d_{n-1}},\ldots 1^{d_1})$, where $d_1,\ldots,d_n$
are positive even integers.

The duality sends an orbit with partition $\lambda$ to the orbit with partition $[\lambda^t]_D$, where
the subscript $D$ means the $D$-collapse, i.e., the largest partition of type $D$ smaller
than a given one. It is easy to see that for $\lambda$ as specified in the previous paragraph,
all parts of $[\lambda^t]_D$ are odd. In particular, the dual orbit is even, the weight
$\frac{1}{2}h^\vee$ is integral, and we are again done.

\subsubsection{A few remarks towards the exceptional cases}
We hope that for the exceptional Lie algebras the same strategy as above should work.
Namely, one should be able to describe all weakly rigid orbits. Then, hopefully,
the dual of a weakly rigid orbit will be even (this is true for all rigid orbits,
as was checked in \cite[Proposition I]{BV_unip})\footnote{As I learned from Jeffrey Adams there
is only one weakly rigid orbit, where this is not true: it appears in type $E_8$}. Then one can hope that $\U/J(\frac{1}{2}h^\vee)$
will have multiplicity 1 on $\Orb$.

In \cite{Miura} we have developed some techniques to classify one-dimensional
representations of W-algebras under certain conditions on the nilpotent element
$e$. The most important condition is that the algebra $\q=\mathfrak{z}_{\g}(e,h,f)$
is semisimple. This condition is satisfied for all rigid orbits but not for all weakly rigid
ones (it is peculiar that, in the classical types, a weakly rigid orbit is rigid
precisely when $\q$ is semisimple, this can be deduced from the classification of
rigid orbits provided, for instance, in \cite[7.3]{CM}). Another condition on $e$ that significantly
simplifies the classification is that $e$ is principal in some Levi. This holds
for all weakly rigid (special) orbits in classical algebras.

The classification result obtained in \cite[Section 5]{Miura} (for elements $e$ subject to the conditions
explained in the previous paragraph) takes the following form. It establishes a family $X$ of
elements in $\h^*$ (given by several conditions, the most implicit being that
$\VA(\U/J(\lambda))=\overline{\Orb}$ for each $\lambda\in X$) such that $X$ is in bijection with 1-dimensional
$\Walg$-modules in such a way that the primitive ideal corresponding to the 1-dimensional module
attached to $\lambda\in X$ is $J(\lambda)$. For several rigid special elements the element $\frac{1}{2}h^\vee$
is in $X$, but for some it is not, which, however, does not mean that $J(\frac{1}{2}h^\vee)$ does not
correspond to a 1-dimensional $\Walg$-module simply because the weight cannot be recovered
from an ideal uniquely.

Finally, let us remark that even if $\q$ is not semisimple, it is still  sometimes possible to get some ramification
of the classification result that will classify $A(e)$-stable 1-dimensional modules, see \cite[Corollary 5.2.2]{Miura}.

\section{Supplements}\label{S_suppl}
\subsection{Functor $\bullet_{\dagger,e}$ for Harish-Chandra modules}\label{SS_HC_mod}
\subsubsection{Setting}\label{SSS_HC_mod_setting}
Our setting in this subsection is very different from the one in the main body of the paper.
We fix an involutive antiautomorphism $\tau$ of $G$ (so that $x\mapsto \tau(x)^{-1}$ is an involutive
automorphism). We set $K:=\{g\in G| \tau(g)=g^{-1}\}^\circ$, this is a reductive subgroup of $G$.
Further set $\mathfrak{s}:=\g^{\tau}$ so that we have $\g=\k\oplus\mathfrak{s}$.

We pick $e\in \mathfrak{s}$. Then  we can choose $h,f$ forming an $\sl_2$-triple with $e$ in such a way
that $h\in \k$ and $f\in \mathfrak{s}$. It is known (and easy to check) that the intersection of $Ge$
with $\mathfrak{s}$ is a union of finitely many $K$-orbits, each being a lagrangian subvariety in
$Ge$. As before, we take $V=[\g,f]$. Then $\mathfrak{u}:=\k\cap V$ coincides with $[\mathfrak{s},f]$ and is a lagrangian
subspace in $V$. Further, $\chi=(e,\cdot)$ clearly vanishes on $\k$. So the pair $(K,V)$
does satisfy the assumptions of \ref{SSS_Cat_OK_def}.

So we can get a functor $\bullet_{\dagger,e}:\OCat^K\rightarrow \Walg$-$\operatorname{mod}^{R}$. In fact,
the functor is easier to construct then in the general case and we also can get some restriction on
its image, it consists of ``HC-modules for $(\Walg,\tau)$''. We are going to define those in the next
part in a more general context.

\subsubsection{Harish-Chandra modules for almost commutative algebras}
Let $\A$ be a unital associative algebra equipped with an exhaustive algebra
filtration $\K=\A_{\leqslant 0}\subset \A_{\leqslant 1}\subset\ldots$. Assume that
$[\A_{\leqslant i},\A_{\leqslant j}]\subset \A_{\leqslant i+j-d}$ for some positive
integer $d$. Finally, assume that $\gr\A$ is a finitely generated algebra.

Suppose that $\A$ is equipped with an involutive antiautomorphism $\tau$ that preserves the filtration.
By definition, a HC $(\A,\tau)$-module is an $\A$-module $\M$ that can be equipped
with an increasing exhaustive filtration $\M_{\leqslant 0}\subset \M_{\leqslant 1}\subset\ldots$ that is
compatible with the filtration on $\A$ (in the sense that $\A_{\leqslant i}\M_{\leqslant j}\subset \M_{\leqslant i+j}$)
satisfying the following two additional conditions:
\begin{itemize}
\item[(i)] $\gr\M$ is a finitely generated $\gr\A$-module.
\item[(ii)] If $a\in \A_{\leqslant i}$ satisfies $\tau(a)=-a$, then $a\M_{\leqslant j}\subset \M_{i+j-d}$.
\end{itemize}

Let us consider two examples supporting this definition.

First, let $\A=\U$ and choose $\tau$ as in \ref{SSS_HC_mod_setting}. In this example, we take the PBW filtration and so $d=1$.
We claim that a HC $(\A,\tau)$-module is the same as a HC $(\g,\k)$-module, i.e., a finitely generated module
with locally finite action of $\k$. Let $\U^{\tau}$, resp. $\U^{-\tau}$, denote the subspace of the elements $u\in \U$
such that $\tau(u)=u$, resp. $\tau(u)=-u$. Then it is easy to see that $\U_{\leqslant i}^{-\tau}\subset\U_{\leqslant i-1}\k\oplus
\U_{\leqslant i-1}$. It follows that a HC $(\g,\k)$-module satisfies (ii). Now (i) becomes a well-known claim that
one has a good $\k$-stable filtration on any HC $(\g,\k)$-module. The claim that a HC $(\A,\tau)$-module is
HC as a $(\g,\k)$-module is easy.

Let us proceed to the second special case. Let $\mathcal{B}$ be a filtered algebra with an almost-commutativity condition
analogous to the condition on $\A$ above.  Set $\A:=\mathcal{B}\otimes \mathcal{B}^{opp}$ and let $\tau$ be
defined by $\tau(b_1\otimes b_2):=b_2\otimes b_1$. Of course, an $\A$-module is just a $\mathcal{B}$-bimodule. We claim
that a HC $(\A,\tau)$-module is the same thing as a HC $\mathcal{B}$-bimodule in the sense of \cite[2.5]{HC}, i.e., a
filtered bimodule $\M$ with $[\mathcal{B}_{\leqslant i}, \M_{\leqslant j}]\subset \M_{i+j-d}$ such that $\gr\M$ is a
finitely generated $\gr\mathcal{B}$-module.

First, since $b\otimes 1-1\otimes b$ lies in $\A^{-\tau}$, we see that a HC $(\A,\tau)$-module $\M$ is a HC $\mathcal{B}$-bimodule
(condition (ii) just says that $\gr\M$ is finitely generated as a $\gr\mathcal{B}$-bimodule but thanks to (i)
the left and right actions of $\gr\mathcal{B}$ on $\gr\M$ coincide).  On the other hand, the space $\A^{-\tau}$ is the linear span
of the elements of the form $b_1\otimes b_2-b_2\otimes b_1$. Using this it is easy to see that a HC $\mathcal{B}$-bimodule
is a HC $(\A,\tau)$-module.

\subsubsection{Functor $\bullet_{\dagger,e}$}
As we have mentioned, the construction of $\bullet_{\dagger,e}$ is easier than in the general case:
we do not need to fix $\iota: V\rightarrow \tilde{I}_\chi$ as in \ref{SSS_iota_choice}: any $R\times\K^\times\times\ZZ/2\ZZ$
and also $\tau$-equivariant $\iota$ works. The  functor $\bullet_{\dagger,e}$ is constructed by complete
analogy with \cite[3.3,3.4]{HC}, a crucial thing to notice is that $(\W_\hbar^{\wedge_0})^{-\tau}\M^{\wedge_\chi}_\hbar
\subset \hbar^2\M^{\wedge_\chi}_\hbar$ because $(\U_\hbar^{\wedge_\chi})^{-\tau}\M^{\wedge_\chi}_\hbar
\subset \hbar^2\M^{\wedge_\chi}_\hbar$. It follows that $[\mathfrak{s},f]\M^{\wedge_\chi}_\hbar\subset \hbar^2\M^{\wedge_\chi}$.

Also this description shows that, for a HC $(\g,\mathfrak{k})$-module $\M$, the image of $\M$ under $\bullet_{\dagger,e}$
is an $R$-equivariant HC $(\Walg,\tau)$-module.

\subsection{Another functor $\bullet_\dagger$ for $\OCat^P$}\label{SS_parab_top_quot}
\subsubsection{$\bullet_{\dagger,e}$ for a general parabolic category $\mathcal{O}$}
Now we again change our setting. Let $P$ be an arbitrary parabolic in $G$. Let $e$
be a Richardson element of $\p^\perp$  meaning that $Pe$ is dense in $\p^\perp$. Then we can choose $h,f$ forming an $\sl_2$-triple
in such a way that $h\in \p$, see \cite[Lemma 6.1.3]{Miura}. Let $V=[\g,f]$. Since $\z_\g(e)\subset \p$, see \cite[Theorem 1.3]{LS}, it is easy to see that $V\cap\p$ is  a lagrangian subspace in $V$. So we can construct 
a functor $\bullet_{\dagger,e}$ from the parabolic category $\OCat^P_\nu$
to the category $\Walg$-$\operatorname{mod}^R_\nu$, where, recall, $R$ is a maximal reductive subgroup
of $Z_P(e)$. The image of this functor consists of finite dimensional modules. We also would
like to point out that $Q^\circ\subset P$ and so  $Q^\circ\subset R$.

We remark that this functor $\bullet_{\dagger,e}$ is a generalization of
$\Vfun$ in the case when $\g=\g_0$.
It is not difficult  to show that all results mentioned in Theorem \ref{Thm:Vfun}
with a possible exception of (v)  still hold for
$\bullet_{\dagger,e}$.

\subsubsection{The orbit codimension condition in type A}
We would like to analyze the condition $\operatorname{codim}_{\p^\perp}\p^\perp\setminus Pe>1$ in the case
when $\g=\sl_n$. In this case a version of  $\bullet_{\dagger,e}$ was previously considered by Brundan
and Kleshchev in \cite{BK2},\cite{BK3}. A special feature of type $A$ is that any nilpotent element
admits a so called {\it good even grading}. Then one can replace $\g(i)$ with the $i$th graded component
with respect to this grading. All constructions of Sections \ref{S_Cats},\ref{S_Vfun} work for that modification.
Brundan and Kleshchev studied the Whittaker coinvariant functor, which is just $\Vfun_2$ in this special case.
So our approach recovers many results from \cite{BK2},\cite{BK3}.

Parabolic subalgebras in $\sl_n$ are parameterized by compositions of $n$, i.e., ordered collections of
positive integers summing to $n$. So let $\p$ correspond to a composition $(s_1,\ldots,s_{\ell})$.
Brundan and Kleshchev checked that the Whittaker coinvariants functor is $0$-faithful provided
$s_1>s_2>\ldots>s_\ell$. The following proposition together with the previous subsection generalizes that.

\begin{Prop}\label{Prop:codim_cond}
Let $e$ be a Richardson element for $P$. The condition
$\operatorname{codim}_{\p^\perp}\p^\perp\setminus Pe>1$ holds provided all $s$'s are distinct.
\end{Prop}
\begin{proof}
Suppose that all $s_i$'s are distinct. We are going to prove that for any nilpotent orbit $\Orb'\subset \overline{\Orb}$
(this condition is necessary for $\Orb'$ to intersect $\p^\perp$), we have $\dim \Orb'<\dim \Orb-2$.
Since $\dim \Orb'\cap \p^\perp\leqslant \dim \Orb'\cap \b^\perp= \frac{1}{2}\dim \Orb'$,
see, for instance, \cite[Theorem 3.3.7]{CG} for the last equality, the inequality $\dim \Orb'<\dim \Orb-2$ implies $\dim \Orb'\cap\p^\perp<\dim \p^\perp-1$. The number of nilpotent orbits is finite and so the previous inequality implies the required codimension condition.

Let $\lambda$ be the Young diagram with rows of lengths $\lambda_1:=s_1,\ldots,\lambda_\ell:=s_\ell$ ordered in the decreasing order.
Then $\Orb$ is the orbit corresponding to the transposed diagram $\lambda^t$. We have $\dim \Orb=n^2-\sum_{i} \lambda_i^2$.
Now let $\mu$ be a Young diagram such that $\mu^t$ corresponds to $\Orb'$. According to \cite[6.2]{CM}, the
condition $\Orb'\subset \overline{\Orb}$
is equivalent to $\lambda\leqslant \mu$ in the dominance ordering, i.e., $\sum_{i=1}^k\lambda_i\leqslant \sum_{i=1}^k \mu_i$
for all $k$. In other words, there is a sequence $\lambda^0:=\lambda,\lambda^1,\ldots,\lambda^m:=\mu$
such that $\lambda^i$ is obtained from $\lambda^{i-1}$ by moving a box from a shorter row to a longer row.
It is therefore enough to consider the case when $\mu=\lambda^1$. We need to show that $\sum_i \mu_i^2\geqslant \sum_i \lambda_i^2+4$.
There are indexes $p<q$ such that $\mu_i=\lambda_i$ if $i\neq p,q$, $\mu_p=\lambda_p+1, \mu_q=\lambda_q-1$.
We have
$$
\sum_{i}(\mu_i^2-\lambda_i^2)=(\lambda_p+1)^2-\lambda_p^2+ (\lambda_q-1)^2-\lambda_q^2=2(\lambda_p-\lambda_q)+2.
$$
Since $\lambda_p>\lambda_q$ by our assumption, we see that $2(\lambda_p-\lambda_q)+2\geqslant 4$, and we are done.
\end{proof}

The 0-faithfulness of the Brundan-Kleshchev functor was a crucial ingredient in the proof of an equivalence
between (the sum of certain blocks) of $\OCat^P$ and a category $\mathcal{O}$ for a suitable cyclotomic Rational
Cherednik algebra in \cite[Theorem 6.9.1]{GL} (under the restriction that $s_1>\ldots>s_\ell$). Using
Proposition \ref{Prop:codim_cond} one should be able to remove that restriction but we are not going to elaborate
on this.

\subsection{Dixmier algebras}\label{SS_Dixmier}
\subsubsection{Functor $\bullet_{\dagger,e}$}
By a Dixmier algebra one means a $G$-algebra
$\A$ equipped with a $G$-equivariant homomorphism $\U\rightarrow \A$ that makes $\A$ into a Harish-Chandra bimodule
such that the differential of the $G$-action on $\A$ coincides with the adjoint $\g$-action. The algebra $\U$ itself
as well as any quotient of $\U$ serve as  examples of  Dixmier algebras. Two more examples come from
Lie superalgebras and from quantum groups at roots of $1$, they are considered below. It is basically
those two families of examples that motivate us to consider arbitrary Dixmier algebras.

Let $K$ be as in \ref{SSS_Cat_OK_def}. Let $\OCat^K_\nu(\Dix)$ denote the category of $\Dix$-modules
that lie in $\OCat^K_\nu$ as $\U$-modules.

Fix a good algebra filtration $\Fi_\bullet \Dix$ on $\Dix$, where, recall, {\it good} means that all $\Fi_i\Dix$
are $G$-stable and $\gr \Dix$ is a finitely generated $S(\g)=\gr\U$-module. The existence of a good {\it algebra}
filtration is checked, for example, in \cite[Proof of Proposition 3.4.5]{HC}. In addition, we can assume that
$\Fi_0\Dix=\Fi_1\Dix=\K$ and that $\F_2\Dix$ contains the image of $\g$. This insures that we have
a filtered algebra homomorphism $U(\g)\rightarrow \Dix $. Consider the corresponding Rees algebra
$\Dix_\hbar:=\bigoplus_{i=0}^\infty (\Fi_i\Dix) \hbar^i$. This is a graded algebra, where $\hbar$
has degree $1$. The corresponding $\Walg$-bimodule $\Dix_\dagger$ has a natural algebra structure, see
\cite[3.4]{HC}.

Pick $\M\in \OCat_\nu^K(\Dix)$.
The $R$-equivariant $\Walg$-module $\M_{\dagger,e}$ has a natural $\Dix_\dagger$-module structure.
So we get an exact functor $\bullet_{\dagger,e}: \OCat^K_\nu(\Dix)\rightarrow \Dix_{\dagger}$-$\operatorname{mod}_\nu^R$.
We have the forgetful functors $\Fun_\U:\OCat^K_\nu(\Dix)\rightarrow \OCat^K_\nu$, $\Fun_{\Walg}:\Dix_\dagger$-$\operatorname{mod}^R\rightarrow \Walg$-$\operatorname{mod}^R$ and they
intertwine the functors $\bullet_{\dagger,e}$, i.e., $\Fun_\Walg(\bullet_{\dagger,e})=\Fun_\U(\bullet)_{\dagger,e}$.

Let us make an observation that will be used later. There is a functor $\bullet^{\dagger,e}:\Dix_{\dagger}$-$\operatorname{mod}^R\rightarrow
\tilde{\OCat}^K_\nu(\Dix)$, where $\tilde{\OCat}^K_\nu(\Dix)$ is a category of not necessarily finitely generated
$\Dix$-modules that are inductive limits of objects in $\OCat^K_\nu(\Dix)$. This functor satisfies $\Hom_{\Dix}(\M,\Nil^{\dagger,e})=
\Hom_{\Dix_\dagger,R}(\M_{\dagger,e},\Nil)$. The functor is constructed
completely analogously to $\bullet^{\dagger,e}$ in \ref{SSS_dag_prop} and, in particular, does not depend on
$\Dix$ in the sense that the functors $\bullet^{\dagger,e}$ for $\Dix$ and for $\U$ are again intertwined
by the forgetful functors.

\subsubsection{Enveloping algebras of Lie superalgebras}
Here is an interesting example of a Dixmier algebra $\Dix$. Let $\tilde{\g}$ be a simple classical Lie  superalgebra or a queer Lie superalgebra. In particular, the even part, denote it by $\g$, is a reductive Lie algebra. So the universal enveloping (super)algebra
$\Dix:=U(\tilde{\g})$ is a Dixmier algebra over $\U=U(\g)$. We also remark that $\Dix$ is a free left (or right)
module over $\U$.

The algebra $\Dix_\dagger$ may be thought as a W-algebra for the Lie superalgebra $\tilde{\g}$. It should not be difficult
to check that $\Dix_\dagger$ is isomorphic to the Clifford algebra over the super W-algebra studied previously,
see, for example, \cite{Zhao},\cite{BBG}. But we are not going to prove this here.

\subsubsection{Quantum groups at roots of 1}
Here we will show that the algebra that is basically the Lusztig form of a quantum group at a root of unity
is a Dixmier algebra. More precisely we will show that the quantum Frobenius epimorphism splits.

Let us recall some generalities on quantum groups. We follow \cite{Lusztig_qgrp_1}.

We fix a Cartan matrix $A=(a_{ij})_{i,j=1}^n$ of finite type. Let $\g$ be the corresponding
finite dimensional semisimple Lie algebra. Let $D=\operatorname{diag}(d_1,\ldots,d_n)$
be the matrix with coprime entries $d_i\in \{1,2,3\}$ such that $DA$ is symmetric. Let $v$
be an indeterminate. We write $[n]_{d_i}$ for the quantum integer $\displaystyle \frac{v^{d_in}-v^{-d_in}}{v^{d_i}-v^{-d_i}}$
and $[n]_{d_i}!$ for the corresponding quantum factorial. We then consider the quantum group $U_v(\g)$ that is
a $\K(v)$-algebra generated by $E_i,F_i, K_i^{\pm 1}$ in a standard way, see, e.g., \cite[Section 1]{Lusztig_qgrp_1}.
Inside we consider the $\K[v^{\pm 1}]$-subalgebra $\dot{U}_v(\g)$ generated by the divided powers $E_i^{(N)}:=\frac{E_i^N}{[N]_{d_i}!}$
and $F_i^{(N)}:=\frac{F_i^N}{[N]_{d_i}!}$ and also $K_i^{\pm 1}$. It contains elements $$\binom{K_i;c}{t}:=\prod_{i=1}^t \frac{K_i v^{d_i(c-s+1)}-K_i^{-1} v^{-d_i(c-s+1)}}{v^{d_i s}-v^{-d_is}},$$
see \cite[Section 6]{Lusztig_qgrp_1}. Inside $\dot{U}_v(\g)$ we consider subalgebras $\dot{U}_v^+$, (resp., $\dot{U}_v^-$ and $\dot{U}_v^0$)
generated by the elements $E_i^{(N)}$ (resp., $F_i^{(N)}$, and $K_i^{\pm 1}, \binom{K_i;c}{t}$).
We have the triangular decomposition $\dot{U}_v(\g)=\dot{U}^+_v\otimes \dot{U}^0_v\otimes \dot{U}^-_v$.

Pick an integer $\ell$ that is odd and, in the case when $\g$ has a component of type $G_2$,
coprime to $3$.  Let $\epsilon$ be an $\ell$th primitive root of $1$. Let $\dot{U}_\epsilon(\g), \dot{U}^+_\epsilon$ etc. denote the
specializations of the corresponding algebras at $v=\epsilon$.
The elements $K_i^{\ell}$ are central in $\dot{U}_\epsilon(\g)$ and, moreover, one can show that
$K_i^{2\ell}=1$ in $\dot{U}^0_\epsilon$. Let $\Dix$ denote the quotient of $\dot{U}_\epsilon(\g)$ by $K_i^{\ell}-1$
for $i=1,\ldots,n$. We still have the triangular decomposition $\Dix=\Dix^-\otimes \Dix^0\otimes \Dix^+$, and
$\Dix^-=\dot{U}_\epsilon^+(\g), \Dix^+=\dot{U}_\epsilon^-(\g)$. 
%In fact, it is easy to see that $\Dix^+$ is generated just by
%the elements $E_i,E_i^{(\ell)}$ and a similar claim holds for $\Dix^-$.

Also we have an epimorphism $\operatorname{Fr}^*:\Dix\rightarrow \U$ called the {\it quantum Frobenius epimorphism}, see
\cite[Section 8]{Lusztig_qgrp_1}. By construction, it maps $E_i^{(N)},F_i^{(N)}$ to $e_i^{(N/\ell)}, f_i^{(N/\ell)}$
if $N$ is divisible by $\ell$ and to $0$ else, and it sends $K_i$ to $1$. Further, it maps $\binom{K_i;0}{\ell}$
to $h_i$.

Our goal is to prove the following proposition.

\begin{Prop}\label{Prop:Fr_section}
There is a monomorphism $\iota:\U\rightarrow \Dix$ that is a section of $\operatorname{Fr}$, meaning that
$\operatorname{Fr}\circ \iota=\operatorname{id}$. Moreover, $\Dix$ becomes a Dixmier algebra
with respect to this homomorphism. Finally, $\Dix$ is free over $\iota(\U)$.
\end{Prop}

The proof of this proposition occupies the remainder of this subsection. First of all,
let ${\bf u}$ denote the small quantum group, the subalgebra of $\Dix$ generated by $K_i, E_i,F_i,
i=1,\ldots,n$. For $u\in {\bf u}$ set $\delta_i(u):=E_i^{(\ell)} u-u E_i^{(\ell)},
\delta'_i(u):=F_i^{(\ell)}u-u F_i^{(\ell)}$. It turns out, see \cite[Section 8]{Lusztig_qgrp_1},
that $\delta_i(u),\delta_i'(u)\in {\bf u}$ and so $\delta_i, \delta_i'$ are derivations
of ${\bf u}$.

Also  Lusztig in {\it loc.cit.}  proved that the map $e_i\mapsto E_i^{(\ell)}$ extends to an algebra
homomorphism $\U^+\rightarrow \A^+$ and a similar claim is true for $\U^-$ and $\A^-$.

Consider the Lie subalgebra $\mathfrak{n}$ in $\Dix$ consisting of all elements $x\in \Dix$
with $[x,{\bf u}]\subset {\bf u}$. Of course, ${\bf u}$ is an ideal in $\mathfrak{n}$,
and $E_i^{(\ell)}, F_i^{(\ell)}\in \mathfrak{n}$.

The following lemma is crucial in the proof of Proposition \ref{Prop:Fr_section}.

\begin{Lem}\label{Lem:Lie_alg_hom}
The assignment $e_i\mapsto E_i^{(\ell)}, f_i\mapsto F_i^{(\ell)}, h_i\mapsto \binom{K_i;0}{\ell}$ extends
to a Lie algebra homomorphism $\g\rightarrow \mathfrak{n}/{\bf u}$.
\end{Lem}

The proof is a computation and we omit it.

Since $\g$ is semisimple, any homomorphism $\g\rightarrow \mathfrak{n}/{\bf u}$ lifts to a Lie algebra homomorphism
$\g\rightarrow \mathfrak{n}$ and so to an algebra homomorphism $\iota: \U\rightarrow \Dix$. We have $\operatorname{Fr}\circ \iota(x)-x\in \K$
for $x=e_i,f_i, h_i$, which implies that the difference is actually $0$ (again, because $\g$ is semisimple). Also, since $\iota(e_i)-E_i^{(\ell)}, \iota(f_i)-F_i^{(\ell)},
\iota(h_i)-\binom{K_i;0}{\ell}\in {\bf u}$ and $E_i^{(\ell)},F_i^{(\ell)},\binom{K_i;0}{\ell}$ normalize ${\bf u}$, we see that
$\Dix$ is generated by ${\bf u}$ as a left (or right) $\U$-module, while the adjoint $\g$-action on $\Dix$
is locally finite. So $\Dix$ is a Dixmier algebra. The same argument shows that $\Dix$ is free over $\iota(\U)$.

\subsubsection{Category $\OCat^P(\Dix)$ and  functor $\Vfun$}
Here we consider a Dixmier algebra $\Dix$ that is free over $\U$. In both examples we have considered above
the algebra $\Dix$ has this property. Further, we take $P,e,h,f$ as in the main body of the paper
(but we also can take them as in Subsection \ref{SS_parab_top_quot}). Then we have the functor $\bullet_{\dagger,e}:\OCat^P_\nu(\Dix)
\rightarrow \OCat^\theta(\Dix,e)^R_\nu$, where the target category is the category of all $R$-equivariant
$\Dix_\dagger$-modules that lie in $\OCat^\theta(\g,e)^R_\nu$.

Recall that the forgetful functors $\mathcal{F}_{\U}:\OCat^P_\nu(\Dix)\rightarrow  \OCat^P_\nu$
and $\mathcal{F}_\Walg:\OCat^\theta(\Dix,e)^R_\nu\rightarrow \OCat^\theta(\g,e)^R_\nu$
intertwine both $\bullet_{\dagger,e}$ and its right adjoint functor $\bullet^{\dagger,e}$, i.e.
$$\mathcal{F}_\Walg(\bullet_{\dagger,e})=\mathcal{F}_{\U}(\bullet)_{\dagger,e}, \mathcal{F}_\U(\bullet^{\dagger,e})=\mathcal{F}_{\Walg}(\bullet)^{\dagger,e}.$$
In particular, we have
\begin{equation}\label{eq:adj_forget}
\mathcal{F}_\U\left((\bullet_{\dagger,e})^{\dagger,e}\right)=(\mathcal{F}_\U(\bullet)_{\dagger,e})^{\dagger,e}.
\end{equation}
It follows that $\bullet_{\dagger,e}$ is still a quotient functor onto its image annihilating precisely the modules
with non-maximal GK dimension. Also the image is closed under taking subquotients.

Now let us prove that $\bullet_{\dagger,e}$ has the double centralizer property.
The right adjoint functor to $\mathcal{F}_\U$, that is $\Hom_\U(\Dix,\bullet)$, is exact
because, by our assumption, $\Dix$ is a free $\U$-module. So if $P$ is a projective
object in $\OCat^P_\nu(\Dix)$, then $\mathcal{F}_\U(P)$ is  projective in $\OCat^P_\nu$. Thanks to (\ref{eq:adj_forget}),
the morphism $P\rightarrow (P_{\dagger,e})^{\dagger,e}$ does not depend on whether we consider
$P$ as a $\Dix$-module or as a $\U$-module. Thanks to the double centralizer property for $\U$, we
see that the natural homomorphism $P\rightarrow (P_{\dagger,e})^{\dagger,e}$ is an isomorphism,
and this implies the double centralizer property for $\Dix$.

\end{document}